\documentclass[11pt,reqno]{amsart}

\usepackage{import}
\usepackage{amsmath}
\usepackage{calc}
\PassOptionsToPackage{hyphens}{url}\usepackage[hidelinks,hypertexnames=false]{hyperref}
\usepackage{graphicx}

\usepackage[british]{babel}
\usepackage[noabbrev, nameinlink]{cleveref}
\usepackage{amsthm}
\usepackage{amssymb}
\usepackage{fullpage}
\usepackage{dsfont}
\usepackage{enumerate}
\usepackage{mathtools}
\usepackage{comment}
\usepackage{xcolor}
\usepackage{centernot}

\usepackage{thmtools}
\usepackage{thm-restate}

\renewcommand{\complement}{\mathsf{c}}
\newcommand{\EE}{\mathbb{E}}
\newcommand{\PP}{\mathbb{P}}
\newcommand{\NN}{\mathbb{N}}
\newcommand{\RR}{\mathbb{R}}
\newcommand{\ZZ}{\mathbb{Z}}
\newcommand{\Gnp}{\mathbb{G}}

\newcommand{\link}[2]{\underline{\partial}_{#2} #1}
\newcommand{\ilink}[2]{\overline{\partial}_{#2} #1}
\newcommand{\uppi}[1]{\overset{\gets}{#1}}
\newcommand{\normcomp}{\mathbin{\hat{\circ}}}
\newcommand{\cal}[1]{\mathcal{#1}}

\newcommand{\frakI}{\mathfrak{I}}

\renewcommand{\le}{\leqslant}
\renewcommand{\ge}{\geqslant}

\usepackage{stackengine}
\stackMath
\newcommand\xxrightarrow[2][]{\mathrel{%
  \setbox2=\hbox{\stackon{\scriptstyle#1}{\scriptstyle#2}}%
  \stackunder[0pt]{%
    \xrightarrow{\makebox[\dimexpr\wd2\relax]{$\scriptstyle#2$}}%
  }{%
   \scriptstyle#1\,%
  }%
}}

\crefname{thm}{theorem}{theorems}
\crefname{lem}{lemma}{lemmas}
\crefname{obs}{observation}{observations}
\crefname{defi}{definition}{definitions}
\Crefname{thm}{Theorem}{Theorems}
\Crefname{prop}{Proposition}{Propositions}
\Crefname{obs}{Observation}{Observations}
\Crefname{claim}{Claim}{Claims}
\Crefname{defi}{Definition}{Definitions}
\crefname{subsection}{subsection}{subsections}

\crefformat{enumi}{#2item~(#1)#3}
\Crefformat{enumi}{#2Item~(#1)#3}
\crefmultiformat{enumi}{items~#2(#1)#3}%
  { and~#2(#1)#3}{, #2(#1)#3}{ and~#2(#1)#3}
\crefrangeformat{enumi}{items~#3(#1)#4--#5(#2)#6}

\newenvironment{claimproof}
{\proof}
{\endproof}

\title{An exponential upper bound for induced Ramsey numbers}

\begin{otherlanguage}{brazilian}
  \author{Lucas Arag\~ao \and Marcelo Campos \and Gabriel Dahia \and \\ Rafael Filipe \and Jo\~ao Pedro Marciano}

\address{UERJ, R. São Francisco Xavier, 524 - Maracan\~a, Rio de Janeiro, Brasil}
\email{lucas.aragao@uerj.br}

\address{IMPA, Estrada Dona Castorina 110, Jardim Bot\^anico, Rio de Janeiro, 22460-320, Brasil}
\email{\{marcelo.campos, gabriel.dahia, rafael.santos, joao.marciano\}@impa.br}

  \thanks{
    During this work, Marcelo Campos was supported by Serrapilheira (grant R-2412-51283), Rafael Filipe was supported by CNPq, and Jo\~{a}o Pedro Marciano was supported by a FAPERJ Bolsa Nota 10.
  }
\end{otherlanguage}

\usepackage[numbers,longnamesfirst]{natbib}

\newtheorem{thm}{Theorem}[section]
\newtheorem{lem}[thm]{Lemma}

\newtheorem{obs}[thm]{Observation}
\newtheorem{claim}[thm]{Claim}

\theoremstyle{definition}
\newtheorem{defi}[thm]{Definition}
\newtheorem{remark}[thm]{Remark}

\newtheoremstyle{named}{}{}{\itshape}{}{\bfseries}{.}{.5em}{\thmnote{#3}}
\theoremstyle{named}

\begin{document}

\begin{abstract}
  The induced Ramsey number $R_{\mathrm{ind}}(H; r)$ of a graph $H$ is the minimum number $N$ such that there exists a graph with $N$ vertices for which all $r$-colourings of its edges contain a monochromatic induced copy of $H$.
  Our main result is the existence of a constant $C > 0$ such that, for every graph $H$ on $k$ vertices, these numbers satisfy
  \begin{equation*}
    R_{\mathrm{ind}}(H; r) \le r^{C r k}.
  \end{equation*}
  When $r = 2$, this resolves a conjecture of Erd\H{o}s from 1975.
  For $r > 2$, it answers a question of Conlon, Fox and Sudakov in a strong form.
\end{abstract}

\maketitle

\section{Introduction}

The Ramsey number of a graph $H$, denoted by $R(H)$, is the minimum number $N$ for which every red/blue colouring of the edges of $K_N$, the complete graph on $N$ vertices, contains a monochromatic copy of $H$.
\citet{Ram29} showed that $R(H)$ is finite for every graph $H$, with \citet{ES35} providing the first explicit upper bound in~\citeyear{ES35}.
In an influential paper from~\citeyear{Erd47}, \citet{Erd47} proved an exponential lower bound for the case $H = K_k$, which established
\begin{equation}\label{eq:erdosBoundsForRamsey}
  2^{k / 2} \le R(K_k) \le 4^k.
\end{equation}
Since then, the lower bound of \citeauthor{Erd47} has only been improved by a constant factor~\cite{Spe75}.
On the other hand, the upper bound in \eqref{eq:erdosBoundsForRamsey} has seen several improvements over the years~\cite{Tho88,Con09,Sah23}, culminating in an exponential improvement by \citet{CGMS25}.
A second (much shorter) proof of this result was given in~\cite{BBCGHMST24+}, which also extended it to the setting of $r$-colourings, though only for fixed $r$ and sufficiently large $k$.

Another major open problem in the area is the natural extension of $R(H)$ to the setting of induced subgraphs.
We will write $G \xrightarrow{\mathrm{ind}} H$ to denote the following property: for any red/blue colouring of the edges of $G$, there exists an induced monochromatic copy of $H$ (that is, a copy of $H$ which is induced in $G$, and all of its edges have the same colour).
We then define
$$R_{\mathrm{ind}}(H) = \min \big\{ v(G): G \xrightarrow{\mathrm{ind}} H \big\}.$$
In particular, observe that we have $R_\mathrm{ind}(K_k) = R(K_k)$ for every $k \in \mathbb{N}$, since every copy of $K_k$ in a graph $G$ is also an induced subgraph of $G$.
For general graphs $H$, on the other hand, \citet{Erd84} remarked that even ``the existence of [the induced Ramsey number] is not at all obvious.''

\citet{Deu75}, \citet{EHP75} and \citet{Rod73} independently established in the 1970s that $R_\mathrm{ind}(H)$ is finite for every graph $H$.
While none of these works provide an explicit dependency on $k$, the number of vertices of $H$, \citet{Erd84} later remarked that the best bound that one can deduce from the proofs in~\cite{Deu75,EHP75,Rod73} is of the form
\[R_\mathrm{ind}(H) \le 2^{2^{k^{\scalebox{0.45}{$1 + o(1)$}}}}.\]
Nevertheless, \citeauthor{Erd84}~\cite{Erd75, Erd84} conjectured, first implicitly in \citeyear{Erd75} and then explicitly in \citeyear{Erd84}, that the function $R_\mathrm{ind}(H)$ should grow at most exponentially as a function of $v(H) = k$ for every $H$.
Note that, if true, this would be best possible, since we have $R_\mathrm{ind}(K_k) = R(K_k) \ge 2^{k / 2}$ by~\eqref{eq:erdosBoundsForRamsey}.

Using the techniques of \citet{Rod73}, one can prove \citeauthor{Erd84}' conjecture for bipartite $H$.
The problem is much harder for non-bipartite $H$, however, and the next significant advance was not obtained until almost 25 years later, by~\citet{KPR98}.
By taking $G$ (the host graph) to be a random graph built using projective planes, they showed, among other results, that
\begin{equation}\label{eq:KPR98}
  R_{\mathrm{ind}}(H) \le k^{O(k \log k)}
\end{equation}
for every graph $H$ with $k$ vertices.
\citet{FS08} later provided an explicit, pseudorandom graph attaining the bound in~\eqref{eq:KPR98}.
To achieve this, they pioneered a versatile approach to Ramsey-type theorems in the setting where one fixed $H$ is forbidden as an induced subgraph, which became influential for other notorious conjectures, like the Erd\H{o}s--Hajnal problem (see e.g.\ \cite{BNSS24, NSS23+, NSS23+b}).

A few years later, \citet{CFS12} removed a factor of $\log k$ from the exponent in \eqref{eq:KPR98} and showed, using an explicit graph, that
\begin{equation}\label{eq:CFS12}
  R_{\mathrm{ind}}(H) \le k^{O(k)}.
\end{equation}
In order to prove \eqref{eq:CFS12}, the authors of~\cite{CFS12} developed a general method for proving Ramsey-type theorems using pseudorandom properties of the host graph $G$.
For example, their method also allowed them to improve a celebrated result of \citet{GRR01} on the Ramsey numbers of bounded degree graphs.

The main result of this paper confirms \citeauthor{Erd84}' conjecture for all graphs $H$.

\begin{thm}\label{stmt:main}
  There exists a constant $C > 0$ such that
  \begin{equation*}
    R_{\mathrm{ind}}(H) \le 2^{Ck}
  \end{equation*}
  for every graph $H$ with $k$ vertices.
\end{thm}

We will also consider the induced Ramsey number for $r$-colourings.
Define $R_\mathrm{ind}(H; r)$ to be the \emph{$r$-colour induced Ramsey number} of a graph $H$; that is, the minimum number of vertices of a graph $G$ such that every $r$-colouring $c \colon E(G) \to [r]$ of the edges of $G$ contains an induced monochromatic copy of $H$.
The techniques used in \cite{CFS12} and \cite{KPR98} do not work in this more general setting, and provide no bounds when $r \ge 3$, but \citet{FS09} introduced a different approach in 2009, which can be used to show that\footnote{The authors of~\cite{FS09} only state the weaker bound $R_{\mathrm{ind}}(H; r) \le r^{O(rk^3)}$, because their focus in that paper was on fixed $k$ and large $r$, but their proof can easily be modified to give the  bound~\eqref{eq:bestBoundForRColours}.}
\begin{equation}\label{eq:bestBoundForRColours}
  R_{\mathrm{ind}}(H; r) \le r^{O(r k^2)}.
\end{equation}
The large gap between \eqref{eq:bestBoundForRColours} and the known bounds for the $r = 2$ case motivated \citet*[Problem 3.5]{CFS15} to ask if one could show that, for fixed $r \in \NN$,
\[
  R_{\mathrm{ind}}(H; r) \le 2^{k^{1+o(1)}}.
\]
We resolve this problem in a very strong form, obtaining a bound that generalises \Cref{stmt:main} for every $r \ge 2$.
In fact, our proof works directly in this more general setting.

\begin{thm}\label{stmt:multicolour}
  There exists a constant $C > 0$ such that
  \begin{equation}\label{eq:multicolour:thm}
    R_{\mathrm{ind}}(H; r) \le r^{C r k}
  \end{equation}
  for every $r \ge 2$ and every graph $H$ with $k$ vertices.
\end{thm}

We remark that, up to the value of the constant $C$, the bound~\eqref{eq:multicolour:thm} matches the classical upper bound of \citet{ES35} on the $r$-colour Ramsey numbers $R(K_k;r) = R_{\mathrm{ind}}(K_k; r)$, which was only recently improved by a small exponential factor in~\cite{BBCGHMST24+}.
The best known lower bound is of the form $R(K_k;r) \ge 2^{\Omega(rk)}$, see~\cite{Abb72,CF21,Wig21,Saw22}.
Our method will moreover imply the stronger statement that there exists a graph $G$ with $N = r^{C r k}$ vertices such that every $r$-colouring of the edges of $G$ contains an induced monochromatic copy of \emph{every} graph $H$ on $k$ vertices.
In fact, we will show that almost every graph $G$ with $N$ vertices has this property.

\subsection{An overview of our approach}\label{sec:ideasInTheProof}

Unlike the earlier approaches in~\cite{KPR98, CFS12, FS08}, where the authors developed ingenious deterministic algorithms to embed $H$ in a pseudorandom host graph $G$, we will instead use a relatively simple vertex-by-vertex embedding strategy in a truly random graph $G$.
Specifically, we consider the Erd\H{o}s--R\'{e}nyi random graph $\Gnp(N, 1/2)$, where each edge of $K_N$ is included independently at random with probability $1/2$; equivalently, we can choose a (labelled) graph $G$ on $N$ vertices uniformly at random.

Our strategy will crucially exploit the fact that the edges of $\Gnp(N, 1/2)$ are chosen randomly, rather than relying on pseudorandom properties that are also satisfied by $G \sim \Gnp(N, 1/2)$.
We will show that, with (extremely) high probability, every $r$-colouring of the edges of $G \sim \Gnp(N, 1/2)$ contains an induced monochromatic copy of an arbitrary $k$-vertex graph $H$.

One way to approach the problem is via an Erd\H{o}s--Szekeres-type induction.
In other words, we might generalise to the setting in which we want to find an induced copy of $H_i$ in colour $i$, apply the induction hypothesis inside a subset $U \subset V(G)$ of size $\delta N$ to find an induced copy of $H_i^-$ (that is, $H_i$ minus a vertex) in colour $i$ for some $i \in [r]$, and then attempt to use the randomness between $U$ and the rest of the vertices to extend this copy of $H_i^-$ to a copy of $H_i$.

There are several major obstacles to utilizing such a strategy.
First, and most obviously, there may be no edges of colour $i$ between $U$ and $V(G) \setminus U$, in which case we have no chance of extending $H_i^-$ to a copy of $H_i$ in colour $i$.
We can easily deal with this issue, however, by instead using the induction hypothesis to find an induced copy of $H_i^- \subset G[U]$ in colour $i$ for \emph{every} $i \in [r]$, and then considering the colour that is used most often between $U$ and $V(G) \setminus U$.

A second (and more challenging) issue is that the colouring of the edges inside the set $U$ is allowed to depend on the (random) edges between $U$ and $V (G) \setminus U$.
We will deal with this problem by taking a union bound over all possible colourings of the edges of $G[U]$.
However, this does not come for free: it requires us to prove an extremely strong bound on the probability of failure in each step of the induction.

In order to prove such a strong bound, we must strengthen our induction hypothesis.
Indeed, if we only have one copy of $H_i^-$, the probability that no vertex of $V(G) \setminus U$ extends this fixed copy of $H_i^-$  to an induced copy of $H_i$ in $G$ is essentially $(1 - 2^{-k})^N$, which is much too large to beat the roughly $r^{\delta^2N^2}$ choices in our union bound.
In order to reduce this failure probability, we will need to instead find many copies of $H_i^-$ in $G[U]$ that are moreover ``well-distributed'' in a certain precise sense, which we define in \Cref{sec:reductionToKey} (see \Cref{def:janson}).
We say that a hypergraph is \emph{$(p, R)$-Janson} if its edges are well-distributed in this sense, since this definition resembles (and was inspired by) the condition in Janson's inequality.

There is still, however, one further (and even more critical) obstruction to this approach: we must also handle all colourings of the edges between $U$ and $V(G) \setminus U$, and in this case we cannot do so simply by taking a union bound, since there are too many possible colourings.
In order to deal with this more serious obstacle, we will use the method of hypergraph containers.

\subsection{Hypergraph containers}

The method of hypergraph containers, which was introduced in 2015 by \citet{BMS15} and \citet{ST15}, is a flexible technique for controlling the probability that a random set avoids some forbidden substructure (see the surveys~\cite{BMS18, Sam15}).
Roughly speaking, the basic container lemma implies that the sets that avoid these substructures are ``clustered'', in the sense that they can be covered by a relatively small number of sets that contain only few copies of the forbidden substructures.

The container method has been used by several different sets of authors to prove Ramsey theoretic properties of random graphs.
For example, it was used by \citet{NS16} and \citet{MNS20} to find monochromatic copies of fixed subgraphs in $r$-colourings of sparse random graphs, and by \citet*{CDLfRS17} and \citet{BS20} to bound the induced Ramsey numbers of graphs and hypergraphs.

In particular, the authors of~\cite{CDLfRS17} significantly improved the best-known upper bound for the induced Ramsey number of an arbitrary $s$-uniform hypergraph with $k$ vertices when $s \ge 3$.
Their result was then improved when $s = 2$ (that is, for graphs) in~\cite{BS20}, where the authors introduced a new container lemma with a dramatically improved dependency on the size of the forbidden structure, and applied it to give a bound of the form
\begin{equation}\label{eq:wojtekReprovedMulticolour}
R_{\mathrm{ind}}(H; r) \le r^{O(r k^2)}.
\end{equation}
However, further improvements to the dependency of the container lemma would not advance the bound beyond \eqref{eq:wojtekReprovedMulticolour}, and we must therefore take a different approach.

Our first significant departure from these earlier applications of the container method is in how we apply this method.
In~\cite{BS20,CDLfRS17,MNS20,NS16}, the authors embed the whole of $H$ in a single step, encoding the edge sets of induced copies of $H$ using a hypergraph whose vertex set consists of $r$ copies of $E(K_n)$.
In contrast, we will have not just one, but roughly $r^{|U|^2}$ different hypergraphs, one for each of the possible colouring of $G[U]$.
Our hypergraphs will encode the vertex sets (not the edge sets) of monochromatic induced copies of $H_i^-$ in $G[U]$ and our aim will be to find a $(p,R)$-Janson collection of (also monochromatic and induced) copies of $H_i$ in an arbitrary $r$-colouring of the edges between $U$ and $V(G) \setminus U$.

Although our final goal is to find a $(p,R)$-Janson collection of copies of $H_i$, applying the method of containers to find a single copy is already instructive.
This will require introducing a correspondence between independent sets in our hypergraph and pairs of neighbourhoods $\mathrm{N}_G(v)$ and $\mathrm{N}_{G_i}(v)$ (where $G_i$ is the subgraph of $i$-coloured edges) that do not extend any copy of $H_i^-$.
In this simplified setting, applying a standard hypergraph container lemma is sufficient to conclude that the probability that a vertex does not extend any copies of $H_i^-$ is exponentially small.

Changing from extending a single copy to finding a $(p, R)$-Janson collection of copies of $H_i$ introduces significant new difficulties, which further set our argument apart from those in~\cite{BS20,CDLfRS17,MNS20, NS16}.
The general strategy will be to have a second induction on the Janson parameter $R$: we will show that adding a vertex to $U$ increases $R$ with very high probability, which results in a ``richer'' family of copies of $H_i$.
In this way, after adding $\delta N$ vertices, $R$ will be as large as we need.

However, even the first step of this incrementing argument requires containers for sets where the hypergraph is not $(p,R)$-Janson, rather than the classical container lemma from~\cite{BMS15,ST15} or the ``efficient'' container lemma proved in~\cite{BS20}, see \Cref{stmt:containersForNonJansonGeneral}.
For this purpose, our main tool will be a strengthening of the container method, proved recently by~\citet{CS24+}.
We will combine this result with an ``efficient'' container lemma in a novel and surprising way, resulting in a flexible method to prove new, more powerful container theorems.

\subsection{Containers for sets that are not well-distributed}

One of the main challenges of carrying out the above outline is that our induction hypothesis is a \emph{global property} (being $(p, R)$-Janson) of the hypergraph that encodes monochromatic induced copies of $H_i$ in colour $i$, whereas previous container lemmas are only able to handle \emph{local properties} of the forbidden substructures.

Our most important technical contribution is a novel method that allows one to prove container \namecrefs{stmt:containersForNonJansonGeneral} that can handle such global properties that depend on local behaviour.
The first step of this method is to represent the global property as edges, of potentially linear size, in a hypergraph.
We then decompose the independent sets of this hypergraph into independent sets of a small collection of \emph{container hypergraphs}.
Together, the container hypergraphs encode all of the local obstructions to this global property, which effectively separates the local obstructions from the global obstructions.
This separation leads to a more favourable setting: to deal with local obstructions, we have a vast array of tools at our disposal, including traditional container theorems.

We illustrate this method by proving a general container \namecref{stmt:containersForNonJansonGeneral} for sets that are not $(p, R)$-Janson, \Cref{stmt:containersForNonJansonGeneral}, which we moreover expect to have further applications.
The proof of this result starts with an application of the aforementioned result of \citet{CS24+} (see \Cref{stmt:containersHardcovers}), yielding a decomposition of the sets that are not $(p,R)$-Janson into a collection of independent sets in container hypergraphs.
Roughly speaking, \Cref{stmt:containersHardcovers} states that each container hypergraph that does not have a ``large local part'' has the following property: a random independent set $I$ with $q |V|$ vertices looks very similar to a binomial random subset $V_q$ of the vertices of the hypergraph.
More precisely, all sets $L$ that are not edges of the container hypergraph satisfy
\begin{equation}\label{eq:approxInd}
\PP\big( L \subset I \big) \approx \PP\big( L \subset V_q \big).
\end{equation}
From \eqref{eq:approxInd}, we will be able to deduce that the random independent set $I$ is $(p,R)$-Janson with positive probability, contradicting the fact that independent sets are not $(p,R)$-Janson.
It then follows that all container hypergraphs have a large local part, and we can apply another container theorem (\Cref{stmt:containersCoversButJanson}, also proved in \cite{CS24+}) to the local part of each container hypergraph to finish the proof.

The next \namecref{sec:reductionToKey} provides key definitions for the rest of the paper and an important reduction.
In particular, we define two events, one encoding the induction hypothesis and the other encoding the failure to contain a $(p, R)$-Janson collection of copies of $H_i$.
We then state our key probabilistic \namecref{stmt:key}, which says that the probability of these two events happening simultaneously is very small, and prove that \Cref{stmt:multicolour} follows from it by induction.
In the end of \Cref{sec:reductionToKey}, we discuss the structure of the remainder of the paper.

\section{Reduction to a key lemma}\label{sec:reductionToKey}

This \namecref{sec:reductionToKey} is devoted to the reduction of \Cref{stmt:multicolour} to a key probabilistic \namecref{stmt:key}.
To do that, we will first give two important definitions, \Cref{def:hypergraphEncodingCopies,def:janson}.
The first will allow us to reason about induced copies of a graph $H$ as edges in a hypergraph, while the other is the key notion of how ``well-distributed'' the edges of a hypergraph are.
This notion will underpin our container \namecref{stmt:containersForNonJansonGeneral}, so immediately afterwards, we note a few simple properties of this latter notion that will be useful later.

Using these two \namecrefs{def:hypergraphEncodingCopies}, we will further define an event to represent our inductive assumption, \Cref{def:mainEvent}, and another to stand for the failure of completing an induction step, \Cref{def:badEvent}.
We follow this with the statement of the key probabilistic \namecref{stmt:key}, \Cref{stmt:key}, which bounds the probability of these events happening simultaneously, and a proof that \Cref{stmt:multicolour} can be deduced from \Cref{stmt:key}.
This \namecref{sec:reductionToKey} concludes with a summary for the rest of the paper, where we provide the tools that we will use to prove the key \namecref{stmt:key} in \Cref{sec:proofOfKey}.

\subsection{Preliminaries}

The first definition that we need is that of the hypergraph which encodes copies of some graph $F \subset G'$ that is also induced in $G$.
We will use \Cref{def:hypergraphEncodingCopies} with $G'$ being the subgraph defined by the edges in some colour $i \in [r]$, and $G$ being the underlying (random) graph.
In what follows and in the rest of the paper, we will identify a hypergraph with its edge set.

\begin{defi}\label{def:hypergraphEncodingCopies}
  Given graphs $F, G'$ and $G$ such that $G' \subset G$, define $\frakI_{F, G', G}$ to be the $v(F)$-uniform hypergraph with vertex set $V(G)$, and
  \begin{equation*}
    \frakI_{F, G', G} = \big\{L \subset V(G) : \, F \cong G'[L] = G[L] \big\}.
  \end{equation*}
  That is, each hyperedge of $\frakI_{F, G', G}$ corresponds to a copy of $F \subset G'$ that is induced in $G$.
\end{defi}

As we will apply the container method to a hypergraph that is closely related to $\frakI_{F, G', G}$, we emphasise that it has $V(G)$ for its vertex set.
This is in contrast with using $E(G)$ for the vertices of the auxiliary hypergraph, the more common definition in applications of the container method.

We now formalise what it means for a hypergraph to be ``well-distributed.''
A crucial aspect of this \namecref{def:janson} is that we let ourselves assign weights to the edges of the hypergraph using measures.
We say that a measure $\nu$ is \emph{supported} on a hypergraph $\cal{G}$ if $\nu$ is non-zero only on the edges of $\cal{G}$.
A hypergraph $\cal{G}$ is $(p, R)$-Janson when there exists a measure $\nu$ with certain properties that is supported on $\cal{G}$.

We call this property $(p, R)$-Janson because, under some extra assumptions, one can apply Janson's inequality and conclude that a $p$-random subset of the hypergraph's vertices is an independent set with probability at most $\exp(-R)$.
Even though this is the original motivation for the \namecref{def:janson}, we will not require these extra assumptions or use Janson's inequality.

\begin{defi}\label{def:janson}
  Let $\cal{G}$ be a hypergraph and $p > 0$.
  For measures $\nu : \cal{G} \to \RR_{\ge 0}$, we define
  \begin{equation}\label{eq:defEdgesAndLambda}
    e(\nu) = \sum_{E \in \cal{G}} \nu(E) \qquad \text{and} \qquad \Lambda_p(\nu) = \sum_{\substack{L \subset V(\cal{G}) \\ |L| \ge 2}} d_\nu(L)^2 \, p^{-|L|},
  \end{equation}
  where
  \begin{equation*}
    d_\nu(L) = \sum_{L \subset E \in \cal{G}} \nu(E)
  \end{equation*}
  is the degree of the set $L$ in the measure $\nu$.
  For every $R > 0$, we say that $\cal{G}$ is $(p, R)$-Janson if there exists a measure $\nu : \cal{G} \to \RR_{\ge 0}$ such that
  \begin{equation}\label{eq:def:janson}
    \Lambda_p(\nu) < \frac{e(\nu)^2}{R}.
  \end{equation}
  If $R = 0$, then every hypergraph is $(p, R)$-Janson.
\end{defi}

We mention three very simple, but useful \namecrefs{stmt:normalisedJansonWitnesses} about this property, for future reference.
The first is so simple that we omit its proof.

\begin{obs}\label{stmt:jansonParametersMonotone}
  If $\cal{G}$ is a $(p, R)$-Janson hypergraph for $p > 0$ and $R \ge 0$, then for all $p' \ge p$ and $R' \le R$, it is also $(p', R')$-Janson.
\end{obs}

The next \namecref{stmt:normalisedJansonWitnesses} is that we can normalise any measure $\nu$ satisfying \eqref{eq:def:janson} to choose the value it takes in $e(\nu)$.

\begin{obs}\label{stmt:normalisedJansonWitnesses}
  Let $p, R > 0$, and let $\cal{G}$ be a hypergraph that is $(p, R)$-Janson.
  For every $y > 0$, there exists $\nu : \cal{G} \to \RR_{\ge 0}$ such that
  \begin{equation*}
    e(\nu) = y \qquad \text{and} \qquad \Lambda_p(\nu) < \frac{y^2}{R}.
  \end{equation*}
\end{obs}

\begin{proof}
  As $\cal{G}$ is $(p, R)$-Janson, there exists $\tilde{\nu} : \cal{G} \to \RR_{\ge 0}$ such that $\Lambda_p(\tilde{\nu}) < e(\tilde{\nu})^2/R$, which implies that $e(\tilde{\nu}) > 0$ by $\Lambda_p(\cdot) \ge 0$.
  One can now verify that taking $\nu = y \tilde{\nu} / e(\tilde{\nu})$ completes the proof.
\end{proof}

Finally, we observe being $(p,R)$-Janson is monotone with respect to taking subhypergraphs.

\begin{obs}\label{stmt:JansonIsDownset}
  Let $R \ge 0$ and $p>0$, and let $\cal{G}$ and $\cal{G}'$ be hypergraphs satisfying $\cal{G}' \subset \cal{G}$.
  If $\cal{G}'$ is $(p, R)$-Janson, then $\cal{G}$ is also $(p, R)$-Janson.
\end{obs}

\begin{proof}
  The observation follows immediately from the fact that any $\nu$ supported on $\cal{G}'$ is also supported on $\cal{G}$, by assigning measure zero to every $E \in \cal{G} \setminus \cal{G}'$.
\end{proof}

\subsection{The key lemma}

As outlined in \Cref{sec:ideasInTheProof}, our proof will be by induction on $k$, and we will need a very strong bound on probability of each inductive step failing, in order to allow for a union bound over colourings of $G[U]$ for some set $U$ of size $\delta N$.
We will now define the event representing our induction hypothesis, \Cref{def:mainEvent}, and the event that represents the failure of completing an induction step, \Cref{def:badEvent}.
Once these are defined, we will state the key lemma, \Cref{stmt:key}, and prove that our main \namecref{stmt:multicolour} is a straightforward consequence of it.

To avoid repetition, we will let, in this \namecref{sec:reductionToKey} and also throughout the paper\footnote{Except in our container theorems, where we will let $p$ vary in an interval that includes this value.}
\begin{equation}\label{eq:fixedParams}
  C = 300, \qquad \delta = \delta(r) = \frac{1}{r^{50}} \qquad \text{and} \qquad p = p(r, k) = \frac{1}{ 2^{25} k^{2} r^{4}}.
\end{equation}
Whenever the specific value of these constants is important, we will remind the reader of them.

As previously mentioned, we have a very strong induction hypothesis.
To state it, we introduce some notation.
Given $c : E(G) \to [r]$ and a colour $i \in [r]$, we will denote by $G_i^{(c)}$ the subgraph consisting of the edges of $G$ that are coloured $i$ by $c$.
We will omit the dependency of $G^{(c)}_i$ on $c$ when the colouring is evident from context, writing simply $G_i$.

We will assume that for every sufficiently large subset $W \subset V(G)$, and every collection of graphs $F_1, \ldots, F_r$ such that
\begin{equation*}
  v(F_1),\ldots, v(F_r) \le k \qquad \text{ and } \qquad \sum_{i = 1}^r v(F_i) = r k - 1,
\end{equation*}
the following holds: in any $r$-colouring of $G[W]$, there is a colour $i \in [r]$ for which the copies of $F_i \subset G_i[W]$ are well-distributed.
To be precise, and using the \namecrefs{def:janson} that we have just introduced, we will have that in this colour $i \in [r]$, the hypergraph $\frakI_{F_i, G_i, G}[W]$ is $(p, p|W|)$-Janson.

\begin{restatable}{defi}{mainEvent}\label{def:mainEvent}
  Given $r \in \NN$ and $s_1, \ldots, s_r \in \NN$, let $\mathbf{s} = (s_i)_{i \in [r]}$ and let $\cal{E}(\mathbf{s})$ be the family of graphs $G$ with the following property.
  For all graphs $F_1, \ldots, F_r$ satisfying
  \begin{equation*}
    \sum_{i = 1}^r v(F_i) = \sum_{i = 1}^r s_i - 1 \qquad \text{ and } \qquad v(F_i) \le s_i ~ \text{ for each } i \in [r],
  \end{equation*}
  for every $W \subset V(G)$ with
  \begin{equation*}
    |W| \ge \frac{\delta}{8r} \, v(G),
  \end{equation*}
  and every colouring $c : E(G[W]) \to [r]$, there is $i \in [r]$ such that $\frakI_{F_i, G_i, G}[W]$ is $(p, p|W|)$-Janson.
\end{restatable}

We now define the event that corresponds to the failure of the induction step, $\cal{B}(\mathbf{H})$.

\begin{restatable}{defi}{badEvent}\label{def:badEvent}
  Given a collection of graphs $H_1, \ldots, H_r$, let $\mathbf{H} = (H_i)_{i \in [r]}$ and let $\cal{B}(\mathbf{H})$ be the family of graphs $G$ with the following property.
  There exists a colouring $c : E(G) \to [r]$ such that, for every $i \in [r]$, the hypergraph $\frakI_{H_i, G_i, G}$ is not $(p, p \, v(G))$-Janson.
\end{restatable}

To simplify the notation, whenever all the $(H_i)_{i \in [r]}$ are equal to a single graph $H$, we denote this by $\cal{B}(H; r)$.
Generalising the notation used in the introduction, we write
\begin{equation*}
  G \xxrightarrow[r]{\mathrm{ind}} H
\end{equation*}
to denote that, for any $r$-colouring of the edges of $G$, there exists an (induced) monochromatic copy of $H$.
Observe that if $\PP\big(G \in \cal{B}(H; r) \big) < 1$, then there exists a graph $G$ such that $G \xxrightarrow[r]{\mathrm{ind}} H$.

We can now state the result from which we will deduce \Cref{stmt:multicolour}.

\begin{restatable}{lem}{keyStmt}\label{stmt:key}
  For all $k, r \in \NN$ with $r \ge 2$ and for all graphs $H_1, \ldots, H_r$ with at most $k$ vertices, if $N \in \NN$ satisfies
  \begin{equation}\label{eq:boundOnNInKey}
    N \ge r^{C(k + t)}
  \end{equation}
  for $t = \sum_{i = 1}^{r} v(H_i)$, then
  \begin{equation*}
    \PP\big(G \in \cal{B}(\mathbf{H}) \cap \cal{E}(\mathbf{s})\big) \le 2^{- \delta^2 N^2},
  \end{equation*}
  where  $G \sim \Gnp(N, 1/2)$, $\mathbf{H} = (H_i)_{i \in [r]}$ and $\mathbf{s} = \big(v(H_i)\big)_{i \in [r]}$.
\end{restatable}

To prove that $\PP\big(G \in \cal{B}(H; r)\big) < 1$, we will find, for each $G \in \cal{B}(H; r)$, a large set $W \subset V(G)$ and graphs $H_1, \ldots, H_r$ such that $G[W] \in \cal{B}(\mathbf{H}) \cap \cal{E}(\mathbf{s})$.
Then, we will apply \Cref{stmt:key} to each $G[W]$ and take a union bound over choices of $W$ and $H_1, \ldots, H_r$ to complete the proof.

\begin{proof}[Proof that \Cref{stmt:key} implies \Cref{stmt:multicolour}]
  Fix
  \begin{equation}\label{eq:proofOfMain:choiceOfN2}
    N = r^{5 C r k},
  \end{equation}
  a choice which makes the constant in \Cref{stmt:multicolour} equal to $5 C$.
  Our goal is to show that there exists $G$ on $N$ vertices such that $G \xxrightarrow[r]{\mathrm{ind}} H$.
  As previously observed, it will suffice to prove that
  \begin{equation}\label{eq:probabilisticGoalForMainResult}
    \PP\big(G \in \cal{B}(H; r)\big) < 1,
  \end{equation}
  for $G \sim \Gnp(N, 1/2)$.
  The next \namecref{stmt:unionBoundOverKeyEvent} will let us bound this probability with a union bound.

  \begin{claim}\label{stmt:unionBoundOverKeyEvent}
    For every $G \in \cal{B}(H; r)$, there are $W \subset V(G)$ and $H_1, \ldots, H_r \subset K_k$ such that
    \begin{equation}\label{eq:sizeOfWInSubevent}
      |W| \ge \left(\frac{\delta}{8r}\right)^{rk - t} N
    \end{equation}
    for $t = \sum_{i = 1}^r v(H_i)$, and
    \begin{equation}\label{eq:subeventCondition}
      G[W] \in \cal{E}(\mathbf{s}) \cap \cal{B}(\mathbf{H})
    \end{equation}
    where $\mathbf{H} = (H_i)_{i \in [r]}$ and $\mathbf{s} = \big( v(H_i) \big)_{i \in [r]}$.
  \end{claim}

  \begin{claimproof}
    Choosing $H_1, \ldots, H_r = H$ and $W = V(G)$ trivially satisfies \eqref{eq:sizeOfWInSubevent} and $G[W] \in \cal{B}(\mathbf{H})$ by $G \in \cal{B}(H; r)$, which is a choice of $\mathbf{H}$ and $W$ that satisfies the assumptions in the \namecref{stmt:unionBoundOverKeyEvent}.
    The existence of one such choice means that we can take $\mathbf{H} = (H_i)_{i \in [r]}$ and $W \subset V(G)$ to minimise $t = \sum_{i = 1}^r v(H_i)$ for $H_1, \ldots, H_r \subset K_k$ among the choices that satisfy $G[W] \in \cal{B}(\mathbf{H})$ and \eqref{eq:sizeOfWInSubevent}.
    We claim that this $\mathbf{H}$ and $W$ also satisfies $G[W] \in \cal{E}(\mathbf{s})$, and therefore \eqref{eq:subeventCondition}.

    Indeed, if $G[W] \not \in \cal{E}(\mathbf{s})$, then, by definition, there are graphs $H_1', \ldots, H_r'$ with
    \begin{equation}\label{eq:tDoublePrimeIsLeqT}
      \sum_{i = 1}^r v(H_i') = t - 1,
    \end{equation}
    a set $W' \subset W$ with
    \begin{equation*}
      |W'| \ge \frac{\delta}{8r} \, |W|,
    \end{equation*}
    and a colouring $c : G[W'] \to [r]$ such that $\frakI_{H_i', G_i, G}[W']$ is not $(p, p|W'|)$-Janson for all $i \in [r]$.
    But then $W'$ and $\mathbf{H}' = (H_i')_{i \in [r]}$ also satisfy $G[W'] \in \cal{B}(\mathbf{H}')$ and
    \begin{equation*}
      |W'| \ge \frac{\delta}{8r} \, |W| \ge \left(\frac{\delta}{8r}\right)^{rk - (t - 1)} N,
    \end{equation*}
    since $W$ satisfies \eqref{eq:sizeOfWInSubevent}.
    Therefore, $W'$ would also satisfy \eqref{eq:sizeOfWInSubevent}, and the existence of $W'$ and $\mathbf{H}'$ would, by \eqref{eq:tDoublePrimeIsLeqT}, contradict the minimality of $t$ in the original choice of $W$ and $\mathbf{H}$, so $G[W] \in \cal{E}(\mathbf{s})$.
  \end{claimproof}

  Applying \Cref{stmt:unionBoundOverKeyEvent}, we can take a union bound over the choices of $\mathbf{H}$ and $W \subset V(G)$ in that statement, obtaining as a result
  \begin{equation}\label{eq:firstprobabilityBound}
    \PP\big(G \in \cal{B}(H; r)\big) \le \sum_\mathbf{H} \sum_{\substack{W \subset V(G) \\ |W| \ge \delta^{2 r k}N}} \PP\big(G[W] \in \cal{B}(\mathbf{H}) \cap \cal{E}(\mathbf{s}_\mathbf{H})\big),
  \end{equation}
  where $\mathbf{s}_\mathbf{H} = \big( v(H_i) \big)_{i \in [r]}$, and we have bounded $(\delta / 8r)^{rk - t} N \ge \delta^{2 rk} N$ for all $t \ge 0$ by $r \ge 2$ and our choice of $\delta = r^{-50} \le 1/(8r)$.

  Now, with the goal of bounding each term in \eqref{eq:firstprobabilityBound} by applying \Cref{stmt:key}, fix $W \subset V(G)$ with $|W| \ge \delta^{2 r k} N$ and $\mathbf{H} = (H_i)_{i \in [r]}$, a collection of graphs, each with at most $k$ vertices.
  Observe that letting $G' = G[W]$, we have $G' \sim \Gnp(N', 1/2)$ where $N' = |W|$, and also that $\mathbf{H}$ trivially satisfies
  \begin{equation}\label{eq:tPrimeLessThanRk}
    t = \sum_{i = 1}^r v(H_i) \le r k.
  \end{equation}
  The only assumption in \Cref{stmt:key} that remains to be checked is that $N'$ satisfies \eqref{eq:boundOnNInKey}, which it does because
  \begin{equation*}
    N' \ge \delta^{2 r k} N \ge r^{-100 r k} r^{5C r k} \ge r^{4 C r k} \ge r^{C(k + t)}
  \end{equation*}
  by our choice of $N = r^{5 C r k}$ in \eqref{eq:proofOfMain:choiceOfN2}, the fixed values of $\delta = r^{-50}$ and $C = 300$ in \eqref{eq:fixedParams}, and \eqref{eq:tPrimeLessThanRk}.

  Applying \Cref{stmt:key} to $G' = G[W]$ and $\mathbf{H}$, we conclude that
  \begin{equation}\label{eq:probabilityBoundForEachBfH}
    \PP\big(G[W] \in \cal{B}(\mathbf{H}) \cap \cal{E}(\mathbf{s}_\mathbf{H})\big) \le 2^{- \delta^2 |W|^2} \le 2^{- \delta^{5 r k} N^2},
  \end{equation}
  for all $|W| \ge \delta^{2 r k}N$.
  As $W$ and $\mathbf{H}$ were arbitrary, \eqref{eq:probabilityBoundForEachBfH} holds for all terms in \eqref{eq:firstprobabilityBound}.

  Replacing \eqref{eq:probabilityBoundForEachBfH} in \eqref{eq:firstprobabilityBound}, bounding the choices for $\mathbf{H}$ by $k^r 2^{r \binom{k}{2}}$ and for $W \subset V(G)$ by $2^N$, we have
  \begin{equation}\label{eq:probThatIsActuallyExp}
    \PP\big(G \in \cal{B}(H; r)\big) \le k^r 2^{r \binom{k}{2}} 2^{N} 2^{- \delta^{5 r k} N^2} \le 2^{- \delta^{5 r k} N^2 / 2} < 1
  \end{equation}
  where the final inequalities hold since
  \begin{equation*}
    \frac{\delta^{5 r k} N^2}{2} > N + r k + r \binom{k}{2} > 0
  \end{equation*}
  by our choice of $N = r^{5 C r k}$ and $\delta = r^{-50}$.
  We have established \eqref{eq:probabilisticGoalForMainResult} and completed the proof.
\end{proof}

\begin{remark}\label{stmt:universality}
  Observe that the bound in \eqref{eq:probThatIsActuallyExp} is sufficiently small for us to take another union bound, this one over all graphs $H$ with $k$ vertices, for which there are at most $2^{\binom{k}{2}}$ choices.
  The conclusion is that, with high probability, if $G \sim \Gnp(N, 1/2)$, then $G \xxrightarrow[r]{\mathrm{ind}} H$ for all such $H$.
\end{remark}

\subsection{Structure of the paper}

The bulk of the paper is concerned with proving \Cref{stmt:key}.
However, before we proceed to the proof of \Cref{stmt:key}, we will need some tools.

First, we will have two expository \namecrefs{sec:warmUpExtension} that contain many of the ideas that we use later, but in simpler settings.
In \Cref{sec:warmUpExtension}, this simpler setting corresponds to showing that the probability that a vertex extends no copy of $H_i^-$ in a fixed colour $i \in [r]$ is exponentially small.
This section shows how to map neighbourhoods that fail to extend a copy of $H_i^-$ to independent sets in a certain hypergraph $\cal{H}$, which will illustrate how the method of hypergraph containers can be applied to this problem.
It is also where we explain how our induction hypothesis, the event $\cal{E}(\mathbf{s})$, yields a supersaturation result for our application(s) of containers.

In \Cref{sec:simplerContainerForNonJanson}, we present the proof of a container theorem for ``sets that are not $(p, R)$-Janson'' (see \Cref{stmt:containersForNonJansonGeneral}, and also \Cref{stmt:containersForNonJansonWithFingerprints}, which implies it).
While this result is not sufficiently strong to prove \Cref{stmt:key}, its proof already presents the general argument that we will later adapt to prove \Cref{stmt:containersForNonJanson}, the container \namecref{stmt:containersForNonJanson} that we do use in the proof of \Cref{stmt:key}.
We see that argument as a general method to reduce a problem about sets avoiding a global property to a simpler problem, where the sets only need to avoid a local property.
In this particular instance, we deal with this local property by using a standard hypergraph container lemma (\Cref{stmt:containersCoversButJanson}).

\Cref{sec:extensionOfJansonCollection} contains the proof that if the hypergraph $\frakI_{H_i,G_i,G}[U]$ is $(p, R)$-Janson, then adding a new vertex to $U$ produces a $(p, R+1)$-Janson hypergraph with probability $1 - \exp(-\gamma |U|)$ for some constant $\gamma = \gamma(r) > 0$ that only depends on $r$.
The proof of this statement mirrors the proof in \Cref{sec:warmUpExtension}, but replacing the hypergraph container lemma there by the (yet unproven) \Cref{stmt:containersForNonJanson}, the container theorem tailored to our application.
We postpone the proof of \Cref{stmt:containersForNonJanson} to the last section since it is the most technical part of the paper.

In \Cref{sec:proofOfKey}, we prove \Cref{stmt:key} in two stages.
The first and easier stage uses a double counting argument to show that if $\frakI_{H_i,G_i,G}[S]$ is roughly $(p, \delta p N / r)$-Janson for all subsets $S$ with $|S| \ge \delta^{2/3}N$, then $\frakI_{H_i,G_i,G}$ is $(p,p N)$-Janson.
The second stage proves \Cref{stmt:key} by induction on the Janson parameter of the family $\frakI_{H_i,G_i,G}[U]$.
To prove the inductive step, we use \Cref{stmt:extensionOfJansonCollection} to show that the probability one cannot increase this Janson parameter by adding a vertex to $U$ is very small.
This is also the part where we take a union bound over colourings of $G[U]$, which ends up being somewhat technical because we must be careful with the randomness outside $G[U]$.

The last section, \Cref{sec:containersForNonJanson}, contains the proof of \Cref{stmt:containersForNonJanson}.
As we mentioned before, this proof follows the blueprint of \Cref{sec:simplerContainerForNonJanson}, but there are extra details.
One is that we want to combine the family of copies of $H_i$ containing $v$ with the copies that are already in $U$, which does not contain $v$.
The second is that the hypergraph $\cal{H}$ that we use to deal with the extensions, defined in \Cref{sec:warmUpExtension}, is not the same as $\frakI_{H_i,G_i,G}$, the hypergraph appearing in the event whose probability we must bound.
Related to this is the fact that the hypergraph defined by $\cal{H}$ corresponds to copies of $H_i^-$, and we are interested in copies of $H_i$.
These complications make the proof more technical, but it turns out the approach we developed in \Cref{sec:simplerContainerForNonJanson} is sufficiently flexible to deal with them.

\section{Warm-up to the extension lemma}\label{sec:warmUpExtension}

In this \namecref{sec:warmUpExtension}, we give an essentially complete proof of a weak ``extension \namecref{stmt:warmUpExtension}'', \Cref{stmt:warmUpExtension}.
Very roughly, this result bounds the probability that a fixed vertex of $G$ fails to complete any copy of a graph $F$.
We will describe the setting for this result in \Cref{sec:extensionLemmas}, connecting it with the application of its stronger variant in the proof of \Cref{stmt:key}.

Of central importance in the proof of \Cref{stmt:warmUpExtension} is a hypergraph container \namecref{stmt:containersCoversButJanson}, here stated as \Cref{stmt:containersCoversButJanson}, which motivates many of the definitions that follow.
The most significant of these is that of the auxiliary hypergraph $\cal{H}$, defined in \Cref{sec:auxiliaryHypergraph}; the other definitions will, above all, assist in reasoning about it.

As we will later see, each edge of this auxiliary hypergraph represents one way to extend, using a fixed vertex $v$, an (induced) copy of $F^-$ to a copy of $F$.
\Cref{stmt:iotaIndepWarm} then fulfils a crucial requirement to apply the container method, establishing that independent sets in $\cal{H}$ correspond to pairs of graphs $(G', G)$ in which the neighbourhood of $v$ fails to extend every $F^-$ in a fixed set $U$.

\subsection{Extension lemmas}\label{sec:extensionLemmas}

The setting of \Cref{stmt:warmUpExtension} is that $\tilde{G}'$ and $\tilde{G}$ are $m$-vertex graphs such that $\tilde{G}' \subset \tilde{G}$ and $\frakI_{F^{-}, \tilde{G}', \tilde{G}}[W]$ is $(p, 2 p m)$-Janson for every $W \subset V(\tilde{G})$ with $|W| \ge m / 16$.
Let $U = V(\tilde{G})$ and let $G$ be the random graph $\Gnp(m + 1, 1/2)$ conditioned on $G[U] = \tilde{G}$.
The \namecref{stmt:warmUpExtension} bounds the probability that $G$ contains a subgraph $G'$ with the following properties:
\begin{enumerate}[(a)]
  \item $G'[U] = \tilde{G}'$,
  \item the vertex $v \in V(G) \setminus U$ does not have small degree in $G'$, and
  \item $v$ does not extend any copy of $F^- \subset \tilde{G}'$ to an induced $F$ in $G$ using the edges in $G'$.
\end{enumerate}

In our application, $U$ will be a set of vertices for which the colouring $c : E(G) \to [r]$ is fixed, $v$ will be a vertex not in $U$ such that $d_{G_i}(v,U) \ge |U|/4r$ for some colour $i \in [r]$, and we will set $\tilde{G}=G[U]$ and $\tilde{G}'=G_i[U]$.
In this setting, we could use this \namecref{stmt:warmUpExtension} (when $r = 2$) to conclude that it is extremely likely that $v$ extends some $i$-coloured copy of $F^- := H_i^-$ to a copy of $F = H_i$.

\begin{lem}\label{stmt:warmUpExtension}
  Let $m, s, k \in \NN$ with $s < k \le m$, $p = 2^{-20} k^{-2}$ and $F$ be a graph on $s + 1$ vertices.
  Further let $\tilde{G}'$ and $\tilde{G}$ be graphs on $m$ vertices satisfying $\tilde{G}' \subset \tilde{G}$.

  If $\frakI_{F^{-}, \tilde{G}', \tilde{G}}[W]$ is $(p, 2 p m)$-Janson for every $W \subset U = V(\tilde{G})$ with $|W| \ge m / 16$, then
  \begin{equation}\label{eq:probOfGPrimeWarm}
    \PP
    \left(
        \exists\, G' \subset G \\
      :
      \begin{array}{@{}c@{}}
        G'[U] = \tilde{G}',  ~ d_{G'}(v) \ge m / 8 \\
         \text{ and } ~ \frakI_{F, G', G}=\emptyset
      \end{array}
      \
      \middle \vert \ G[U] = \tilde{G}
    \right) \le 2^{-m / 32},
  \end{equation}
  where $G \sim \Gnp(m + 1, 1/2)$ and $V(G) = U \cup \{v\}$.
\end{lem}

We refer to \Cref{stmt:warmUpExtension} as a weak extension lemma because it involves the event $\{\frakI_{F, G', G} = \emptyset\}$, whereas our main extension result, \Cref{stmt:extensionOfJansonCollection}, involves instead $\{\frakI_{F, G', G}$ is not $(p, R)$-Janson$\}$.
The proof of \Cref{stmt:warmUpExtension} presents most of the ideas that we will require in \Cref{sec:extensionOfJansonCollection}, while avoiding some technicalities from dealing with this more complicated property of $\frakI_{F, G', G}$.

\subsection{The auxiliary hypergraph}\label{sec:auxiliaryHypergraph}

We want to bound the probability in \eqref{eq:probOfGPrimeWarm} using the method of hypergraph containers.
To accomplish that, the first step is to represent the graphs $G' \subset G$ such that $G'[U] = \tilde{G}'$ and $\frakI_{F, G', G} = \emptyset$ as independent sets in an auxiliary hypergraph $\cal{H}$.
This auxiliary hypergraph will be a function of $\tilde{G}'$ and $\tilde{G}$ only, which crucially means that $\cal{H}$ does not depend on the edges of $G$ (or $G' \subset G$) between $v$ and $U$.

Define $w \in V(F)\setminus V(F^{-})$ and note that the only randomness in the event inside the probability in \eqref{eq:probOfGPrimeWarm} comes from the edges of $G$ between $U$ and $v$, since we condition on the event $\big\{G[U]=\tilde{G}\big\}$.
Our goal is to have each edge of $\cal{H}$ correspond to one way of extending a copy of $F^{-}$ in $\tilde{G}'$, which is also induced in $\tilde{G}$, to a (still induced) copy of $F$, by adding the vertex $v$ in the role of $w$.
To be precise, for each edge $L \in \frakI_{F^{-},\tilde{G}',\tilde{G}}$, fix a bijection $\phi_L : L \to V(F^{-})$ such that $F^{-} \cong_{\phi_L} \tilde{G}'[L] = \tilde{G}[L]$.
Now define $\cal{H}$ to be the hypergraph with vertex set $U \times \{0, 1\}$ and edge set
\begin{equation}\label{eq:defOfCalHwarm}
  \cal{H} = \Big\{E_L : L \in \frakI_{F^-, \tilde{G}', \tilde{G}}\Big\},
\end{equation}
where
\begin{equation*}
  E_L = \bigg\{\Big(u, \mathds{1}\big[\phi_L(u) \in \mathrm{N}_F(w)\big]\Big) : u \in L\bigg\},
\end{equation*}
for every $L \in \frakI_{F^{-},\tilde{G}',\tilde{G}}$.
In words, the vertex set of $\cal{H}$ is two copies of $U$, corresponding to the neighbours and non-neighbours of $w$ in $F$, and its edge set has the following property: if $u \in \mathrm{N}_{G'}(v)$ for every $(u,1) \in E_L$ and $u \not \in \mathrm{N}_G(v)$ for every $(u,0) \in E_L$, then $L \cup \{v\} \in \frakI_{F, G', G}$.

To formalise that independent sets in $\cal{H}$ correspond to pairs of graphs $(G', G)$ of our interest, denote the non-neighbours of $v \not \in U$ by
\begin{equation}
  \mathrm{N}_G(v)^\complement = U \setminus \mathrm{N}_G(v)
\end{equation}
and, for each $A \subset V(\cal{H})$ and $i \in \{0, 1\}$, let
\begin{equation*}
  A^{(i)} = \big\{u \in U : (u, i) \in A\big\}.
\end{equation*}
We repeat the following property, previously stated without these definitions, for future reference.

\begin{obs}\label{stmt:defOfCalHIsCorrect}
  Let $G'$ and $G$ be graphs satisfying $G' \subset G$, $G'[U] = \tilde{G}'$ and $G[U] = \tilde{G}$.
  If $L \in \frakI_{F^-, \tilde{G}', \tilde{G}}$ and $E = E_L \in \cal{H}$, then we can extend $L$ to a copy of $F \subset G'$ induced in $G$, i.e.\ $L \cup \{v\} \in \frakI_{F,G',G}$, whenever
  \begin{equation*}
    E^{(0)} \subset \mathrm{N}_G(v)^\complement \qquad \text{and} \qquad E^{(1)} \subset \mathrm{N}_{G'}(v).
  \end{equation*}
\end{obs}

The most useful consequence of this fact, which allows applying the method of hypergraph containers, is more conveniently stated with some additional notation.
Let
\begin{equation}\label{eq:eventOfGPrimeWarm}
  \Gamma_G = \left\{
    G' \subset G :
  \begin{array}{@{}c@{}}
    G'[U] = \tilde{G}', ~ d_{G'}(v) \ge m / 8 \\
    \text{and }\frakI_{F,G',G} = \emptyset
  \end{array}
  \right\}
\end{equation}
be the collection of $G'$ whose existence is the event in the probability of \eqref{eq:probOfGPrimeWarm}.
To index vertex subsets of $\cal{H}$ by graphs $G'$ and $G$, further let
\begin{equation}\label{eq:definitionIotaWarm}
  \iota(G', G) = \big\{(u,0) : u  \in \mathrm{N}_G(v)^{\complement}\big\} \cup \big\{(u,1) : u \in \mathrm{N}_{G'}(v)\big\},
\end{equation}
and note that if $I=\iota(G',G)$, then
\begin{equation}\label{eq:projectionOfIotaWarm}
   I^{(0)} = \mathrm{N}_G(v)^\complement \qquad \text{and} \qquad I^{(1)} = \mathrm{N}_{G'}(v).
\end{equation}
To justify one of the premises in \Cref{stmt:iotaIndepWarm}, recall that we have conditioned the distribution of $G \sim \Gnp(m + 1, 1/2)$ on $\big\{G[U] = \tilde{G}\big\}$ in \eqref{eq:probOfGPrimeWarm}.

\begin{obs}\label{stmt:iotaIndepWarm}
  If $G[U] = \tilde{G}$ and $G' \in \Gamma_G$, then $\iota(G', G)$ is an independent set in $\cal{H}$.
\end{obs}

\begin{proof}
  Let $E = E_L \in \cal{H}$, and suppose by contradiction that $E \subset I = \iota(G',G)$.
  It then follows from \eqref{eq:projectionOfIotaWarm} that $E^{(1)} \subset \mathrm{N}_{G'}(v)$ and $E^{(0)} \subset \mathrm{N}_G(v)^\complement$, so \Cref{stmt:defOfCalHIsCorrect} implies that $L \cup \{v\} \in \frakI_{F,G',G}$.
  Therefore, the hypergraph $\frakI_{F,G',G}$ is not empty, and we could not have $G' \in \Gamma_G$ by definition, \eqref{eq:eventOfGPrimeWarm}.
  This contradiction means that $E \not \subset I$, which, as $E$ was an arbitrary edge of $\cal{H}$, means that $I = \iota(G',G)$ is an independent in $\cal{H}$.
\end{proof}

\subsection{Containers}

Having connected independent sets in $\cal{H}$ to pairs of graphs $(G', G)$ such that $G' \in \Gamma_G$, we can proceed to the next step of the proof, which is applying a container \namecref{stmt:containersCoversButJanson} to $\cal{H}$.
In this section, we use a simpler version of the container \namecref{stmt:containersCoversButJanson} than the one required in \Cref{sec:extensionOfJansonCollection}, since this avoids technical complications that arise in the proof of \Cref{stmt:extensionOfJansonCollection}.

Even though a previous result of \citet[Theorem 2.1]{BS20} would most likely suffice to prove \Cref{stmt:warmUpExtension}, we prefer to adapt (in Appendix~\ref{app:containersCoversButJanson}) the mildly stronger \namecref{stmt:containersCoversButJanson} of \citet{CS24+}.
This result provides a small family of containers for the independent sets of $\cal{H}$, with each container inducing a subhypergraph that is not $(p, R)$-Janson.

\begin{restatable}[{\citet[modified Theorem A]{CS24+}}]{thm}{containersCoversButJanson}\label{stmt:containersCoversButJanson}
  Let $\cal{G}$ be an $s$-uniform hypergraph with $n$ vertices.
  For all $0 < \zeta \le 1$ and $0 < p \le \zeta/(8 s^2)$, there is a family $\cal{S} \subset 2^{V(\cal{G})}$ and functions
  \begin{equation}
    \phi : \cal{I}(\cal{G}) \to \cal{S} \quad \text{ and } \quad \psi : \cal{S} \to 2^{V(\cal{G})}
  \end{equation}
  such that:
  \begin{enumerate}[\normalfont (i)]
    \item For each $I \in \cal{I}(\cal{G})$, we have $\phi(I) \subset I \subset \psi(\phi(I))$. \label{item:containersCoversHaveFingerprints}
    \item Each $S \in \cal{S}$ has at most $8 s^2 p n / \zeta$ elements. \label{item:containersCoversHaveSmallFingerprints}
    \item \label{item:containersAreNonJanson} For every $S \in \cal{S}$, letting $X = \psi(S)$, $\cal{G}[X]$ is not $(p, \zeta p |X|)$-Janson.
  \end{enumerate}
\end{restatable}

We will apply \Cref{stmt:containersCoversButJanson} with $\cal{G} = \cal{H}$ and define $\cal{X} = \{\psi(S) : S \in \cal{S}\}$ to be our container family.
Combining \cref{item:containersCoversHaveFingerprints} in the statement of that \namecref{stmt:containersCoversButJanson} with \Cref{stmt:iotaIndepWarm} guarantees that whenever $G'$ and $G$ satisfy $G[U]=\tilde{G}$ and $G' \in \Gamma_G$, there is some container $X \in \cal{X}$ for which $\iota(G', G) \subset X$.
Our goal is now to show that $|X^{(0)}| \le (1 - \gamma)|U|$ for some constant $\gamma > 0$ using \cref{item:containersAreNonJanson} in \Cref{stmt:containersCoversButJanson} and a supersaturation result.
This will be sufficient to obtain the probability bound in \Cref{stmt:warmUpExtension}, because $\mathrm{N}_G(v)^\complement \subset X_0$ and this non-neighbourhood is a binomial random set, due to $G \sim \Gnp(m + 1, 1/2)$.

To prove the required supersaturation result, we will use our assumption that $\frakI_{F^-, \tilde{G}', \tilde{G}}$ is $(p, 2 p m)$-Janson for every $W \subset U$ with $|W| \ge m / 16$.
However, to connect this assumption with the fact that $\cal{H}[X]$ is not $(p, p |X|)$-Janson for a container $X \in \cal{X}$, we must relate the Janson properties of the hypergraphs $\cal{H}$ and $\frakI_{F^-, \tilde{G}', \tilde{G}}$.

Towards that goal, let $\pi : U \times \{0, 1\} \to U$ be the projection of each pair $(u, i)$ onto its first coordinate $u$.
Observe that if $E = E_L \in \cal{H}$ for some $L \in \frakI_{F^-, \tilde{G}', \tilde{G}}$, then $\pi(E) = L$.
Moreover, for every $L \in \frakI_{F^-, \tilde{G}', \tilde{G}}$, there is $E \in \cal{H}$ such that $\pi(E) = L$.
We conclude that $\pi(\cal{H}) = \frakI_{F^{-}, \tilde{G}',\tilde{G}}$, where we extend the application of $\pi$ from a single vertex to hypergraphs $\cal{G}$ by
\begin{equation*}
  V\big(\pi(\cal{G})\big) = \pi(V(\cal{G})) \qquad \text{and} \qquad E\big(\pi(\cal{G})\big) = \big\{\pi(E) : E \in \cal{G}\big\}.
\end{equation*}

\begin{obs}\label{stmt:piCalHIsHypergraphOfInducedCopies}
  If $\pi : U \times \{0, 1\} \to U$ is the projection onto the first coordinate, then $$\pi(\cal{H}) = \frakI_{F^{-}, \tilde{G}',\tilde{G}}.$$
\end{obs}

The next \namecref{stmt:wholeCopiesAreJanson} will allow us to prove that if $\frakI_{F^{-},\tilde{G}',\tilde{G}}$ is $(p, R)$-Janson, then so is $\cal{H}$.
We state it in greater generality than we need, but it is a trivial exercise to observe that $\pi$ satisfies this more general condition.
\Cref{stmt:wholeCopiesAreJanson} is in fact a corollary of \Cref{stmt:pullbackPreservesLambdaProperties}, and we will therefore defer its proof to when that other result is proven (see \Cref{sec:containersForNonJanson} and Appendix~\ref{app:propertiesOfMeasures}).

\begin{lem}\label{stmt:wholeCopiesAreJanson}
    Let $R \ge 0$ and $p > 0$.
    Further let $\cal{G}$ be a hypergraph and $\pi : V(\cal{G}) \to U$ be a function satisfying
    \begin{equation}\label{eq:piAssumptionInWarmUp}
        |\pi(E)| = |E| \qquad \text{for every } E \in \cal{G}.
    \end{equation}
    If $\cal{G}$ is not $(p,R)$-Janson, then $\pi(\cal{G})$ is also not $(p,R)$-Janson.
\end{lem}

We now record the trivial observation that $\pi$ satisfies \eqref{eq:piAssumptionInWarmUp} when $\cal{G} = \cal{H}$ for future reference.

\begin{obs}\label{stmt:piSatisfiesAssumption}
   $|\pi(E)| = |E|$ for every $E \in \cal{H}$.
\end{obs}

By \cref{item:containersAreNonJanson} in \Cref{stmt:containersCoversButJanson} and \Cref{stmt:wholeCopiesAreJanson}, it follows that $\pi(\cal{H}[X])$ is not $(p, p |X|)$-Janson for each container $X$.
However, our assumption in \Cref{stmt:warmUpExtension} is that $\pi(\cal{H})[W]$ is $(p, 2 p m)$-Janson for all large sets $W \subset U$, and this does not immediately imply anything about the size of $X$.
For example, we might try to apply this assumption with $W = \pi(X)$, and hope that
\begin{equation}\label{eq:idealContainment}
  \pi(\cal{H})[W] \subset \pi(\cal{H}[X]),
\end{equation}
but unfortunately \eqref{eq:idealContainment} does not always hold: if $X = U \times \{0\}$, then we have $\pi(\cal{H})[\pi(X)] = \pi(\cal{H})$ and $\cal{H}[X] = \emptyset$.
In order to deal with this issue, we will instead apply our assumption to the set $W = X^{(0)}\cap X^{(1)}$.
By the following \namecref{stmt:containmentHypergraphs}, this choice \emph{does} satisfy \eqref{eq:idealContainment}.

\begin{lem}\label{stmt:containmentHypergraphs}
  For all $Y \subset U \times \{0, 1\}$, we have
  \begin{equation}
    \pi(\cal{H})[W] \subset \pi(\cal{H}[Y]),
  \end{equation}
  where $W = Y^{(0)} \cap Y^{(1)}$.
\end{lem}

\begin{proof}
  Take $L \in \pi(\cal{H})[W]$ with the goal of proving that $L \in \pi(\cal{H}[Y])$.
  By definition of $W$,
  \begin{equation}\label{eq:containmentOfL}
    L \subset Y^{(0)} \cap Y^{(1)}.
  \end{equation}
  Moreover, since $L \in \pi(\cal{H})$, there is $E \in \cal{H}$ such that $\pi(E) = L$.
  Therefore,
  \begin{equation*}
      E \subset \pi^{-1}(L) = \big\{(u,a):u\in L,~a\in\{0,1\}\big\},
  \end{equation*}
  but we also have, by \eqref{eq:containmentOfL}, that
  \begin{equation*}
    \pi^{-1}(L) \subset \big(Y^{(0)} \times \{0\}\big) \cup \big(Y^{(1)} \times \{1\}\big) = Y,
  \end{equation*}
  and so $E \subset Y$.
  We conclude that $E \in \cal{H}[Y]$ and also $L = \pi(E) \in \pi(\cal{H}[Y])$, as required.
\end{proof}

Now, taking $W = X^{(0)}\cap X^{(1)}$ for some container $X \in \cal{X}$ and applying \Cref{stmt:containmentHypergraphs}, we have that $\pi(\cal{H})[W] \subset \pi(\cal{H}[X])$.
Using \cref{item:containersAreNonJanson} of \Cref{stmt:containersCoversButJanson}, \Cref{stmt:piCalHIsHypergraphOfInducedCopies}, \Cref{stmt:wholeCopiesAreJanson} and \Cref{stmt:JansonIsDownset}, we will conclude that $\frakI_{F^-, \tilde{G'}, \tilde{G}}[W]$ is not $(p, 2 p m)$-Janson (see \Cref{stmt:WisnotJanson}).
Our assumption in \Cref{stmt:warmUpExtension} then implies that $|W| < m / 16$, so every $X \in \cal{X}$ satisfies
\begin{equation*}
  \big|X^{(0)} \cap X^{(1)}\big| < \frac{m}{16},
\end{equation*}
by definition of $W$, and, since $X^{(0)} \cup X^{(1)} \subset U$, it follows that
\begin{equation}\label{eq:smallIntersection}
  \big|X^{(0)}\big| + \big|X^{(1)}\big| = \big|X^{(0)} \cup X^{(1)}\big| + \big|X^{(0)} \cap X^{(1)}\big| < \Big(1 + \frac{1}{16}\Big) m.
\end{equation}
To establish that $X^{(0)}$ is not large when $I = \iota(G', G) \subset X$, recall that it follows from $G' \in \Gamma_G$ that
\begin{equation}\label{eq:X1IsLarge}
  |X^{(1)}| \ge d_{G'}(v) \ge \frac{m}{8}
\end{equation}
where the first inequality holds because $\mathrm{N}_{G'}(v) \subset X^{(1)}$ by \eqref{eq:projectionOfIotaWarm}.
Combining \eqref{eq:smallIntersection} and \eqref{eq:X1IsLarge} yields
\begin{equation*}
  |X^{(0)}| \le \Big(1 - \frac{1}{16}\Big)m,
\end{equation*}
which, together with
\begin{equation}\label{eq:eventWithTinyProbability}
  \mathrm{N}_G(v)^\complement \subset X^{(0)}
\end{equation}
will allow us to use the randomness in the distribution of $G \sim \Gnp(m + 1, 1/2)$ to bound the probability of \eqref{eq:eventWithTinyProbability} by an exponentially small term.
The resulting upper bound is sufficiently strong to overcome the union bound over all $X \in \cal{X}$, so we can now formalise the proof of \Cref{stmt:warmUpExtension}.

\begin{proof}[Proof of \Cref{stmt:warmUpExtension}]
  Applying \Cref{stmt:containersCoversButJanson} to $\cal{H}$, defined in \eqref{eq:defOfCalHwarm}, with $\zeta = 1$ we obtain a family $\cal{S}$ and functions $\phi$, $\psi$ satisfying \cref{item:containersCoversHaveFingerprints,item:containersCoversHaveSmallFingerprints,item:containersAreNonJanson} in \Cref{stmt:containersCoversButJanson}, so let
  \begin{equation*}
    \cal{X} = \{\psi(S) : S \in \cal{S}\}.
  \end{equation*}

  With the goal of taking a union bound over $X \in \cal{X}$, fix $G$ on $m + 1$ vertices satisfying $G[U] = \tilde{G}$, and observe that $\iota(G', G) \in \cal{I}(\cal{H})$ for all $G' \in \Gamma_G$ by \Cref{stmt:iotaIndepWarm}.
  Now, it follows from \cref{item:containersCoversHaveFingerprints} of \Cref{stmt:containersCoversButJanson} that there exists $X \in \cal{X}$ such that $\iota(G', G) \subset X$, and therefore
  \begin{equation}\label{eq:neighbourhoodsAndNonNeighborhoodsOfVInContainerWarm}
    \mathrm{N}_G(v)^\complement = I^{(0)} \subset X^{(0)} \qquad \text{ and } \qquad \mathrm{N}_{G'}(v) = I^{(1)} \subset X^{(1)},
  \end{equation}
  where $I = \iota(G', G)$.
  Using \eqref{eq:neighbourhoodsAndNonNeighborhoodsOfVInContainerWarm}, we further refine $\cal{X}$ to
  \begin{equation*}
    \cal{X}' = \big\{X \in \cal{X} : |X^{(1)}| \ge m / 8 \big\}
  \end{equation*}
  and preserve the property that for every $G' \in \Gamma_G$, there is $X \in \cal{X}'$ such that $\iota(G', G) \subset X$, because
  \begin{equation}\label{eq:X1IsLargeWarm}
    |X^{(1)}| \ge d_{G'}(v) \ge \frac{m}{8}
  \end{equation}
  follows from $G' \in \Gamma_G$, defined in \eqref{eq:eventOfGPrimeWarm}.
  That is, if $X \in \cal{X}$ contains some $\iota(G', G)$, then it is also in $\cal{X}'$ by \eqref{eq:X1IsLargeWarm}, and crucially, for every $G' \in \Gamma_G$, there exists $X \in \cal{X}'$ such that
  \begin{equation}\label{eq:nonNeighbourhoodInContainer}
    \mathrm{N}_G(v)^\complement \subset X^{(0)}.
  \end{equation}

  Taking a union bound over $\cal{X}'$ then yields, for the probability in \eqref{eq:probOfGPrimeWarm},
  \begin{equation}\label{eq:upperBoundViaContainersWarm}
    \PP\big(\exists G' \subset G : G' \in \Gamma_G \mid G[U] = \tilde{G}\big) \le \sum_{X \in \cal{X}'} \PP\big(\mathrm{N}_G(v)^\complement \subset X^{(0)}\big),
  \end{equation}
  where we used \eqref{eq:nonNeighbourhoodInContainer} to bound the probability of the event $\{\iota(G',G)\subset X\}$ by that of the event $\{\mathrm{N}_G(v)^\complement \subset X^{(0)}\big\}$.
  We shift our focus to obtaining an upper bound for the probability of this event that holds for every $X \in \cal{X}'$, so we fix $X \in \cal{X}'$ and set $W = X^{(0)} \cap X^{(1)}$.
  The following claim, together with the assumption that $\frakI_{F^{-}, \tilde{G}', \tilde{G}}[W']$ is $(p, 2 p m)$-Janson for every $W' \subset U$ with $|W'| \ge m / 16$, will allow us to bound the size of $X^{(0)}$ from above.

  \begin{claim}\label{stmt:WisnotJanson}
    $\frakI_{F^{-}, \tilde{G}', \tilde{G}}[W]$ is not $(p, 2 p m)$-Janson.
  \end{claim}

  \begin{claimproof}
    By \Cref{stmt:piCalHIsHypergraphOfInducedCopies} and \Cref{stmt:containmentHypergraphs}, we have $$\frakI_{F^{-}, \tilde{G}', \tilde{G}}[W] = \pi(\cal{H})[W] \subset \pi(\cal{H}[X]).$$
    Therefore, if we establish that the hypergraph $\pi(\cal{H}[X])$ is not $(p, 2 p m)$-Janson, then we are done, because being $(p, 2 p m)$-Janson is increasing with respect to inclusion, by \Cref{stmt:JansonIsDownset}.

    To show this, observe first that $\cal{H}[X]$ is not $(p, p |X|)$-Janson by \cref{item:containersAreNonJanson} in \Cref{stmt:containersCoversButJanson} and our choice of $\zeta = 1$.
    Since $|X| \le 2m = v(\cal{H})$, and recalling \Cref{stmt:jansonParametersMonotone}, it follows that the hypergraph $\cal{H}[X]$ is also not $(p, 2 p m)$-Janson.
    Thus, by \Cref{stmt:wholeCopiesAreJanson}, we deduce that $\pi(\cal{H}[X])$ is not $(p, 2 p m)$-Janson, and, by our previous reasoning, that neither is $\frakI_{F^{-}, \tilde{G}', \tilde{G}}[W]$, as claimed.
  \end{claimproof}

  Since we know, by assumption, that $\frakI_{F^{-}, \tilde{G}', \tilde{G}}[W']$ is $(p, 2 p m)$-Janson for every $W' \subset U$ with $|W'| \ge m / 16$, \Cref{stmt:WisnotJanson} implies that $|W| < m / 16$ and therefore
  \[|X^{(0)}| + |X^{(1)}| = \big|X^{(0)} \cup X^{(1)}\big| + \big|X^{(0)} \cap X^{(1)}\big| < m + \frac{m}{16}.\]
  As $X \in \cal{X}'$, we have $|X^{(1)}| \ge m / 8$, and therefore
  \[|X^{(0)}| < \Big(1+ \frac{1}{16}\Big)m - |X^{(1)}| \le m + \frac{m}{16} - \frac{m}{8} = \Big(1 - \frac{1}{16}\Big) m,\]
  that is,
  \begin{equation}\label{eq:containerMissesU}
    |U \setminus X^{(0)}| \ge \frac{m}{16}.
  \end{equation}
  We conclude from \eqref{eq:containerMissesU} that
  \begin{equation}\label{eq:boundInProbabilityOfEachX}
    \PP\Big(\mathrm{N}_G(v)^\complement \cap \big(U\setminus X^{(0)}\big) = \emptyset\Big) = 2^{-|U \setminus X^{(0)}|} \le 2^{-m / 16},
  \end{equation}
  where we used that $G \sim \Gnp(m + 1, 1/2)$ and $v \not \in U$.
  Replacing \eqref{eq:boundInProbabilityOfEachX} in \eqref{eq:upperBoundViaContainersWarm}, we obtain
  \begin{equation}\label{eq:afterTheUnionBound}
    \sum_{X \in \cal{X}'} \PP\big(\mathrm{N}_G(v)^\complement \subset X^{(0)}\big) \le |\cal{X}'| \,  2^{-m / 16}.
  \end{equation}

  To bound the size of $\cal{X}'$ by $|\cal{X}|$, we enumerate the latter using that each container $X$ is a function of $S \in \cal{S}$.
  As each $S \in \cal{S}$ satisfies
  \[|S| \le 16 p s^2 m \le 2^{-16} m\]
  by \cref{item:containersCoversHaveSmallFingerprints} in \Cref{stmt:containersCoversButJanson} and our choice of $\zeta = 1$, where the last inequality holds by our choice of $p = 2^{-20} k^{-2}$ and the assumption that $v(F) = s \le k$, we have
  \begin{equation}\label{eq:sizeOfCalXPrime}
    |\cal{X}'| \le |\cal{X}| \le \sum_{t = 0}^{2^{-16} m} \binom{2 m}{t} \le 2^{m / 32},
  \end{equation}
  where, recall, $\cal{H}$ has $2 m$ vertices.
  Combining \eqref{eq:sizeOfCalXPrime} and \eqref{eq:afterTheUnionBound}, we complete the proof:
  \begin{equation*}
    \PP\big(\exists G' \subset G : G' \in \Gamma_G \mid G[U] = \tilde{G}\big) \le 2^{m / 32} \, 2^{-m / 16} \le 2^{-m / 32}
  \end{equation*}
  as required.
\end{proof}

\section{A general container theorem for non-Janson sets}\label{sec:simplerContainerForNonJanson}

In this \namecref{sec:simplerContainerForNonJanson}, we prove a preliminary version of our main technical result.

\begin{thm}\label{stmt:containersForNonJansonGeneral}
  Let $s, n \in \NN$ with $s \le n$, and let $q, p, R, \eta > 0$ satisfy
  \begin{equation*}
    q \le \dfrac{1}{16}, \qquad p \le \frac{q}{2^{10} s^2}, \qquad R \ge 2^{-6} p n \qquad \text{ and } \qquad \eta = 2^{-2s - 2}.
  \end{equation*}
  For every $s$-uniform hypergraph $\cal{H}$ with $n$ vertices, there exists a family $\cal{X} \subset 2^{V(\cal{H})}$ with
  \begin{equation}\label{eq:sizeOfContainerFamilyForNonJansonGeneral}
    |\cal{X}| \le \bigg(\frac{\,2\,}{q}\bigg)^{8 q n}
  \end{equation}
  such that the following hold.

  \begin{enumerate}[\normalfont (i)]
      \smallskip
    \item If $L \subset V(\cal{H})$ and $\cal{H}[L]$ is not $(p/q,\eta R)$-Janson, then $L \subset X$ for some $X \in \cal{X}$. \label{item:containersForNonJansonGeneral:containerInclusion}
      \smallskip
    \item For each $X \in \cal{X}$, the hypergraph $\cal{H}[X]$ is not $(p, R)$-Janson. \label{item:containersFornonJansonGeneral:containersAreNotJanson}
  \end{enumerate}
\end{thm}

We now discuss the proof of \Cref{stmt:containersForNonJansonGeneral}.
The first step is defining the auxiliary hypergraph
\begin{equation*}
  \cal{J} = \big\{ L \subset V : \cal{H}[L] \text{ is $(p/q, \eta R)$-Janson} \big\},
\end{equation*}
where $V = V(\cal{H})$.
It follows immediately from the definition that sets $L \subset V$ for which $\cal{H}[L]$ is not $(p/q, \eta R)$-Janson are not edges of $\cal{J}$, but, more importantly, \Cref{stmt:JansonIsDownset} implies that each such $L$ is an independent set in $\cal{J}$.

We can now see that \cref{item:containersForNonJansonGeneral:containerInclusion} in \Cref{stmt:containersForNonJansonGeneral} could be equivalently phrased as ``for all $I \in \cal{I}(\cal{J})$, there exists $X \in \cal{X}$ such that $I \subset X$''.
This formulation suggests applying a container \namecref{stmt:containersCoversButJanson} to this hypergraph, but the edges of $\cal{J}$ could have size comparable to $n = |V|$.
To avoid that issue, we first reduce to an alternate setting involving independent sets in $s$-uniform hypergraphs.

To state the \namecref{stmt:containersHardcovers} that we use in that reduction, we need two definitions.
The first is standard: we say that a hypergraph $\cal{C}$ covers another hypergraph $\cal{G}$ when $\cal{G} \subset \langle \cal{C} \rangle$, where $$\langle \cal{C} \rangle = \{L \subset V(\cal{C}) : \exists A \in \cal{C}, A \subset L\}$$ is the up-set of $\cal{C}$.
The second is that of the non-strict link of a hypergraph, which is deceptively similar to the ordinary notion of hypergraph link.

\begin{defi}
  For a hypergraph $\cal{G}$ and a set $Y$, let
  \begin{equation*}
    \link{\cal{G}}{Y} = \big\{E \setminus Y : E \in \cal{G}\big\}
  \end{equation*}
  denote the non-strict link of $\cal{G}$ with respect to $Y$.
\end{defi}

Interestingly, this reduction to $s$-uniform hypergraphs is proven by applying a container \namecref{stmt:containersHardcovers} (\Cref{stmt:containersHardcovers}, below) that has no dependency on the uniformity of the hypergraph.
It is not immediate that Theorem~E in \cite{CS24+} implies \Cref{stmt:containersHardcovers}, so we give a short (and, in fact, self-contained) proof of the result we will use in Appendix~\ref{app:containersHardcovers}.
In the statement and in the rest of the paper, when $0 \le q \le 1$, we write $V_q$ to denote a $q$-random subset of $V$.

\begin{restatable}[{\citet[modified Theorem E]{CS24+}}]{thm}{containersHardcovers}\label{stmt:containersHardcovers}
  Let $\cal{G}$ be a hypergraph with vertex set $V$.
  For all $\alpha, q \in \RR$ satisfying $0 < q \le \alpha < 1$, there exists a family $\cal{T} \subset 2^V$ and a function $\varphi : \cal{I}(\cal{G}) \to \cal{T}$ such that:
  \begin{enumerate}[\normalfont (a)]
      \smallskip
    \item For each $I \in \cal{I}\big(\cal{G}\big)$, we have $\varphi(I) \subset I$. \label{item:containersHardcoversHaveFingerprints}
      \smallskip
    \item Each $T \in \cal{T}$ has at most $q n / \alpha$ elements, where $n = |V|$. \label{item:containersHardcoversHaveSmallFingerprints}
      \smallskip
    \item For every $T \in \cal{T}$, there exists a hypergraph $\cal{C}_T$ with vertex set $V$ that covers $\cal{G}$ and satisfies
      \begin{equation}\label{eq:conditionalProbOfLInQSet}
        \PP\big(L \subset V_q \mid V_q \in \cal{I}(\link{\cal{G}}{T})\big) > (1 - \alpha)^{|L|} q^{|L|}
      \end{equation}
      for all $L \not \in \cal{C}_T$; moreover, for all $I \in \cal{I}(\cal{G})$ such that $T = \varphi(I)$, we have $I \in \cal{I}(\cal{C}_T)$.\label{item:containersHardcoversAreCoverable}
  \end{enumerate}
\end{restatable}

We will apply \Cref{stmt:containersHardcovers} to $\cal{J}$, obtaining as a result a family of sets $T \in \cal{T}$, each with a corresponding hypergraph $\cal{C}_T$.
Now, for each $I \in \cal{I}(\cal{J})$, there is some $T \in \cal{T}$ such that $I \in \cal{I}(\cal{C}_T)$ by \cref{item:containersHardcoversAreCoverable} in \Cref{stmt:containersHardcovers}.
Therefore,
\begin{equation}\label{eq:coveringIndenpendentSets}
  \cal{I}(\cal{J}) \subset \bigcup_{T \in \cal{T}} \cal{I}(\cal{C}_T).
\end{equation}
However, we also want the hypergraphs in the right-hand side of \eqref{eq:coveringIndenpendentSets} to be $s$-uniform.
The first step is replacing each $\cal{C}_T$ by its up-set.

\begin{obs}\label{stmt:independentSetsWithUpsets}
  Let $\cal{G}$ be a hypergraph.
  If $I \in \cal{I}(\cal{G})$, then $I \in \cal{I}(\langle \cal{G} \rangle)$.
\end{obs}

\begin{proof}
  Assume that $I \not \in \cal{I}(\langle \cal{G} \rangle)$.
  Hence, there is $E \in \langle \cal{G} \rangle$ such that $E \subset I$, and by definition of the up-set, there is also $E' \in \cal{G}$ such that $E' \subset E \subset I$, so $I \not \in \cal{I}(\cal{G})$.
\end{proof}

The trivial observation that completes the reduction to $s$-uniform hypergraphs says that \Cref{stmt:independentSetsWithUpsets} also holds if we replace $\cal{C}_T$ by a subhypergraph.
In our case, the implication is that if $\cal{C}'_T \subset \langle \cal{C}_T \rangle$ for all $T \in \cal{T}$, then $$\cal{I}(\cal{J}) \subset \bigcup_{T \in \cal{T}} \cal{I}(\cal{C}'_T).$$
The subhypergraph $\cal{C}'_T$ that we will take is simply the set of edges with size $s$, i.e.\ $$\cal{C}'_T = \langle \cal{C}_T \rangle_{= s}$$ where we define $$\cal{G}_{= s} = \{E \in \cal{G} : |E| = s\}$$ for any hypergraph $\cal{G}$ and $s \in \NN$.

Having reduced the problem to a collection of hypergraphs with edges of size $s$, we can apply another container \namecref{stmt:containersCoversButJanson} to each $\cal{C}'_T$.
Recall that we have previously used the next statement in \Cref{sec:warmUpExtension}, and that its easy deduction from a result in the literature \cite[Theorem A]{CS24+} can be found in Appendix~\ref{app:containersCoversButJanson}.

\containersCoversButJanson*

Consider the following recap of our overview so far.
First, we take an $I \in \cal{I}(\cal{J})$, and from \Cref{stmt:containersHardcovers} and \Cref{stmt:independentSetsWithUpsets} we obtain $\varphi(I) = T \subset I$ and $\cal{C}_T'$ such that $I \in \cal{I}(\cal{C}'_T)$.
Then, applying \Cref{stmt:containersCoversButJanson} with $\cal{G} = \cal{C}'_T$ yields a family $\cal{S}_T$ from which we retrieve a container $X \supset I$.
This container has the property that $\cal{C}'_T[X]$ is not $(p, 2^{-8} p |X|)$-Janson, and we want to take $X$ to be the container for this fixed $I$ in \Cref{stmt:containersForNonJansonGeneral}.
It is easy to deduce the claimed bound~\eqref{eq:sizeOfContainerFamilyForNonJansonGeneral} on the number of such containers from the bounds on $|\cal{T}|$ and $|\cal{S}_T|$ given by the corresponding container \namecrefs{stmt:containersHardcovers}.
However, it is not yet clear how to deduce that $\cal{H}[X]$ is not $(p, R)$-Janson when we only know the Janson properties of $\cal{C}'_T[X]$.

To show that $\cal{H}[X]$ is not $(p, R)$-Janson, we fix an arbitrary $\nu : \cal{H}[X] \to \RR_{\ge 0}$ with the objective of establishing that
\begin{equation}\label{eq:goalInOverview}
  \Lambda_{p}(\nu) \ge \frac{e(\nu)^2}{R}
\end{equation}
which, recall, is the definition of what it means to not be $(p, R)$-Janson.
We will consider two cases, depending on the relation between $\nu$ and $\nu'$, the restriction of $\nu$ to $\cal{C}'_T[X]$, where $T = \varphi(I)$.
The first (easier) case is when $e(\nu') \ge e(\nu)/2$.
Here, we can easily show that \eqref{eq:goalInOverview} follows from $\cal{C}'_T[X]$ not being $(p, 2^{-8} p|X|)$-Janson (see \Cref{stmt:nuPrimeHasLessThanHalfTheMeasureOfNu}).

In the other, more delicate case, we will have that $\nu'' = \nu - \nu'$ satisfies
\begin{equation}\label{eq:nuDoublePrimeMass}
  e(\nu'') \ge \frac{e(\nu)}{2}.
\end{equation}
Our goal is now to obtain bounds relating $\Lambda_p(\nu)$ and $e(\nu'')^2$ -- it will be easy to see that combining them with \eqref{eq:nuDoublePrimeMass} will reach \eqref{eq:goalInOverview}.
Concretely, we will show that
\begin{equation}\label{eq:chainOfInequalitiesInOverview}
  2^{2 s}\Lambda_p(\nu) > \EE\big[\Lambda_{p/q}(\nu''_q) \mid V_q \in \cal{I}(\link{\cal{J}}{T})\big] \ge \frac{\EE\big[e(\nu''_q)^2 \mid V_q \in \cal{I}(\link{\cal{J}}{T})\big]}{\eta R} \ge \frac{e(\nu'')^2}{\eta R}
\end{equation}
where $\nu''_q : \cal{H}[X] \to \RR_{\ge 0}$ is defined by
\begin{equation*}
  \nu''_q(E) =  \frac{ \mathds{1}[E \subset V_q] }{\PP\big(E \subset V_q \mid V_q \in \cal{I}(\link{\cal{J}}{T})\big)} \, \nu''(E).
\end{equation*}

The simple proof of \Cref{stmt:nuDoublePrimeQLowerBound} will establish, by inspecting the definitions, that
\begin{equation*}
  \EE\big[e(\nu''_q) \mid V_q \in \cal{I}(\link{\cal{J}}{T})\big] = e(\nu''),
\end{equation*}
which implies the rightmost inequality of \eqref{eq:chainOfInequalitiesInOverview} by convexity.
To prove the second inequality, we will in fact show that (deterministically) whenever $V_q \in \cal{I}(\link{\cal{J}}{T})$, we have
\begin{equation}\label{eq:assumeVqAndConcludeIneq}
  \Lambda_{p/q}(\nu''_q) \ge \frac{e(\nu''_q)^2}{\eta R}.
\end{equation}
Proving \eqref{eq:assumeVqAndConcludeIneq} will require a trivial observation about independent sets in the non-strict link of hypergraphs with respect to a set $T$.

\begin{obs}\label{stmt:independentInLinkImpliesIndependentInOriginal}
  Let $\cal{G}$ be a hypergraph and $T \subset V(\cal{G})$.
  If $I \in \cal{I}(\link{\cal{G}}{T})$, then $I \in \cal{I}(\cal{G})$.
\end{obs}

\begin{proof}
  For each $E \in \cal{G}$, observe that  $E \setminus T \subset E$ and $E \setminus T \in \link{\cal{G}}{T}$ by definition, so $\cal{G} \subset \langle \link{\cal{G}}{T} \rangle$.
  The statement now follows from \Cref{stmt:independentSetsWithUpsets}.
\end{proof}

With \Cref{stmt:independentInLinkImpliesIndependentInOriginal}, we will see that the definition of $\cal{J}$ and the monotonicity of being $(p/q, \eta R)$-Janson will imply \eqref{eq:assumeVqAndConcludeIneq}, see \Cref{stmt:lambdaPropertiesWithConditionalExpectation}.
The remaining inequality, \Cref{stmt:finalBoundLambdaPQNuDoublePrimeQ}, establishes that
\begin{equation}
  2^{2 s}\Lambda_p(\nu) > \EE\big[\Lambda_{p/q}(\nu''_q) \mid V_q \in \cal{I}(\link{\cal{J}}{T})\big].
\end{equation}
To prove it, we will crucially rely on the fact that, for all $E \in \cal{H}[X] \setminus \cal{C}'_T$,
\begin{equation*}
  \PP\big(E \subset V_q \mid V_q \in \cal{I}(\link{\cal{J}}{T})\big) > \left(\frac{\,q\,}{2}\right)^{|E|}
\end{equation*}
since $\cal{H}[X] \setminus \cal{C}'_T$ and $\cal{C}_T$ are disjoint, and $\cal{C}_T$ satisfies \cref{item:containersHardcoversAreCoverable} of \Cref{stmt:containersHardcovers} with $\alpha=1/2$.

The preceding overview and proof strategy in fact proves \Cref{stmt:containersForNonJansonWithFingerprints}, which strengthens the original statement and adds the characterization of sets $L$ for which $\cal{H}$ is not $(p/q, \eta R)$-Janson as independent sets in the auxiliary hypergraph $\cal{J}$.
To obtain this stronger statement, we will redefine $\cal{J}$ in \eqref{eq:containersForNonJansonWithFingerprints:definitionOfJansonHypergraph} to consider instead sets $L$ for which $\cal{H}[L]$ is $(p/(q - p), \eta R)$-Janson.
We will then deduce \Cref{stmt:containersForNonJansonGeneral} from \Cref{stmt:containersForNonJansonWithFingerprints} by applying the latter with $q$ being $q + p$.

\begin{thm}\label{stmt:containersForNonJansonWithFingerprints}
  Let $s, n \in \NN$ with $s \le n$, and let $p, q, \alpha, R, \eta > 0$ satisfy
  \begin{equation}\label{eq:containersForNonJansonWithFingerprints:assumptions}
    p \le \frac{1}{2^{11}s^2}, \qquad 2p \le q \le \alpha < 1, \qquad R \ge 2^{-6} p n \qquad \text{ and } \qquad \eta \le \dfrac{(1-\alpha)^{2s}}{4}.
  \end{equation}
  Furthermore, for every $s$-uniform hypergraph $\cal{H}$ with $n$ vertices, let
  \begin{equation}\label{eq:containersForNonJansonWithFingerprints:definitionOfJansonHypergraph}
    \cal{J} = \big\{ L \subset V(\cal{H}) : \cal{H}[L] \text{ is $(p/(q-p), \eta R)$-Janson} \big\}.
  \end{equation}
  There exists a family $\cal{Y} \subset 2^{V(\cal{H})} \times 2^{V(\cal{H})}$ and functions
  \begin{equation*}
    g : \cal{I}(\cal{J}) \to \cal{Y} \qquad \text{and} \qquad f : \cal{Y} \to 2^{V(\cal{H})}
  \end{equation*}
  such that:
  \begin{enumerate}[\normalfont (1)]
    \item For every $I \in \cal{I}\big(\cal{J}\big)$, if $g(I) = (S, T)$, then $S \cup T \subset I \subset f(S, T)$. \label{item:containersForNonJansonWithFingerprints:containerAndFingerprintInclusions}
    \item \label{item:containersForNonJansonWithFingerprints:sizeOfFingerprints} Each $(S, T) \in \cal{Y}$ satisfies
      \begin{equation}\label{eq:containersForNonJansonWithFingerprints:sizeOfFingerprints}
        |S| \le 2^{11} p s^2 n \qquad \text{ and } \qquad |T| \le q n/\alpha.
      \end{equation}
    \item For all $Y \in \cal{Y}$, the hypergraph $\cal{H}[X]$ is not $(p, R)$-Janson, where $X = f(Y)$. \label{item:containersForNonJansonWithFingerprints:containerHypergraphIsNotJanson}
  \end{enumerate}
\end{thm}

Before proving \Cref{stmt:containersForNonJansonWithFingerprints}, let us quickly observe that it implies \Cref{stmt:containersForNonJansonGeneral}.

\begin{proof}[Proof that \Cref{stmt:containersForNonJansonWithFingerprints} implies \Cref{stmt:containersForNonJansonGeneral}]
  Apply \Cref{stmt:containersForNonJansonWithFingerprints} to $\cal{H}$ with $\alpha = 1/2$ and $q$ replaced by $q + p$.
  Note that $$p \le \frac{q}{2^{10} s^2} \le \frac{1}{ 2^{11} s^2 } \qquad \text{and} \qquad 2p \le q + p \le \alpha,$$ since $q \le 1/16$.
  As a result, we obtain functions $g, f$ and a family $\cal{Y}$, so let
  \begin{equation*}
    \cal{X} = \{f(S, T) : (S, T) \in \cal{Y}\}.
  \end{equation*}
  Observe that
  \begin{equation}\label{eq:sizeOfContainerReferenceForTheFirstTime}
    |\cal{X}| \le \sum_{m = 0}^{2^{11} p s^2 n} \binom{n}{m} \sum_{t = 0}^{2 (q + p) n} \binom{n}{t} \le \bigg(\frac{\,2\,}{q}\bigg)^{8 q n}
  \end{equation}
  by enumerating every possible $S$ and $T$, and combining the bound on their sizes, \eqref{eq:containersForNonJansonWithFingerprints:sizeOfFingerprints}, with
  \begin{equation}\label{eq:boundOnTheSizesOfFingerprintsForBinomial}
    2^{11}p s^{2}n \le  2(q + p)n \le 4qn \le \frac{n}{4}.
  \end{equation}
  This choice of $\cal{X}$ therefore satisfies the bound on the size of the family in \Cref{stmt:containersForNonJansonGeneral}.

  Take an arbitrary $L \subset V(\cal{H})$ such that $\cal{H}[L]$ is not $(p/q, \eta R)$-Janson, and note that $L \not \in E(\cal{J})$, by our choice of $q$ as $q + p$, and hence $L \in \cal{I}(\cal{J})$, by \Cref{stmt:JansonIsDownset}.
  Applying \cref{item:containersForNonJansonWithFingerprints:containerAndFingerprintInclusions} of \Cref{stmt:containersForNonJansonWithFingerprints} implies that there is $(S, T) \in \cal{Y}$ such that $L \subset f(S, T)$, which establishes \namecref{item:containersForNonJansonGeneral:containerInclusion}~\eqref{item:containersForNonJansonGeneral:containerInclusion} of \Cref{stmt:containersForNonJansonGeneral}.
  \nameCref{item:containersFornonJansonGeneral:containersAreNotJanson}~\eqref{item:containersFornonJansonGeneral:containersAreNotJanson} of \Cref{stmt:containersForNonJansonGeneral} follows from our choice of $\cal{X}$ and \cref{item:containersForNonJansonWithFingerprints:containerHypergraphIsNotJanson} of \Cref{stmt:containersForNonJansonWithFingerprints}.
\end{proof}

We now proceed to prove \Cref{stmt:containersForNonJansonWithFingerprints}.

\begin{proof}[Proof of \Cref{stmt:containersForNonJansonWithFingerprints}]
  Apply \Cref{stmt:containersHardcovers} with $\cal{G} = \cal{J}$ and parameters $q$ and $\alpha$ to obtain $\cal{T}$ and $\varphi$.
  Now, for each $T \in \cal{T}$, there is $\cal{C}_T$ satisfying \cref{item:containersHardcoversAreCoverable} in \Cref{stmt:containersHardcovers}, so we let $\cal{C}'_T = \langle \cal{C}_T \rangle_{= s}$ be the edges of $\langle \cal{C}_T \rangle$ with size $s$.
  As $\cal{C}'_T$ is $s$-uniform and $p \le 1/(2^{11} s^2)$ by \eqref{eq:containersForNonJansonWithFingerprints:assumptions}, we can apply \Cref{stmt:containersCoversButJanson} with $\cal{G} = \cal{C}'_T$ and $\zeta = 2^{-8}$, obtaining as a result $\cal{S}_T$, $\psi_T$ and $\phi_T$.

  Fix an $I \in \cal{I}(\cal{J})$ and note that if $T = \varphi(I)$, then it follows from the ``moreover'' part in \cref{item:containersHardcoversAreCoverable} of \Cref{stmt:containersHardcovers} that $I \in \cal{I}(\cal{C}_T)$.
  Combining this with \Cref{stmt:independentSetsWithUpsets} and $\cal{C}'_T \subset \langle \cal{C}_T \rangle$, we conclude that $I \in \cal{I}(\cal{C}'_T)$.
  As we have applied \Cref{stmt:containersCoversButJanson} with $\cal{G} = \cal{C}'_T$ for every $T \in \cal{T}$, and since $I \in \cal{I}(\cal{C}'_T)$, we obtain, by \cref{item:containersCoversHaveFingerprints}, sets $S = \phi_T(I) \in \cal{S}_T$ and $X = \psi_T(S)$, where $T = \varphi(I)$, such that
  \begin{equation}\label{eq:containerAndFingerprintInclusions}
    S \cup T \subset I \subset X.
  \end{equation}
  We then define $g(I) = (S, T)$ and $f(S, T) = X$ for $T = \varphi(I)$ and $S = \phi_T(I)$ and set
  \begin{equation*}
    \cal{Y} = \{g(I) : I \in \cal{I}(\cal{J})\},
  \end{equation*}
  which, by \eqref{eq:containerAndFingerprintInclusions} and the fact that $I$ was arbitrary, is a definition that satisfies \cref{item:containersForNonJansonWithFingerprints:containerAndFingerprintInclusions} of \Cref{stmt:containersForNonJansonWithFingerprints}.
  Moreover, each $T \in \cal{T}$ and $S \in \cal{S}_T$ satisfy
  \begin{equation}\label{eq:sizeOfFingerprintsInGeneralContainerJanson}
    |T| \le q n / \alpha  \qquad \text{and} \qquad |S| \le 2^{11} s^2 p n
  \end{equation}
  by \cref{item:containersHardcoversHaveSmallFingerprints} in \Cref{stmt:containersHardcovers} and \cref{item:containersCoversHaveSmallFingerprints} in \Cref{stmt:containersCoversButJanson} with our choice of $\zeta = 2^{-8}$, which proves \cref{item:containersForNonJansonWithFingerprints:sizeOfFingerprints}.
  It remains only to show that \cref{item:containersForNonJansonWithFingerprints:containerHypergraphIsNotJanson} holds.

  Take an arbitrary $(S, T) \in \cal{Y}$ with $X = f(S, T)$ with the goal of showing that $\cal{H}[X]$ is not $(p, R)$-Janson.
  To do so, it suffices to show that, for any $\nu : \cal{H}[X] \to \RR_{\ge 0}$, we have
  \begin{equation}\label{eq:goal}
    \Lambda_p(\nu) \ge \frac{e(\nu)^2}{R},
  \end{equation}
  so we fix a measure $\nu : \cal{H}[X] \to \RR_{\ge 0}$ and want to establish \eqref{eq:goal}.

  Recall that we have defined $\cal{C}'_T = \langle \cal{C}_T \rangle_{= s}$, where $\cal{C}_T$ is the hypergraph given by \cref{item:containersHardcoversAreCoverable} in \Cref{stmt:containersHardcovers} for $T \in \cal{T}$.
  Let $\nu'$ be the restriction of $\nu$ to $\cal{H}[X] \cap \cal{C}'_T[X]$, i.e. for each $E \in \cal{H}[X]$, let
  \begin{equation*}
    \nu'(E) = \begin{cases*}
      \nu(E) \quad \text{if } E \in \cal{C}'_T[X], \\
      0 \phantom{(E)} \quad \text{otherwise.}
    \end{cases*}
  \end{equation*}

  \begin{claim}\label{stmt:nuPrimeHasLessThanHalfTheMeasureOfNu}
    If the measure $\nu'$ satisfies
    \begin{equation}\label{eq:converseOfWhatIsTrue}
      e(\nu') \ge \frac{e(\nu)}{2},
    \end{equation}
    then \eqref{eq:goal} holds.
  \end{claim}

  \begin{claimproof}
    Assume that \eqref{eq:converseOfWhatIsTrue} holds.
    We have
    \begin{equation*}
      \Lambda_p(\nu) \ge \Lambda_{p}(\nu') \ge \frac{2^8 e(\nu')^2}{p |X|} \ge \frac{4 e(\nu')^2}{R}  \ge \frac{e(\nu)^2}{R},
    \end{equation*}
    first because $\nu' \le \nu$ and $\Lambda_p(\cdot)$ is monotone increasing, second since $\cal{C}'_T[X]$ is not $(p, 2^{-8} p |X|)$-Janson by \cref{item:containersAreNonJanson} of \Cref{stmt:containersCoversButJanson} and our choice of $\zeta = 2^{-8}$, then because $R \ge 2^{-6} p n$ by \eqref{eq:containersForNonJansonWithFingerprints:assumptions}, and the last step is due to \eqref{eq:converseOfWhatIsTrue}.
  \end{claimproof}

  We now define the measure $\nu'' = \nu - \nu'$, which corresponds to the restriction of $\nu$ to the hypergraph ${\cal{H}' := \cal{H}[X] \setminus \cal{C}'_T} = \cal{H}[X] \setminus \cal{C}'_T[X]$.
  By \Cref{stmt:nuPrimeHasLessThanHalfTheMeasureOfNu}, we may assume that
  \begin{equation}\label{eq:nuDoublePrimeHasHalfTheMeasureOfNu}
    e(\nu'') = e(\nu) - e(\nu') > \frac{e(\nu)}{2},
  \end{equation}
  otherwise we are done.

  With the goal of defining a random measure $\nu''_q$ on the hypergraph induced by the random set $X_q = V_q \cap X$, where $V_q$ is a $q$-random subset of $V$, we introduce some notation.
  First, let
  \begin{equation}\label{eq:definitionOfP}
    P_q(E) = \PP\big(E \subset V_q \mid V_q \in \cal{I}(\link{\cal{J}}{T})\big)
  \end{equation}
  and observe that
  \begin{equation}\label{eq:lowerBoundOnPOfE}
    P_q(E) > (1-\alpha)^{|E|}q^{|E|}
  \end{equation}
  for all $E \in \cal{H}'$, by \eqref{eq:conditionalProbOfLInQSet}, since the hypergraphs $\cal{H}'$ and $\cal{C}_T$ are disjoint.
  Indeed, $\cal{H}$ is $s$-uniform, which means that so are $\cal{H}'$ and $\cal{H}' \cap \cal{C}_T$.
  The only edges of $\cal{C}_T$ that could be edges of $\cal{H}'$ thus have size exactly equal to $s$.
  But every $s$-sized edge of $\cal{C}_T$ is also an edge of $\cal{C}'_T = \langle \cal{C}_T \rangle_{= s}$, and is therefore not in $\cal{H}' = \cal{H}[X] \setminus \cal{C}'_T$.

  Now, let $\nu_q'' : \cal{H}' \to \RR_{\ge 0}$ be defined by
  \begin{equation*}
    \nu''_q(E) = \frac{\mathds{1}\big[E \subset V_q\big] }{\PP\big(E \subset V_q \mid V_q \in \cal{I}(\link{\cal{J}}{T})\big)} \, \nu''(E).
  \end{equation*}
  We will first show that
  \begin{equation*}
    \EE\big[e(\nu''_q) \mid V_q \in \cal{I}(\link{\cal{J}}{T})\big] = \sum_{E \in \cal{H}'} \nu''(E) = e(\nu''),
  \end{equation*}
  which will allows us to relate $e(\nu_q'')^2$ to $e(\nu)^2$.

  \begin{claim}\label{stmt:nuDoublePrimeQLowerBound}
    \begin{equation*}
      \EE\big[e(\nu''_q)^2 \mid V_q \in \cal{I}(\link{\cal{J}}{T})\big] \ge e(\nu'')^2.
    \end{equation*}
  \end{claim}

  \begin{claimproof}
    The definitions of $e(\nu''_q)$ and $\nu''_q$,
    \begin{equation*}
      e(\nu''_q) = \sum_{E \in \cal{H}'} \nu''_q(E) = \sum_{E \in \cal{H}'} \frac{\nu''(E) \mathds{1}\big[E \subset V_q\big]}{P_q(E)},
    \end{equation*}
    imply that
    \begin{equation}\label{eq:stepBeforeJensen}
      \EE\big[e(\nu''_q) \mid V_q \in \cal{I}(\link{\cal{J}}{T})\big] = \sum_{E \in \cal{H}'} \nu''(E) =  e(\nu''),
    \end{equation}
    since, for all $E \in \cal{H}'$, we have that
    \begin{equation*}
      \EE\big[\mathds{1}[E \subset V_q] \mid V_q \in \cal{I}(\link{\cal{J}}{T})\big] = P_q(E)
    \end{equation*}
    from \eqref{eq:definitionOfP}, the definition of $P_q(E)$.
    The \namecref{stmt:nuDoublePrimeQLowerBound} now follows from \eqref{eq:stepBeforeJensen} by Jensen's inequality.
  \end{claimproof}

  The next step is relating $\Lambda_{p/(q-p)}(\nu''_q)$ and $e(\nu''_q)$ when $V_q \in \cal{I}(\link{\cal{J}}{T})$.

  \begin{claim}\label{stmt:lambdaPropertiesWithConditionalExpectation}
    If $V_q \in \cal{I}(\link{\cal{J}}{T})$, then
    \begin{equation*}
      \Lambda_{p/(q-p)}(\nu''_q) \ge \frac{e(\nu''_q)^2}{\eta R}.
    \end{equation*}
  \end{claim}

  \begin{claimproof}
    It follows from $V_q \in \cal{I}(\link{\cal{J}}{T})$ and \Cref{stmt:independentInLinkImpliesIndependentInOriginal} that $V_q \in \cal{I}(\cal{J})$, and hence $\cal{H}[V_q]$ is not $(p/(q-p), \eta R)$-Janson, by the definition of $\cal{J}$, \eqref{eq:containersForNonJansonWithFingerprints:definitionOfJansonHypergraph}.
    In particular, it follows that
    \begin{equation*}
      \Lambda_{p/(q-p)}(\nu''_q) \ge \frac{e(\nu''_q)^2}{\eta R}
    \end{equation*}
    as $\nu_q''$ is also a measure supported on $\cal{H}[V_q]$ by $\cal{H}' \subset \cal{H}$.
  \end{claimproof}

  Our final inequality bounds $\Lambda_{p/(q-p)}(\nu''_q)$ in expectation by $\Lambda_p(\nu)$, up to an exponential factor.

  \begin{claim}\label{stmt:finalBoundLambdaPQNuDoublePrimeQ}
    \begin{equation*}
      \EE\big[\Lambda_{p/(q-p)}(\nu''_q) \mid V_q \in \cal{I}(\link{\cal{J}}{T})\big] < (1-\alpha)^{-2 s} \Lambda_p(\nu).
    \end{equation*}
  \end{claim}

  \begin{claimproof}
    Recall that $d_{\nu''_q}(L)$ is defined as
    \begin{equation*}
      d_{\nu''_q}(L) = \sum_{L \subset E \in \cal{H}'} \nu_q''(E) = \sum_{L \subset E \in \cal{H}'} \frac{\nu''(E) \mathds{1}\big[E \subset V_q\big]}{P_q(E)},
    \end{equation*}
    where the last equality is using the definition of $\nu''_q$, and hence we can write, for every $L \subset V$,
    \begin{equation*}
      d_{\nu''_q}(L)^2 = \sum_{L \subset E_1 \in \cal{H}'} \frac{\nu''(E_1)}{P_q(E_1)} \sum_{L \subset E_2 \in \cal{H}'} \frac{\nu''(E_2)}{P_q(E_2)} \, \mathds{1}\big[E_1 \cup E_2 \subset V_q\big].
    \end{equation*}
    Observe that the event $V_q \in \cal{I}(\link{\cal{J}}{T})$ is decreasing and also that, for every $E_1, E_2 \in \cal{H}'$, the event $E_1 \cup E_2 \subset V_q$ is increasing.
    We can therefore use Harris' inequality to bound, for $E_1, E_2 \in \cal{H}'$,
    \begin{equation}\label{eq:applicationOfHarris}
      \PP\big(E_1 \cup E_2 \subset V_q \mid V_q \in \cal{I}(\link{\cal{J}}{T})\big) \le \PP(E_1 \cup E_2 \subset V_q) = q^{|E_1 \cup E_2|} = q^{2 s - |E_1 \cap E_2|}
    \end{equation}
    because $\cal{H}'$ is $s$-uniform.
    Taking the conditional expectation and applying \eqref{eq:applicationOfHarris}, we obtain
    \begin{equation}\label{eq:finalBoundDegreeNuDoublePrimeQSquared}
      \EE\big[d_{\nu''_q}(L)^2 \mid V_q \subset \cal{I}(\link{\cal{J}}{T})\big] < (1 - \alpha)^{-2 s} \sum_{L \subset E_1 \in \cal{H}'} \nu(E_1) \sum_{L \subset E_2 \in \cal{H}'} \nu(E_2) \, q^{-|E_1 \cap E_2|}
    \end{equation}
    where we used $P_q(E) > (1-\alpha)^s q^s$, by \eqref{eq:lowerBoundOnPOfE} and since $\cal{H}'$ is $s$-uniform, and the fact that $\nu'' \le \nu$.

    By the definition \eqref{eq:defEdgesAndLambda} of $\Lambda_{p/(q-p)}(\nu''_q)$,
    \begin{equation}\label{eq:consequenceOfDefinitionOfLambda}
      \EE\big[\Lambda_{p/(q-p)}(\nu''_q) \mid V_q \in \cal{I}(\link{\cal{J}}{T})\big] = \sum_{\substack{L \subset V \\ |L| \ge 2}} \EE\big[d_{\nu''_q}(L)^2 \mid V_q \in \cal{I}(\link{\cal{J}}{T})\big] \bigg(\frac{q - p}{p}\bigg)^{|L|},
    \end{equation}
    and then applying \eqref{eq:finalBoundDegreeNuDoublePrimeQSquared} to each $d_{\nu''_q}(L)$ term in \eqref{eq:consequenceOfDefinitionOfLambda} yields
    \begin{equation*}
      \EE\big[\Lambda_{p/(q-p)}(\nu''_q) \mid V_q \in \cal{I}(\link{\cal{J}}{T})\big] < (1-\alpha)^{-2 s} \sum_{\substack{L \subset V \\ |L| \ge 2}} \sum_{L \subset E_1 \in \cal{H}'} \sum_{L \subset E_2 \in \cal{H}'} \frac{\nu(E_1) \nu(E_2)}{q^{|E_1 \cap E_2|}} \bigg(\frac{q-p}{p}\bigg)^{|L|}
    \end{equation*}
    or, equivalently,
    \begin{equation}\label{eq:readyForBinomialTheorem}
      \EE\big[\Lambda_{p/(q-p)}(\nu''_q) \mid V_q \in \cal{I}(\link{\cal{J}}{T})\big] < (1-\alpha)^{-2 s} \sum_{\substack{E_1, E_2 \in \cal{H}' \\ |E_1 \cap E_2| \ge 2}} \frac{ \nu(E_1) \nu(E_2) }{q^{|E_1 \cap E_2|}} \sum_{\ell = 2}^s \bigg(\frac{q-p}{p}\bigg)^{\ell} \binom{|E_1 \cap E_2|}{\ell}
    \end{equation}
    by first choosing $E_1, E_2 \in \cal{H}'$ and then $L \subset E_1 \cap E_2$, grouping terms according to $\ell = |L|$.
    Bounding the innermost sum in \eqref{eq:readyForBinomialTheorem} for fixed $E_1, E_2 \in \cal{H}'$ with $|E_1 \cap E_2| \ge 2$ then yields
    \begin{equation*}
      \sum_{\ell = 2}^s \bigg(\frac{q-p}{p}\bigg)^{\ell} \binom{|E_1 \cap E_2|}{\ell} \le \bigg(\frac{\, q \,}{p}\bigg)^{|E_1 \cap E_2|},
    \end{equation*}
    which replaced in \eqref{eq:readyForBinomialTheorem} and simplified, results in
    \begin{equation}\label{eq:readyForFinalSubstitution}
      \EE\big[\Lambda_{p/(q-p)}(\nu''_q) \mid V_q \in \cal{I}(\link{\cal{J}}{T})\big] < (1 - \alpha)^{-2 s} \sum_{\substack{E_1, E_2 \in \cal{H}' \\ |E_1 \cap E_2| \ge 2}} \frac{\nu(E_1) \nu(E_2)}{p^{|E_1 \cap E_2|}}.
    \end{equation}

    To complete the proof, note that
    \begin{equation*}
      \sum_{\substack{E_1, E_2 \in \cal{H}' \\ |E_1 \cap E_2| \ge 2}} \frac{\nu(E_1) \nu(E_2)}{p^{|E_1 \cap E_2|}} \le \sum_{\substack{L \subset V \\ |L| \ge 2}} \, \sum_{L \subset E_1 \in \cal{H}}  \sum_{L \subset E_2 \in \cal{H}} \frac{\nu(E_1) \nu(E_2)}{p^{|L|}} = \sum_{\substack{L \subset V \\ |L| \ge 2}} d_\nu(L)^2 p^{-|L|},
    \end{equation*}
    where the last term is equal to $\Lambda_p(\nu)$ by definition, so we obtain, replacing it back in \eqref{eq:readyForFinalSubstitution}, the inequality that we wanted.
  \end{claimproof}

  Observe that \Cref{stmt:lambdaPropertiesWithConditionalExpectation} implies that
  \begin{equation}\label{eq:lambdaPropertiesWithConditionalExpectationInProof}
    \EE\big[\Lambda_{p/(q-p)}(\nu''_q) \mid V_q \in \cal{I}(\link{\cal{J}}{T})\big] \ge \frac{\EE\big[e(\nu''_q)^2 \mid V_q \in \cal{I}(\link{\cal{J}}{T})\big]}{\eta R}.
  \end{equation}
  Combining \Cref{stmt:finalBoundLambdaPQNuDoublePrimeQ} and \Cref{stmt:nuDoublePrimeQLowerBound} with \eqref{eq:lambdaPropertiesWithConditionalExpectationInProof} yields
  \begin{equation*}
    (1-\alpha)^{-2 s} \Lambda_p(\nu) \ge \frac{e(\nu'')^2}{\eta R}
  \end{equation*}
  and thus, since we are in the case where \eqref{eq:nuDoublePrimeHasHalfTheMeasureOfNu} holds, it follows that
  \begin{equation}\label{eq:goalAccomplished}
    \Lambda_p(\nu) \ge \frac{e(\nu)^2}{R}
  \end{equation}
  by our choice of $\eta$ satisfying $4 \eta \le (1-\alpha)^{2 s}$.
  As \eqref{eq:goalAccomplished} was exactly our goal, \eqref{eq:goal}, and $\nu$ was arbitrary, we conclude that $\cal{H}[X]$ is not $(p, R)$-Janson.
  Moreover, our choice of $(S, T) \in \cal{Y}$ was also arbitrary, so we have established that \cref{item:containersForNonJansonWithFingerprints:containerHypergraphIsNotJanson} holds, and the proof is complete.
\end{proof}

\section{Extending collections of copies}\label{sec:extensionOfJansonCollection}

In this \namecref{sec:extensionOfJansonCollection}, we use a novel container \namecref{stmt:containersForNonJanson} to prove the core statement that we need in the proof of \Cref{stmt:key}.
We defer the proof of this container theorem to \Cref{sec:containersForNonJanson}, since that is the most technical part of the entire argument.

The setting is very similar to \Cref{stmt:warmUpExtension}, but now we will be able to extend many copies of $F^{-}$ to $F$ by adding a single vertex $v$ to $U = V(\tilde{G}) = V(\tilde{G}')$.
Moreover, we will be able to show that the set of copies created by adding $v$ to $U$ will be well-distributed in relation to the copies of $F$ fully contained in $U$.
To obtain this stronger conclusion, we assume that, besides $\frakI_{F^{-}, \tilde{G}', \tilde{G}}[W]$ being $(p, R)$-Janson for every $W \subset U = V(\tilde{G})$ with $|W| \ge m / (8r)$, we also have that $\frakI_{F,G',G}[U]$ is $(p,R')$-Janson.
Under these circumstances, the \namecref{stmt:extensionOfJansonCollection} states that, when $G$ is distributed as $\Gnp(m + 1, 2)$ conditioned on $\{G[U] = \tilde{G}\}$, the following holds with extremely high probability: for every choice of $G' \subset G$ such that $\mathrm{N}_{G'}(v)$ is not too small, the hypergraph $\frakI_{F,G',G}$ is $(p,R'+1)$-Janson.

\begin{lem}\label{stmt:extensionOfJansonCollection}
  Let $m, k, r, s \in \NN$ with $r \ge 2$, $s < k \le m$, and let
  \begin{equation}\label{eq:paramsInExtensionOfJansonCollection}
    p = \frac{1}{2^{25}k^2r^4}, \qquad m \ge r^{Ck}, \qquad R = 2^{-5} r^{-1} p m \qquad \text{and} \qquad 0 \le R' \le \frac{R}{16}.
  \end{equation}
  Further let $F$, $\tilde{G}'$ and $\tilde{G}$ be graphs such that $v(F) = s + 1 < m$, $\tilde{G}' \subset \tilde{G}$ and $v(\tilde{G}) = m$.

  If $\frakI_{F, \tilde{G}', \tilde{G}}$ is $(p, R')$-Janson and $\frakI_{F^{-}, \tilde{G}', \tilde{G}}[W]$ is $(p, R)$-Janson for every $W \subset U = V(\tilde{G})$ with $|W| \ge m / (8r)$, then
  \begin{equation}\label{eq:probOfGPrime}
    \PP
    \left(
        \exists \, G' \subset G
      \,
      :
      \,
      \begin{array}{@{}c@{}}
        G'[U] = \tilde{G}', ~ d_{G'}(v) \ge m / (4r) \text{ and} \\
        \frakI_{F, G', G}
        \text{ is not }  (p, R' + 1)\text{-Janson}
      \end{array}
      ~
      \middle \vert \ G[U] = \tilde{G}
    \right) \le 2^{- m /(32 r)},
  \end{equation}
  where $G \sim \Gnp(m + 1, 1/2)$ and $V(G) = U \cup \{v\}$.
\end{lem}

The first change that we need to make to the proof in \Cref{sec:warmUpExtension} is to replace $\Gamma_G$, defined in \eqref{eq:eventOfGPrimeWarm}, by $\Psi_G$, which is just the collection of $G' \subset G$ satisfying the event in \eqref{eq:probOfGPrime}:
\begin{equation}\label{eq:eventOfGPrime}
  \Psi_G = \left\{
    G' \subset G :
  \begin{array}{@{}c@{}}
    G'[U] = \tilde{G}', ~ d_{G'}(v) \ge m / (4r) \text{ and } \\
     \frakI_{F, G', G} \text{ is not }  (p, R'+1)\text{-Janson}
  \end{array}
  \right\}.
\end{equation}
That is, we now require that $\frakI_{F, G', G}$ is not $(p, R'+1)$-Janson, instead of requiring to it be empty as in \Cref{stmt:warmUpExtension}.
The argument follows very closely the one in \Cref{sec:warmUpExtension}, so we briefly summarise the ideas in the proof of \Cref{stmt:warmUpExtension}, referring to some of the definitions in that section as we progress.

Recall that we defined
\begin{equation*}
  \cal{H} = \Big\{E_L \subset U \times \{0, 1\} : L \in \frakI_{F^-, \tilde{G}', \tilde{G}}\Big\},
\end{equation*}
where
\begin{equation*}
  E_L = \bigg\{\Big(u, \mathds{1}\big[\phi_L(u) \in \mathrm{N}_F(w)\big]\Big) : u \in L\bigg\},
\end{equation*}
with $\phi_L : L \to V(F^-)$ being a fixed bijection and $w$ being the unique vertex in $V(F) \setminus V(F^-)$.
We then defined
\begin{equation*}
  \iota(G', G) = \big\{(u,0) : u  \in \mathrm{N}_G(v)^{\complement}\big\} \cup \big\{(u,1) : u \in \mathrm{N}_{G'}(v)\big\}
\end{equation*}
in \eqref{eq:definitionIotaWarm}, proved \Cref{stmt:iotaIndepWarm}, which states that if $G[U] = \tilde{G}$ and $G' \in \Gamma_G$, then $\iota(G', G) \in \cal{I}(\cal{H})$, and applied a container \namecref{stmt:containersCoversButJanson} to bound the probability of that event.
However, this is not immediately possible here, since \Cref{stmt:iotaIndepWarm} is not true if we replace $G' \in \Gamma_G$ by $G' \in \Psi_G$, the analogous collection for this \namecref{sec:extensionOfJansonCollection}: requiring $\frakI_{F, G', G}$ to not be $(p, R' + 1)$-Janson instead of $\frakI_{F, G', G} = \emptyset$ means that we are not interested in independent sets in $\cal{H}$, but in vertex subsets $I$ such that $\cal{H}[I]$ is not $(p, R' + 1)$-Janson.

We remedy that situation by relying on a container \namecref{stmt:containersForNonJanson} for such sets, like the one that we proved in \Cref{sec:simplerContainerForNonJanson}.
However, \Cref{stmt:containersForNonJansonGeneral} is not adequate for several reasons, which we discuss while introducing some notation and new definitions.
We then prove the analogue of \Cref{stmt:iotaIndepWarm} for this section, \Cref{stmt:motivationForJprime}, and state the container \namecref{stmt:containersForNonJanson} that we end up using, \Cref{stmt:containersForNonJanson}.

It will be helpful to partition the copies of $F \subset G'[U \cup \{v\}]$ in two natural classes.
The first one consists of the copies of $F$ that use $v$, which correspond to copies of $F^- \subset \tilde{G}'[U]$ that are extended with the addition of $v \not \in U$.
Every other copy of $F \subset G'$, i.e.\ those that do not use $v$, belong in the second class, and are contained in $\tilde{G}' = G'[U]$.
The next \namecref{def:edgeWiseInclusion} will relate the hypergraph of copies of $F^-$ that can be extended with $v$ and the hypergraph of the resulting copies of $F$.

\begin{defi}\label{def:edgeWiseInclusion}
  For a hypergraph $\cal{G}$ and a vertex $v$ not in $V(\cal{G})$, let
  \begin{equation*}
    \ilink{\cal{G}}{v} = \big\{E \cup \{v\} : E \in \cal{G}\big\}
  \end{equation*}
  denote the edge-wise inclusion of $v$ in $\cal{G}$.
\end{defi}

Like in \Cref{sec:warmUpExtension}, let $\pi : V(\cal{H}) \to U$ be the projection onto the first coordinate and define $\pi_v = \ilink{\circ \pi}{v}$.
Recalling \Cref{stmt:piCalHIsHypergraphOfInducedCopies}, that is, $$\pi(\cal{H}) = \frakI_{F^-, \tilde{G}', \tilde{G}},$$ we now relate $\frakI_{F, G', G}$ to both $\frakI_{F, \tilde{G}', \tilde{G}}$ and $\cal{H}$ when $G[U]=\tilde{G}$ and $G'[U]=\tilde{G}'$.
We will use this fact to conclude that, if $\frakI_{F, G', G}$ is not $(p, R' + 1)$-Janson, then neither is $\pi_v\big(\cal{H}[I]\big) \cup \frakI_{F, \tilde{G}', \tilde{G}}$ when $I = \iota(G',G)$.

\begin{obs}\label{stmt:motivationForJprime}
  For all graphs $G'$ and $G$ such that $G[U] = \tilde{G}$, $G' \subset G$ and $G'[U] = \tilde{G}'$, if $I = \iota(G',G)$, then
  \begin{equation}\label{eq:linkOfProjectionIsInducedCopiesWithV}
    \pi_v\big(\cal{H}[I]\big) \cup \frakI_{F, \tilde{G}', \tilde{G}} \subset \frakI_{F, G', G}.
  \end{equation}
\end{obs}

The inclusion in \Cref{stmt:motivationForJprime} is in fact an equality, but we will not use that fact, and therefore avoid giving its (trivial) proof for the sake of brevity.
As we will see, \Cref{stmt:motivationForJprime} follows easily from expanding the definitions, especially after we recall \eqref{eq:projectionOfIotaWarm}, that is, if $I=\iota(G',G)$, then
\begin{equation}\label{eq:projectionOfIota}
   I^{(0)} = \mathrm{N}_G(v)^\complement \qquad \text{and} \qquad I^{(1)} = \mathrm{N}_{G'}(v).
\end{equation}

\begin{proof}[Proof of \Cref{stmt:motivationForJprime}]
  It follows immediately from
  \begin{equation}\label{eq:assumptionInMotivationForJPrime}
    G'[U] = \tilde{G}' \qquad \text{and} \qquad G[U] = \tilde{G}
  \end{equation}
  that $\frakI_{F, \tilde{G}', \tilde{G}} \subset \frakI_{F, G', G}$, so it only remains to show that
  \begin{equation*}
    \pi_v\big(\cal{H}[I]\big) \subset \frakI_{F, G', G}.
  \end{equation*}

  Let $E \in \cal{H}[I]$ be of the form $E = E_L$ for $L \in \frakI_{F^-, \tilde{G}', \tilde{G}}$ and recall that $\pi(E) = L$, so our goal is to show that
  \begin{equation}\label{eq:finalStepInMotivationForJPrime}
    L \cup \{v\} \in \frakI_{F, G', G},
  \end{equation}
  where $\pi_v(E) = L\cup\{v\}$ by the definition of $\pi_v$.
  As $L \in \frakI_{F^-, \tilde{G}', \tilde{G}}$, \eqref{eq:finalStepInMotivationForJPrime} follows from \Cref{stmt:defOfCalHIsCorrect} using \eqref{eq:projectionOfIota}, \eqref{eq:assumptionInMotivationForJPrime} and the fact that $E \subset I$.
\end{proof}

In analogy to the proof in \Cref{sec:warmUpExtension}, by \Cref{stmt:motivationForJprime} and the fact that the Janson property is increasing, \Cref{stmt:JansonIsDownset}, it suffices to have a family of containers $\cal{X}$ with the following property.
For all $\iota(G', G) = I \subset V(\cal{H})$ such that $\pi_v(\cal{H}[I]) \cup \frakI_{F, G', G}$ is not $(p, R' + 1)$-Janson, there is $X \in \cal{X}$ with $I \subset X$.
This is the statement of \Cref{stmt:containersForNonJanson}, the container \namecref{stmt:containersForNonJanson} that we need to prove \Cref{stmt:extensionOfJansonCollection}.
We state that \namecref{stmt:containersForNonJanson} below, but defer its proof, an implementation of the methods discussed in \Cref{sec:simplerContainerForNonJanson} tailored to this specific setting, to \Cref{sec:containersForNonJanson}.

Continuing the comparison with \Cref{stmt:warmUpExtension}, ideally each $X \in \cal{X}$ would be such that $\pi(\cal{H}[X])$ is not $(p, R)$-Janson.
We are unable to prove such a statement, because the inclusion of $v$ by $\pi_v$ adds a constraint in the one-degrees of the vertices in $\cal{H}$.
To deal with this extra constraint, we (roughly) delete a small proportion of vertices to reduce the maximum degree.

Implementing this modification to our method yields something slightly weaker that nonetheless suffices: we show that if $X \in \cal{X}$ is sufficiently large, then there is $Y$ covering almost all of $X$ such that $\pi(\cal{H}[Y])$ is not $(p, R)$-Janson.
Our final note before the statement is that, despite applying \Cref{stmt:containersForNonJanson} with the function $\pi$ being a projection, as we previously defined it in this \namecref{sec:extensionOfJansonCollection}, we state the \namecref{stmt:containersForNonJanson} in a slightly more general setting.

\begin{restatable}{thm}{containersForNonJanson}\label{stmt:containersForNonJanson}
  Let $n, r, s \in \NN$ with $n \ge s$ and $r \ge 2$, and let $q, p, R, R', \eta \in \RR$ satisfy
  \begin{equation}\label{eq:assumptionsOnContainerParams}
    0 < q < \frac{1}{8}, \quad 0 < p \le \frac{q}{2^{11} r s^2}, \quad R = 2^{-6} r^{-1} p n, \quad 0 \le R' \le \frac{R}{16} \quad \text{and} \quad \eta = p^4 \Big(\frac{\,q\,}{2}\Big)^{4 s}.
  \end{equation}
  Further let $\cal{F}$ be a $(s + 1)$-uniform hypergraph with vertex set $U$ that is $(p, R')$-Janson, let $\cal{H}$ be an $s$-uniform hypergraph with vertex set $V$, where $|V| = n$, and let $\pi: V \to U$ satisfy
  \begin{equation}\label{eq:assumptionsPi}
    |\pi(L)| \ge \frac{|L|}{2} \quad \text{for every } L \subset V
    \qquad \text{and} \qquad
    |\pi(E)| = |E| \quad \text{for every } E \in \cal{H}.
  \end{equation}
  Finally, let $v$ be a vertex not in $U$.
  There exists a family $\cal{X} \subset 2^V$ with
  \begin{equation}\label{eq:containersForNonJansonSmallFamily}
    |\cal{X}| \le \bigg(\frac{\,2\,}{q}\bigg)^{2 q n}
  \end{equation}
  such that the following hold.
  \begin{enumerate}[\normalfont (1)]
    \item If $I \subset V$ and $\pi_v(\cal{H}[I]) \cup \cal{F}$ is not $(p, R' + \eta R)$-Janson, then $I \subset X$ for some $X \in \cal{X}$. \label{item:containersForNonJansonInclusions}
    \item For each $X \in \cal{X}$ with $|X| \ge n/(8r)$, there exists $Y \subset X$ with
      \begin{equation}
        |Y| \ge |X| - 2^{-8}r^{-1} n
      \end{equation}
      such that $\pi(\cal{H}[Y])$ is not $(p, R)$-Janson. \label{item:containersForNonJansonAreNonJanson}
  \end{enumerate}
\end{restatable}

We are now ready to prove \Cref{stmt:extensionOfJansonCollection}.

\begin{proof}[Proof of \Cref{stmt:extensionOfJansonCollection}]
  The first step is applying \Cref{stmt:containersForNonJanson} with $q = 2^{-15} r^{-2}$ and $\cal{F} = \frakI_{F, \tilde{G}', \tilde{G}}$, so we must check that these choices satisfy the assumptions of the \namecref{stmt:containersForNonJanson}.

  We assumed that $\cal{F} = \frakI_{F, \tilde{G}', \tilde{G}}$ is $(p, R')$-Janson and $(s + 1)$-uniform, and $\cal{H}$ being $s$-uniform follows from its definition and the fact that $v(F^-) = v(F) - 1$.
  To apply \Cref{stmt:containersForNonJanson} with this choice of $\cal{H}$, we implicitly set
  \[V = U \times \{0, 1\}, \qquad n = 2 m \qquad \text{ and } \qquad R = 2^{-5} r^{-1} p m = 2^{-6} r^{-1} p n\]
  and therefore the value of $R$ coincides in both statements, resulting in the condition $0 \le R' \le R/16$ also being satisfied by the identical assumption in \Cref{stmt:extensionOfJansonCollection}.
  Furthermore, our choices for the parameters $0 < q = 2^{-15}r^{-2} < 1/8$ and $p > 0$ satisfy
  \[p = \frac{1}{2^{25} k^2 r^4} < \frac{q}{2^{10} s^{2} r^{2}}\]
  because $s < k$, and we have checked that all the conditions in \eqref{eq:assumptionsOnContainerParams} hold.

  We now check that $\pi$ satisfies \eqref{eq:assumptionsPi}.
  The first assumption follows trivially from $V = U \times \{0, 1\}$ and $\pi : V \to U$ being a projection into the first coordinate, while the other requirement is \Cref{stmt:piSatisfiesAssumption}.
  This concludes the checking of the assumptions and requirements in \Cref{stmt:containersForNonJanson}.

  Applying \Cref{stmt:containersForNonJanson}, we obtain a family $\cal{X}$ satisfying \eqref{eq:containersForNonJansonSmallFamily} and \cref{item:containersForNonJansonInclusions,item:containersForNonJansonAreNonJanson} in its statement.
  The next claim, a simple combination of \Cref{stmt:motivationForJprime}, the definition of $\Psi_G$ and the choice of $q$, shows that \cref{item:containersForNonJansonInclusions} holds for $I = \iota(G', G)$ when $G' \in \Psi_G$, where, recall,
  \begin{equation*}
    \iota(G', G) = \big\{(u,0) : u  \in \mathrm{N}_G(v)^{\complement}\big\} \cup \big\{(u,1) : u \in \mathrm{N}_{G'}(v)\big\}.
  \end{equation*}

  \begin{claim}\label{stmt:containersAreInFactContainers}
    Let $G'$ and $G$ be graphs and let $I = \iota(G', G)$.
    If $G[U] = \tilde{G}$ and $G' \in \Psi_G$, then the hypergraph $\pi_v\big(\cal{H}[I]\big) \cup \cal{F}$ is not $(p, R' + \eta R)$-Janson.
  \end{claim}

  \begin{claimproof}
    Recall that the definition of $\Psi_G$, \eqref{eq:eventOfGPrime}, implies that $G'[U] = \tilde{G}'$, so we can apply \Cref{stmt:motivationForJprime} to conclude that
    \begin{equation}\label{eq:inclusionInClaim}
      \pi_v\big(\cal{H}[I]\big) \cup \cal{F} \subset \frakI_{F, G', G},
    \end{equation}
    where we replaced $\cal{F} = \frakI_{F, \tilde{G}', \tilde{G}}$ in \eqref{eq:linkOfProjectionIsInducedCopiesWithV}.

    It also follows from $G' \in \Psi_G$ that $\frakI_{F, G', G}$ is not $(p, R' + 1)$-Janson, so we can combine \eqref{eq:inclusionInClaim} with the fact that being Janson is increasing, \Cref{stmt:JansonIsDownset}, to deduce that $\pi_v\big(\cal{H}[I]\big) \cup \cal{F}$ is also not $(p, R' + 1)$-Janson.
    But now, as we chose
    $$\eta = p^4 \Big(\frac{\,q\,}{2}\Big)^{4 s}, \qquad R = 2^{-5} r^{-1} p m, \qquad q = 2^{-15} r^{-2} \qquad \text{ and } \qquad p = 2^{-25} k^{-2} r^{-4},$$ one can verify that
    \begin{equation*}
      \eta R = 2^{-4s - 5} r^{-1} p^5 q^{4 s} m \ge 1
    \end{equation*}
    by $m \ge r^{C k}$, $s \le k$ and $C = 300$.
    We therefore conclude that $\pi_v\big(\cal{H}[I]\big) \cup \cal{F}$ is not $(p, R' + \eta R)$-Janson, because being $(p, R)$-Janson is decreasing in $R$ by \Cref{stmt:jansonParametersMonotone}.
  \end{claimproof}

  By \cref{item:containersForNonJansonInclusions} in \Cref{stmt:containersForNonJanson} and \Cref{stmt:containersAreInFactContainers}, for all graphs $G'$ and $G$ such that $G[U] = \tilde{G}$ and $G' \in \Psi_G$, there exists $X \in \cal{X}$ such that $\iota(G', G) \subset X$.
  Let
  \begin{equation*}
    \cal{X}' = \big\{X \in \cal{X} : |X^{(1)}| \ge m/(4r)\big\},
  \end{equation*}
  and we will show that there is also $X \in \cal{X}'$ such that $\iota(G', G) \subset X$.

  \begin{claim}\label{stmt:containersPrimeAreContainers}
    If $G[U] = \tilde{G}$ and $G' \in \Psi_G$, then there exists $X \in \cal{X}'$ such that $\iota(G', G) \subset X$.
  \end{claim}

  \begin{claimproof}
    Fix $G$ with $G[U] = \tilde{G}$ and $G' \in \Psi_G$, and let $X \in \cal{X}$ be such that $I = \iota(G', G) \subset X$.
    By the definition of $\iota$, we have
    \begin{equation*}
      \mathrm{N}_{G'}(v) = I^{(1)} \subset X^{(1)} \qquad \text{ and } \qquad \mathrm{N}_G(v)^\complement = I^{(0)} \subset X^{(0)},
    \end{equation*}
    and therefore
    \begin{equation}\label{eq:containerForNeighborhoodMustBeLarge}
      |X^{(1)}| \ge d_{G'}(v) \ge \frac{m}{4r}
    \end{equation}
    where the last inequality is due to $G' \in \Psi_G$.
    We conclude that $X \in \cal{X}'$.
  \end{claimproof}

  Taking a union bound over choices of $\cal{X}'$, we can bound the probability in \eqref{eq:probOfGPrime} from above by
  \begin{equation}\label{eq:upperBoundViaContainers}
    \PP\big(\exists G' \subset G : G' \in \Psi_G \mid G[U] = \tilde{G}\big) \le \sum_{X \in \cal{X}'} \PP\big(\mathrm{N}_G(v)^\complement \subset X^{(0)}\big)
  \end{equation}
  using \Cref{stmt:containersPrimeAreContainers} and replacing the event $\{\iota(G', G) \subset X\}$ by $\big\{\mathrm{N}_G(v)^\complement \subset X^{(0)}\big\}$, which it implies by \eqref{eq:projectionOfIota}.
  We now want an upper bound for the probability of this event for each $X \in \cal{X}'$.

  \begin{claim}\label{stmt:UMinusX0IsLarge}
    For every $X \in \cal{X}'$, we have
    \begin{equation}\label{eq:UMinusX0IsLarge}
      |U \setminus X^{(0)}| \ge \frac{m}{16 r}.
    \end{equation}
  \end{claim}

  \begin{claimproof}
    By \cref{item:containersForNonJansonAreNonJanson} in \Cref{stmt:containersForNonJanson}, there is $Y \subset X$ with
    \begin{equation}\label{eq:YIsLarge}
      |Y| \ge |X| - \frac{n}{2^8r} = |X^{(0)}| + |X^{(1)}| - \frac{n}{2^8 r} \ge |X^{(0)}| + \Big(\frac{1}{4 r} - \frac{1}{2^7 r}\Big) m
    \end{equation}
    such that the hypergraph $\pi(\cal{H}[Y])$ is not $(p, R)$-Janson, where we used \eqref{eq:containerForNeighborhoodMustBeLarge} and $n = 2m$.

    Taking $W = Y^{(0)} \cap Y^{(1)}$, observe that $\frakI_{F^{-}, \tilde{G}', \tilde{G}}[W]$ is not $(p, R)$-Janson.
    To check that, recall that the hypergraph $\pi(\cal{H}[Y])$ is not $(p, R)$-Janson by \cref{item:containersForNonJansonAreNonJanson} in \Cref{stmt:containersForNonJanson}.
    It then follows from $$\frakI_{F^{-}, \tilde{G}', \tilde{G}}[W] = \pi(\cal{H})[W] \subset \pi\big(\cal{H}[Y]\big)$$ by \Cref{stmt:piCalHIsHypergraphOfInducedCopies} and \Cref{stmt:containmentHypergraphs} and the fact that being $(p, R)$-Janson is increasing, \Cref{stmt:JansonIsDownset}, that $\frakI_{F^{-}, \tilde{G}', \tilde{G}}[W]$ cannot be $(p, R)$-Janson.

    Since we have assumed in the statement of \Cref{stmt:extensionOfJansonCollection} that $\frakI_{F^{-}, \tilde{G}', \tilde{G}}[W]$ is $(p, R)$-Janson whenever $W \subset U$ satisfies $|W| \ge m / (8 r)$, the fact that choosing $W = Y^{(0)} \cap Y^{(1)}$ results in a subhypergraph that is not $(p, R)$-Janson implies that
    \begin{equation}\label{eq:boundOnTheSizeOfIntersection}
      |Y^{(0)} \cap Y^{(1)}| < \frac{m}{8 r}.
    \end{equation}
    Manipulating \eqref{eq:boundOnTheSizeOfIntersection}, we obtain
    \begin{equation*}
      |Y| = |Y^{(0)} \cup Y^{(1)}| + |Y^{(0)} \cap Y^{(1)}| < \Big(1+\frac{1}{8r}\Big)|U|,
    \end{equation*}
    which combined with \eqref{eq:YIsLarge} yields $$|X^{(0)}| < \Big(1 - \frac{1}{16 r}\Big)m,$$ and hence \eqref{eq:UMinusX0IsLarge}, as desired.
  \end{claimproof}

  Applying \Cref{stmt:UMinusX0IsLarge} to each term in \eqref{eq:upperBoundViaContainers}, we obtain
  \begin{equation}\label{eq:boundOnEachContainerEvent}
    \PP\big(\mathrm{N}_G(v)^\complement \subset X^{(0)}\big) = \PP\Big(\mathrm{N}_G(v)^\complement \cap \big(U\setminus X^{(0)}\big) = \emptyset\Big) = 2^{-|U \setminus X^{(0)}|} \le 2^{- m / (16 r)}
  \end{equation}
  for each $X \in \cal{X}'$, using that $G \sim \Gnp(m + 1, 1/2)$.
  Replacing \eqref{eq:boundOnEachContainerEvent} back in \eqref{eq:upperBoundViaContainers} yields
  \begin{equation*}
    \sum_{X \in \cal{X}'} \PP\big(\mathrm{N}_G(v)^\complement \subset X^{(0)}\big) \le |\cal{X}'| \, 2^{- m / (16 r)}.
  \end{equation*}

  Now, we can bound the size of $\cal{X}'$ using \eqref{eq:containersForNonJansonSmallFamily} and $2 q n \le m / (2^{13} r^2)$, where the latter holds by $n = 2m$ and our choice of $q = 2^{-15} r^{-2}$, to obtain
  \begin{equation*}
    \PP\big(\exists G' \subset G : G' \in \Psi_G \mid G[U] = \tilde{G}\big) \le \bigg(\frac{\,2\,}{q}\bigg)^{2 q n} 2^{- m / (16 r)} \le 2^{ m/(2^9 r) - m / (16 r)} \le 2^{ - m / (32 r)}
  \end{equation*}
  since $r \ge 2$.
\end{proof}

\section{Proof of \texorpdfstring{\Cref{stmt:key}}{\ref{stmt:key}}}\label{sec:proofOfKey}

The purpose of this \namecref{sec:proofOfKey} is to give a proof of \hyperlink{keyRestated}{\Cref{stmt:key}}, restated below.
We begin with an intuitive and informal overview of the proof, with the purpose of motivating the intermediate results in the \namecref{sec:proofOfKey}, and then introduce the details and technicalities in the \namecrefs{sec:changeToAnotherBadEvent} that follow.

\hypertarget{keyRestated}{\keyStmt*}

Our informal overview of the proof of \hyperlink{keyRestated}{\Cref{stmt:key}} starts with a statement of our setting and strategy.
We assume that $G \in \cal{B}(\mathbf{H}) \cap \cal{E}(\mathbf{s})$, i.e.\ $G$ admits a ``bad'' colouring $c : E(G) \to [r]$ in which the copies of $H_i \subset G_i$ that are induced in $G$ are not $(p, pN)$-Janson, even though $G$ satisfies the inductive assumption, represented here by the event $\cal{E}(\mathbf{s})$.
Our goal is to show that such $G$ are extremely rare when $G \sim \Gnp(N, 1/2)$, which we accomplish by applying \Cref{stmt:extensionOfJansonCollection} with $\tilde{G} = G[U]$ and $\tilde{G}' = G_\ell^{(c)}[U]$ for a certain vertex subset $U$ and a specific colour $\ell \in [r]$.
Proving the existence of this set $U$ is not difficult, and is the main purpose of the intermediate results in this \namecref{sec:proofOfKey}.

To reach a point where we can apply \Cref{stmt:extensionOfJansonCollection}, that is, to show that this set $U$ exists, we combine \Cref{stmt:existenceOfTupleForEachG} and \Cref{stmt:maximalTupleForG}.
The proof of the former \namecref{stmt:existenceOfTupleForEachG} uses the induction hypothesis to conclude that $\frakI_{H_i^-, G_i^{(c)}, G}[W]$ is $(p, p|W|)$-Janson for all ``bad'' colourings $c$ and all large $W \subset U$.

In the proof \Cref{stmt:maximalTupleForG}, we (roughly) construct a set $U \subset V(G)$ vertex-by-vertex, starting from an arbitrary vertex subset of size $\delta N$.
Adding a vertex $v$ to $U$ increments the Janson parameter of $H_i \subset G_i[U]$ for some colour $i$, in the sense that if $\frakI_{H_i, G_i, G}[U]$ was $(p, R_i)$-Janson for some $R_i \ge 0$, then $\frakI_{H_i, G_i, G}[U \cup \{v\}]$ is $(p, R_i + 1)$-Janson.
The set $U$ is complete when there are no more vertices whose addition to $U$ would increase the Janson parameter of $H_i$ for some $i \in [r]$, and we show that the final size of $U$ is at most $2 \delta N$.

These are the two preliminaries that we require before the proof of \hyperlink{keyRestated}{\Cref{stmt:key}}, which we briefly and informally discuss now.
With $U$ given by \Cref{stmt:existenceOfTupleForEachG} and \Cref{stmt:maximalTupleForG}, we will apply \Cref{stmt:extensionOfJansonCollection} to the graphs $\tilde{G} = G[U]$ and $\tilde{G}' = G_\ell^{(c)}[U]$ for a certain colour $\ell$ and many vertices $v \not \in U$ when the colouring $c$ is bad, relying on the independence of these events for each $v$ to obtain the required bound on their joint probability.
These vertices $v \not \in U$ are chosen first to ensure that their degree to $U$ is not small, so one colour $i_v \in [r]$ also has sufficiently many neighbours in $U$.
We then select $\ell$ as the majority colour among the $i_v$, and further restrict to those $v$ for which $i_v = \ell$.

To formally implement this outline, we first address a technicality: we are not able to directly prove that the final collection of copies of $H_\ell$ is $(p, p N)$-Janson, only $(p, 2^{-9}r^{-1}\delta pN)$-Janson\footnote{This is because we can only apply \Cref{stmt:extensionOfJansonCollection} for $R' \le R/16$, where the collection of copies of $H_\ell^-$ contained in $W$ is $(p, R)$-Janson for every $W \subset U$ with $|W| \ge |U|/(8r)$, and the event $\cal{E}(\mathbf{s})$ only tells us that this hypergraph is $(p,p|W|)$-Janson.}.
As we will see, this is not a problem, because we show in \Cref{stmt:badEventContainedInBadEventPrime}, using double counting, that a hypergraph $\cal{G}$ that is not $(p, p N)$-Janson contains a subset $S$ of $\delta^{2/3} N$ vertices such that $\cal{G}[S]$ is not $(p, 2^{-9}r^{-1}\delta pN)$-Janson, and we will be able to work entirely inside this set $S$.
The next \namecref{sec:changeToAnotherBadEvent} proves \Cref{stmt:badEventContainedInBadEventPrime}, and the rest of the section implements the above outline to show that $U$ exists, and finally to prove \hyperlink{keyRestated}{\Cref{stmt:key}} in \Cref{sec:proofOfKeyProper}.

\subsection{Changing to another bad event}\label{sec:changeToAnotherBadEvent}

First, recall the definition of the bad event $\cal{B}(\mathbf{H})$.

\badEvent*

As previously mentioned, the first stage in the proof is to show that we can replace $\cal{B}(\mathbf{H})$ by an alternative event $\cal{B}'(\mathbf{H})$.
The main advantage of this change is that it reduces the Janson parameter by a factor of order $r^{-1} \delta$, at the cost of assuming that it only holds for a single subset $S$ of size $|S| \ge \delta^{2/3}N$.

\begin{defi}\label{def:newBadEvent}
  Given a collection of graphs $H_1, \ldots, H_r$, let $\mathbf{H} = (H_i)_{i \in [r]}$ and let $\cal{B}'(\mathbf{H})$ be the family of graphs $G$ with the following property.
  There exists $S \subset V(G)$ with $|S| \ge \delta^{2/3} v(G)$ and a colouring $c : E(G[S]) \to [r]$ such that $\frakI_{H_i, G_i, G}[S]$ is not $(p, 2^{-9}r^{-1}\delta p \, v(G))$-Janson for all $i \in [r]$.
\end{defi}

It is crucial that the original bad event $\cal{B}(\mathbf{H})$ is contained in the variant $\cal{B}'(\mathbf{H})$, which we now prove with a simple double counting argument.

\begin{lem}\label{stmt:badEventContainedInBadEventPrime}
  For every $k \in \NN$ and graphs $H_1, \ldots, H_r$,
  \begin{equation*}
    \cal{B}(\mathbf{H}) \subset \cal{B}'(\mathbf{H}),
  \end{equation*}
  where $\mathbf{H}=(H_i)_{i \in [r]}$.
\end{lem}

As we need to work with measures in the Janson property, the following \namecref{stmt:sumMeasureDegreeIsSumOfMeasureDegrees}, albeit a trivial consequence of the definitions, will be useful when proving \Cref{stmt:badEventContainedInBadEventPrime}.

\begin{restatable}{obs}{sumMeasureDegreeIsSumOfMeasureDegrees}\label{stmt:sumMeasureDegreeIsSumOfMeasureDegrees}
  For all $s \in \NN$ and hypergraphs $\cal{G}$, if $\vartheta_1, \ldots, \vartheta_s : \cal{G} \to \RR_{\ge 0}$ satisfy
  \begin{equation*}
    \vartheta = \sum_{i = 1}^s \vartheta_i,
  \end{equation*}
  then,
  \begin{equation*}
    e(\vartheta) = \sum_{i = 1}^s e(\vartheta_i) \qquad \text{and} \qquad d_\vartheta(L) = \sum_{i = 1}^s d_{\vartheta_i}(L)
  \end{equation*}
  for all $L \subset V(\cal{G})$.
\end{restatable}

We can now prove \Cref{stmt:badEventContainedInBadEventPrime}.

\begin{proof}[Proof of \Cref{stmt:badEventContainedInBadEventPrime}]
  Assume that $G \not \in \cal{B}'(\mathbf{H})$ and let $V = V(G)$ and $N = |V|$.
  Let $c : E(G) \to [r]$ be an arbitrary $r$-colouring of the edges of $G$, and observe that, by \Cref{def:newBadEvent}, for every $S \subset V$ with $|S| \ge \delta^{2/3} N$ there exists $j \in [r]$ such that $\frakI_{H_j, G_j, G}[S]$ is $(p, 2^{-9} r^{-1}\delta p N)$-Janson.
  Thus, if for each $j \in [r]$, we define
  $$\mathbf{V}_j = \Big\{S \subset V : |S| = \delta^{2/3} N \text{ and } \frakI_{H_j, G_j, G}[S] \text{ is } (p, 2^{-9} r^{-1}\delta p N)\text{-Janson}\Big\}$$
  then
  $$ \bigcup_{j = 1}^r \mathbf{V}_j = \binom{V}{\delta^{2/3} N}$$
  and hence there exists $i \in [r]$ such that
  \begin{equation}\label{eq:size-of-bfSi-is-big}
    |\mathbf{V}_i| \ge \frac{1}{r}\binom{N}{\delta^{2/3}N},
  \end{equation}
  so fix such an $i$.
  For each $S \in \mathbf{V}_i$, as $\frakI_{H_i, G_i, G}[S]$ is $(p, 2^{-9} r^{-1}\delta p N)$-Janson by assumption, there exists $\nu_S : \frakI_{H_i, G_i, G}[S] \to \RR_{\ge 0}$ satisfying
  \begin{equation}\label{eq:badEvents:edgesSndLambdaNuS}
    e(\nu_S) = 1 \qquad \text{and} \qquad \Lambda_p(\nu_S) < \frac{2^{9}r}{\delta pN}
  \end{equation}
  by \Cref{stmt:normalisedJansonWitnesses}.
  Now, define $\nu : \frakI_{H_i, G_i, G} \to \RR_{\ge 0}$ by
  \begin{equation*}
    \nu = \sum_{S \in \mathbf{V}_i} \nu_{S}
  \end{equation*}
  and note that if
  \begin{equation}\label{eq:badEvents:goal}
    \Lambda_p(\nu) < \frac{e(\nu)^2}{pN},
  \end{equation}
  then $\frakI_{H_i, G_i, G}$ is $(p, p N)$-Janson and therefore $G \not \in \cal{B}(\mathbf{H})$ by definition, since $c$ is an arbitrary $r$-colouring of the edges of $G$.

  Recall that $V = V(\frakI_{H_i, G_i, G})$ and also that, by definition,
  \begin{equation*}
    \Lambda_p(\nu) = \sum_{\substack{L \subset V \\ |L| \ge 2}} d_\nu(L)^2 p^{-|L|} = \sum_{\substack{L \subset V \\ |L| \ge 2}} \Big(\sum_{S \in \mathbf{V}_i} d_{\nu_S}(L) \Big)^2p^{-|L|}
  \end{equation*}
  where the last equality is due to \Cref{stmt:sumMeasureDegreeIsSumOfMeasureDegrees}.
  Further observe that since $\nu_S$ is supported only on $\frakI_{H_i, G_i, G}[S]$, then we can only have $d_{\nu_S}(L) > 0$ if $L \subset S$.
  Hence, denoting the subfamily
  $$\mathbf{T}_L = \{ S \in \mathbf{V}_i : L \subset S\}$$
  for every $L \subset V$, we have
  \begin{equation}\label{eq:badEvents:cauchySchwarz}
    \Lambda_p(\nu) = \sum_{\substack{L \subset V \\ |L| \ge 2}} \Big(\sum_{S \in \mathbf{T}_L} d_{\nu_S}(L) \Big)^2 p^{-|L|} \le \sum_{\substack{L \subset V \\ |L| \ge 2}} |\mathbf{T}_L| \sum_{S \in \mathbf{T}_L} d_{\nu_S}(L)^2 p^{-|L|},
  \end{equation}
  where the last step holds by the Cauchy--Schwarz inequality.
  Now, note that, as every $L$ in the sums of \eqref{eq:badEvents:cauchySchwarz} satisfies $|L| \ge 2$, we can bound $|\mathbf{T}_L|$ for these $L$ by
  \begin{equation*}
    |\mathbf{T}_L| \le \binom{N - 2}{\delta^{2/3}N - 2} \le \delta^{4/3} \binom{N}{\delta^{2/3}N} \le r \delta^{4/3} |\mathbf{V}_i|,
  \end{equation*}
  where we used \eqref{eq:size-of-bfSi-is-big} in the last inequality.
  Substituting this into \eqref{eq:badEvents:cauchySchwarz} and using \eqref{eq:badEvents:edgesSndLambdaNuS} together with the definition of $\Lambda_p(\nu)$ yields
  \begin{equation}\label{eq:badEvents:lambdaPNus}
    \Lambda_p(\nu) \le r \delta^{4/3} |\mathbf{V}_i| \sum_{S \in \mathbf{V}_i} \Lambda_p(\nu_S) < r \delta^{4/3} |\mathbf{V}_i|^2 \frac{2^9 r}{\delta p N} = 2^{9} r^2 \delta^{1/3} \frac{|\mathbf{V}_i|^2}{pN}.
  \end{equation}

  Note that, since $e(\nu_S) = 1$ for all $S \in \mathbf{V}_i$ by \eqref{eq:badEvents:edgesSndLambdaNuS}, it follows from \Cref{stmt:sumMeasureDegreeIsSumOfMeasureDegrees} that
  \begin{equation*}
    e(\nu) = |\mathbf{V}_i|,
  \end{equation*}
  which, replaced in \eqref{eq:badEvents:lambdaPNus}, results in our goal, \eqref{eq:badEvents:goal},
  \begin{equation*}
    \Lambda_p(\nu) < 2^{9} r^2 \delta^{1/3} \frac{e(\nu)^2}{pN} < \frac{e(\nu)^2}{p N}
  \end{equation*}
  where the last inequality is due to our choice of $\delta = r^{-50}$.
  Therefore, the hypergraph $\frakI_{H_i, G_i, G}$ is $(p, p N)$-Janson, which implies that $G \not \in \cal{B}(\mathbf{H})$ since the colouring $c : E(G) \to [r]$ was arbitrary.
\end{proof}

\subsection{Finding a set \texorpdfstring{$U$}{U} to apply \texorpdfstring{\Cref{stmt:extensionOfJansonCollection}}{\ref{stmt:extensionOfJansonCollection}}}

Using \Cref{stmt:badEventContainedInBadEventPrime}, we can bound
\begin{equation}\label{eq:changingOfBadEvents}
  \PP\big(G \in \cal{B}(\mathbf{H}) \cap \cal{E}(\mathbf{s})\big) \le \PP\big(G \in \cal{B}'(\mathbf{H}) \cap \cal{E}(\mathbf{s})\big)
\end{equation}
where $G \sim \Gnp(N, 1/2)$, which leads us to the second stage in the proof of \hyperlink{keyRestated}{\Cref{stmt:key}}.
In it, we will apply \Cref{stmt:extensionOfJansonCollection} to bound $\PP\big(G \in \cal{B}'(\mathbf{H})\cap \cal{E}(\mathbf{s})\big)$, but the setup of this application requires some work.
First we show, with a deterministic argument, that for any graph $G \in \cal{B}'(\mathbf{H})$ with a ``bad'' colouring $c$, we can find a set $U$ that satisfy the requirements of \Cref{stmt:extensionOfJansonCollection}.

Before describing the concrete properties of $U$, we establish a correspondence between $G \in \cal{B}'(\mathbf{H})$, the sets $S$ that appear in \Cref{def:newBadEvent}, and these bad colourings $c$.
To do that, it will be helpful to define the common setting for the rest of this \namecref{sec:extensionOfJansonCollection}.
Fix then $k \in \NN$ and $s_1, \ldots, s_r \in \NN$ such that $s_i \le k$ for each $i \in [r]$.
Further fix graphs $H_1, \ldots , H_r$ such that $v(H_i) \le s_i$ for all $i \in [r]$, let $ \mathbf{s} = (s_i)_{i \in [r]}$ and $ \mathbf{H} = (H_i)_{i \in [r]}$, and fix $N \in \NN$ satisfying \hyperlink{keyRestated}{\eqref{eq:boundOnNInKey}}.

\begin{defi}
  For a graph $G$, define the collection $\mathbf{S}(G)$ by
  \begin{equation*}
    \mathbf{S}(G) = \left\{ (S, c) : ~
    \begin{gathered}
      S \subset V(G) \text{ with } |S| \ge \delta^{2/3} N ~ \text{and} ~ c : E(G[S]) \to [r] \\
      \text{such that} ~ \forall i \in [r], ~ \frakI_{H_i, G_i, G}[S] \text{ is not } (p, 2^{-9} r^{-1}\delta p N)\text{-Janson}
    \end{gathered}
     \right\}.
  \end{equation*}
\end{defi}

Recall from \Cref{def:newBadEvent} that if $G \in \cal{B}'(\mathbf{H})$, then there exist $S \subset V(G)$ and $c : E(G[S]) \to [r]$ such that $(S, c) \in \mathbf{S}(G)$.
We will next show that if $(S, c) \in \mathbf{S}(G)$, then for every large subset $W \subset S$, the hypergraph $\frakI_{H_i^-, G_i, G}[W]$ is $(p, p|W|)$-Janson.
To prove that, we will require the event $\cal{E}(\mathbf{s})$, whose definition we recall for the reader's convenience.

\mainEvent*

The following \namecref{stmt:existenceOfTupleForEachG} is the only place in the proof of \Cref{stmt:key} where we will use the event $\cal{E}(\mathbf{s})$.
Since later we will choose the set $U$ to be a subset of $S$, the \namecref{stmt:existenceOfTupleForEachG} immediately implies that $\frakI_{H_i^-, G_i, G}[W]$ is $(p, p|W|)$-Janson for every large subset $W \subset U$.

\begin{lem}\label{stmt:existenceOfTupleForEachG}
  For all $G \in \cal{E}(\mathbf{s})$ and $(S, c) \in \mathbf{S}(G)$, the hypergraph $\frakI_{H_i^{-}, G_i, G}[W]$ is $(p, p |W|)$-Janson for every $i \in [r]$ and every $W \subset S$ with $|W| \ge \delta N/(8r)$.
\end{lem}

\begin{proof}
  Fix an arbitrary $i \in [r]$, let $F_j = H_j$ for all $j \in [r] \setminus \{i\}$, and take $F_i = H_i^-$, a choice that satisfies
  \begin{equation}\label{eq:choiceSatisfiesEventsRequirements}
    \sum_{j = 1}^r v(F_j) = \sum_{j = 1}^r s_j - 1.
  \end{equation}
  By $G \in \cal{E}(\mathbf{s})$ and \eqref{eq:choiceSatisfiesEventsRequirements}, for every $W \subset S \subset V(G)$ with $|W| \ge \delta N / (8r)$, there exists a colour $\ell \in [r]$ such that $\frakI_{F_\ell, G_\ell, G}[W]$ is $(p, p |W|)$-Janson.
  However, as
  \begin{equation*}
    p |W| \ge 2^{-9} r^{-1}\delta p N,
  \end{equation*}
  we conclude that $\frakI_{F_\ell, G_\ell, G}[W]$ is $(p, 2^{-9} r^{-1}\delta p N)$-Janson.
  Since the Janson property is increasing by \Cref{stmt:JansonIsDownset}, it follows that $\frakI_{F_\ell, G_\ell, G}[S]$ is also $(p, 2^{-9} r^{-1}\delta p N)$-Janson.

  Now, recall that $\frakI_{H_j, G_j, G}[S]$ is not $(p, 2^{-9} r^{-1}\delta p N)$-Janson for all $j \in [r]$ since $(S, c) \in \mathbf{S}(G)$.
  Combining our choice of $F_j = H_j$ for all $j \neq i$ with the fact that $\frakI_{F_\ell, G_\ell, G}[S]$ is $(p, 2^{-9} r^{-1}\delta p N)$-Janson, the only remaining possibility is that $\ell = i$.
  It follows that, for every $W\subset S$ satisfying $|W| \ge \delta N/(8r)$, the hypergraph $\frakI_{H_i^-, G_i, G}[W]$ is $(p, p |W|)$-Janson.
  Since $i$ was arbitrary, this completes the proof of the \namecref{stmt:existenceOfTupleForEachG}.
\end{proof}

The remaining properties that $U$ will have are all guaranteed simultaneously by the way we construct it.
For each $i \in [r]$, the hypergraph $\frakI_{H_i, G_i, G}[U]$ will be $(p, R_i)$-Janson for some $R_i \ge 0$, and $\frakI_{H_i, G_i, G}[U \cup \{v\}]$ will not be $(p, R_i + 1)$-Janson for any $v \not \in U$.
The proof of \Cref{stmt:maximalTupleForG} shows that we can find such a $U$ with a routine induction.

\begin{lem}\label{stmt:maximalTupleForG}
  Let $G$ be a graph.
  If $(S, c) \in \mathbf{S}(G)$, then there exist $U \subset S$ and $R_1, \ldots, R_r \in \ZZ_{\ge 0}$ such that
  \begin{equation}\label{eq:BoundOnRi}
    |U| = \delta N + \sum_{i=1}^r R_i \qquad \text{and} \qquad R_1,\ldots, R_r \le \frac{p |U|}{2^9 r}
  \end{equation}
  which further satisfy, for all $i \in [r]$,
  \begin{enumerate}[\normalfont (a)]
    \item $\frakI_{H_i, G_i, G}[U]$ is $(p, R_i)$-Janson, and \label{item:maximalJanson}
      \smallskip
    \item $\frakI_{H_i, G_i, G}\big[U \cup \{v\}\big]$ is not $(p, R_i + 1)$-Janson for all $v \in S \setminus U$. \label{item:nonExtendableJanson}
  \end{enumerate}
\end{lem}

\begin{proof}
  Observe first that every set $U \subset S$ of size $\delta N$ satisfies \cref{item:maximalJanson} with $R_i = 0$ for all $i \in [r]$, since $\frakI_{H_i, G_i, G}[U]$ is $(p, 0)$-Janson by definition.
  Now take $U$ of size $\delta N + \sum_{i=1}^r R_i$ maximising $\sum_{i=1}^{r} R_i$ among the choices that satisfy \cref{item:maximalJanson}.
  Note that $\frakI_{H_i, G_i, G}[U \cup \{v\}]$ is $(p,R_i)$-Janson for every $v \in S$, since $\frakI_{H_i, G_i, G}[U]$ is $(p, R_i)$-Janson, and by \Cref{stmt:JansonIsDownset}, so \cref{item:nonExtendableJanson} holds by the maximality of $\sum_{i = 1}^{r} R_i$.
  It therefore remains to prove the inequality in \eqref{eq:BoundOnRi}.

  Suppose for a contradiction that $R_i > p |U|/(2^9 r)$ for some $i \in [r]$ and observe that
  \[R_i > \frac{p|U|}{2^9 r} \ge \frac{\delta p N}{2^9 r},\]
  so $\frakI_{H_i, G_i, G}[U]$ is $(p, 2^{-9}r^{-1} \delta p N)$-Janson by \Cref{stmt:jansonParametersMonotone}.
  Since $U \subset S$, \Cref{stmt:JansonIsDownset} implies that $\frakI_{H_i, G_i, G}[S]$ is also $(p, 2^{-9}r^{-1} \delta p N)$-Janson, contradicting the fact that $(S, c) \in \mathbf{S}(G)$.
\end{proof}

\subsection{The proof}\label{sec:proofOfKeyProper}

Having established that there is a $U$ satisfying most of the requirements of \Cref{stmt:extensionOfJansonCollection}, we will combine the following trivial consequence of the Chernoff bound with a pigeonhole argument to find many vertices $v$ whose degree to $U$ in some colour $\ell\in [r]$ is not small.

\begin{restatable}{obs}{lowDegreeVerticesInRandomGraph}\label{stmt:lowDegreeVerticesInRandomGraph}
  Let $G \sim \Gnp(N, 1/2)$, let $U \subset S \subset V(G)$ and set
  \begin{equation*}
    S' = \big\{v \in S \setminus U : d_G(v, U) > |U| / 4\big\}.
  \end{equation*}
  If $2^{10} \le |U| \le |S| /4 $, then
  \begin{equation*}
    \PP\big(|S'| \le |S| / 4\big) \le \exp\big(- 2^{-6} |U| \/ |S|\big).
  \end{equation*}
\end{restatable}

\begin{proof}
  For a vertex $v \in S \setminus U$, the Chernoff bound implies that
  \begin{equation*}
    \PP( v \not\in S') = \PP\big( d_G(v, U) \le |U| / 4\big) \le \exp\big( - |U|/16\big),
  \end{equation*}
  and observe that these events are independent for distinct vertices of $S \setminus U$.
  If $|S'|\le |S|/4$, then at least $|S|/2$ vertices of $S \setminus U$ fail to be in $S'$.
  Taking a union bound then yields
  \begin{equation*}
    \PP\big(|S'| \le |S|/4\big)  \le 2^{|S|} \exp\bigg( - \frac{|S|}{2} \,\frac{|U|}{16}\bigg) \le \exp\big(-2^{-6}|U||S|\big),
  \end{equation*}
  where last inequality uses that $|U| \ge 2^{10}$.
\end{proof}

With all of the above, the proof of \hyperlink{keyRestated}{\Cref{stmt:key}} proceeds by fixing the sets $S$, $U$ and the colouring of $G[U]$, all of which we will take union bounds over.
We will use \Cref{stmt:lowDegreeVerticesInRandomGraph} to find many vertices whose degree to $U$ in colour $\ell$ is not small, and we will then be able to check that the graphs $\tilde{G}' = G_\ell^{(c)}[U]$ and $\tilde{G} = G[U]$ satisfy all the conditions required by \Cref{stmt:extensionOfJansonCollection}.
To complete the proof, it will then suffice to observe that for each $v$ the events bounded by \Cref{stmt:extensionOfJansonCollection} are independent, and hence the probabilities multiply.

\begin{proof}[Proof of \Cref{stmt:key}]
  Recall first that we fixed $p = \big(2^{25} k^2 r^4\big)^{-1}$ in \eqref{eq:fixedParams}.
  We claim that if $v(H_i) = 1$ for some $i \in [r]$, then trivially we have $\cal{B}(\mathbf{H}) = \emptyset$.
  Indeed, every non-empty 1-uniform hypergraph is $(p, R)$-Janson for all $p > 0$ and $R$, since $\Lambda_p(\nu) = 0$ regardless of the measure $\nu$.
  We may therefore assume that $v(H_i) \ge 2$ for all $i \in [r]$.

  By \Cref{stmt:badEventContainedInBadEventPrime}, we have $\cal{B}(\mathbf{H}) \subset \cal{B}'(\mathbf{H})$, and therefore
  \begin{equation*}
    \PP\big(G \in \cal{B}(\mathbf{H}) \cap \cal{E}(\mathbf{s})\big) \le \PP\big(G \in \cal{B}'(\mathbf{H}) \cap \cal{E}(\mathbf{s})\big),
  \end{equation*}
  where $G \sim \Gnp(N, 1/2)$, here and in every probability statement in this proof.
  It will also be convenient to assume that every such graph shares the same vertex set $V(G) = [N]$.

  Let $\mathbf{U}$ denote the collection of tuples $(S, U, (\tilde{G}_i)_{i \in [r]})$ with the following properties.
  The sets $U \subset S \subset [N]$ satisfy
  \begin{equation}\label{eq:bfUSetInequalities}
    |S| \ge \delta^{2/3} N \quad \text{ and } \quad |U| = \delta N + \sum_{i = 1}^{r} R_i, \quad \text{ with } \quad 0 \le R_1,\ldots , R_r \le \frac{p |U|}{2^9 r}.
  \end{equation}
  Moreover, letting
  \begin{equation}\label{eq:defTildeG}
    \tilde{G} = \bigcup_{i \in [r]} \tilde{G}_i,
  \end{equation}
  each tuple $(S, U, (\tilde{G}_i)_{i \in [r]}) \in \mathbf{U}$ satisfies, for all $i \in [r]$, that $V(\tilde{G}_i) = U$,
  \begin{enumerate}[(1)]
    \item \label{item:propsOfBfU:largeSetsAreJanson} $\frakI_{H_i^-, \tilde{G}_i, \tilde{G}}[W]$ is $(p, p |W|)$-Janson for every $W \subset U$ with $|W| \ge |U|/(8r)$, and
      \smallskip
    \item \label{item:propsOfBfU:baseSetIsJanson} $\frakI_{H_i, \tilde{G}_i, \tilde{G}}[U]$ is $(p, R_i)$-Janson.
  \end{enumerate}
  This collection is important because of the following \namecref{stmt:mappingEachGToAMaximalTuple}.
  Before stating it, define $f(G) = (S, c)$ to map each $G \in \cal{B}'(\mathbf{H}) \cap \cal{E}(\mathbf{s})$ to a fixed choice of $(S, c) \in \mathbf{S}(G)$.

  \begin{claim}\label{stmt:mappingEachGToAMaximalTuple}
    For each $G \in \cal{B}'(\mathbf{H}) \cap \cal{E}(\mathbf{s})$ and $(S, c) = f(G)$, there is $\sigma = \big(S, U, (\tilde{G}_i)_{i \in [r]}\big) \in \mathbf{U}$ for which $\tilde{G}_i = G_i^{(c)}[U]$ and the hypergraph $\frakI_{H_i, G_i, G}\big[U \cup \{v\}\big]$ is not $(p, R_i + 1)$-Janson for all $i \in [r]$ and all $v \in S \setminus U$.
  \end{claim}

  Note that the latter property of $U$ and every $v \in S \setminus U$ in \Cref{stmt:mappingEachGToAMaximalTuple} cannot be defined only in terms of $(\tilde{G}_i)_{i \in [r]}$, because it also depends on the colouring $c : E(G[S]) \to [r]$ in $(S, c) \in f(G)$.

  \begin{claimproof}[Proof of \Cref{stmt:mappingEachGToAMaximalTuple}]
    Fix $G \in \cal{B}'(\mathbf{H}) \cap \cal{E}(\mathbf{s})$ and $(S, c) = f(G)$.
    Let $U \subset S$ be given by \Cref{stmt:maximalTupleForG}, which implies that it satisfies \eqref{eq:bfUSetInequalities}.
    Further set $\tilde{G}_i = G_i^{(c)}[U]$ for every $i \in [r]$, and recall that $\tilde{G} = \bigcup_{i \in [r]} \tilde{G}_i$ by \eqref{eq:defTildeG}.
    It follows from \cref{item:maximalJanson} in \Cref{stmt:maximalTupleForG} that this choice satisfies \cref{item:propsOfBfU:baseSetIsJanson} in the definition of $\mathbf{U}$ and also that $\frakI_{H_i, G_i, G}\big[U \cup \{v\}\big]$ is not $(p, R_i + 1)$-Janson for all $i \in [r]$ and all $v \in S \setminus U$, where the values of each $R_i$ are given by \Cref{stmt:maximalTupleForG}.

    To prove that this choice also satisfies \cref{item:propsOfBfU:largeSetsAreJanson} in the definition of $\mathbf{U}$, observe first that \Cref{stmt:existenceOfTupleForEachG} implies that $\frakI_{H_i^{-}, G_i, G}[W]$ is $(p, p |W|)$-Janson for every $i \in [r]$ and every $W \subset S$ with $|W| \ge \delta N/(8r)$.
    As $U \subset S$, this conclusion also holds whenever $W \subset U$ and $|W| \ge |U| / (8 r)$, since $|U| \ge \delta N$ by \eqref{eq:bfUSetInequalities}.
    The final observation is that $$\frakI_{H_i^-, \tilde{G}_i, \tilde{G}}[W] = \frakI_{H_i^{-}, G_i, G}[W]$$ by our choice of $\tilde{G}_i = G_i^{(c)}[U]$, so the previous reasoning indeed establishes \cref{item:propsOfBfU:largeSetsAreJanson}.
  \end{claimproof}

  Now, for $\sigma = \big(S, U, (\tilde{G}_i)_{i \in [r]}\big)$, let $\cal{A}(\sigma)$ be the collection of pairs of graphs $G$ and colourings $c : E(G) \to [r]$ such that
  \begin{enumerate}[(a)]
    \item $ G_i[U] = G_i^{(c)}[U] = \tilde{G}_i $ for all $i \in [r]$, and
      \smallskip
    \item the hypergraph $\frakI_{H_i, G_i, G}\big[U \cup \{v\}\big]$ is not $(p, R_i + 1)$-Janson for all $i \in [r]$ and all $v \in S \setminus U$. \label{item:propsOfCalA:extensionIsNotJanson}
  \end{enumerate}
  By \Cref{stmt:mappingEachGToAMaximalTuple}, we know that if $G \in \cal{B}'(\mathbf{H}) \cap \cal{E}(\mathbf{s})$, then there exists $\sigma \in \mathbf{U}$ and a colouring $c : E(G) \to [r]$ such that $(G, c) \in \cal{A}(\sigma)$.
  Taking a union bound over choices of $\sigma \in \mathbf{U}$, but, crucially, \emph{not} over the choices of $c : E(G) \to [r]$, then yields
  \begin{equation}\label{eq:unionBound}
    \PP\big(G \in \cal{B}'(\mathbf{H}) \cap \cal{E}(\mathbf{s})\big) \le \sum_{\sigma \in \mathbf{U}} \PP\big(\exists c \in [r]^{E(G)} : (G, c) \in \cal{A}(\sigma) \big).
  \end{equation}

  Most of the remainder of the proof will be dedicated to proving the following \namecref{stmt:boundForEachSigmaInBfU}.

  \begin{claim}\label{stmt:boundForEachSigmaInBfU}
    For every $\sigma \in \mathbf{U}$,
    \begin{equation}\label{eq:boundForEachSigmaInBfU}
      \PP\big(\exists c \in [r]^{E(G)} : (G, c) \in \cal{A}(\sigma) \big) \le 2^{- 8 r \delta^2 N^2}.
    \end{equation}
  \end{claim}

  To prove \Cref{stmt:boundForEachSigmaInBfU}, we will modify $\cal{A}(\sigma)$ until the event(s) whose probability we need to bound become(s) $\big\{\exists G' \subset G : (G', G) \in \cal{M}_{v, \ell}\big\}$, where
  \begin{equation}\label{eq:mainEventInsideKey}
    \cal{M}_{v, \ell} = \left\{
      (G', G)
      : \,
      \begin{array}{@{}c@{}}
        G'[U] = \tilde{G}_\ell, ~ d_{G'}(v, U) \ge |U| / (4r) ~ \text{ and} \\
        \frakI_{H_\ell, G', G}\big[U \cup \{v\}\big] \,
        \text{ is not } \, (p, R_\ell + 1)\text{-Janson}
      \end{array}
      \right\}.
  \end{equation}
  We will then observe that not only the definition of $\cal{M}_{v, \ell}$ in \eqref{eq:mainEventInsideKey} corresponds to the event whose probability \Cref{stmt:extensionOfJansonCollection} bounds if we take $F = H_\ell$, $R' = R_\ell$ and $\tilde{G}' = \tilde{G}_\ell$, but also that the current setting satisfies the assumptions to apply that \namecref{stmt:extensionOfJansonCollection}.

  To make the proof easier to follow, we will first establish several intermediate \namecrefs{stmt:boundForEachSigmaInBfU} towards \Cref{stmt:boundForEachSigmaInBfU}.
  Because of that, we will fix $\sigma = \big(S, U, (\tilde{G}_i)_{i \in [r]}\big) \in \mathbf{U}$ and abbreviate $\cal{A}(\sigma) = \cal{A}$ until the proof of \Cref{stmt:boundForEachSigmaInBfU}.

  The first step towards proving \eqref{eq:boundForEachSigmaInBfU} is to replace the set $S$ with a large subset $S' \subset S$ by discarding vertices with small degree into $U$.
  By \Cref{stmt:lowDegreeVerticesInRandomGraph}, $S$ contains such a subset with very high probability.
  More precisely, let
  \begin{equation}\label{eq:definitionOfSPrime}
    S' = S'(G) = \big\{v \in S \setminus U : d_G(v, U) > |U| / 4\big\}
  \end{equation}
  for every graph $G$, and define
  $$\cal{A}' = \big\{(G, c) \in \cal{A} : |S'(G)| \ge |S|/4 \big\}.$$

  \begin{claim}\label{stmt:splitIntoCasesWrtSPrime}
    \begin{equation}\label{eq:splitIntoCasesWrtSPrime}
      \PP\big(\exists c \in [r]^{E(G)} : (G, c) \in \cal{A} \big) \le \PP\big(\exists c \in [r]^{E(G)} : (G, c) \in \cal{A}' \big) + \exp\big(- 2^{-6} \delta^{5/3} N^2\big).
    \end{equation}
  \end{claim}

  \begin{claimproof}
    Recalling \eqref{eq:bfUSetInequalities} in the definition of $\mathbf{U}$ and that $p \le 1$, we have
    \begin{equation*}
      |U| \le \delta N + \frac{p|U|}{2^{9}} \le \delta N + \frac{|U|}{2},
    \end{equation*}
    which, since we also have $|U| \ge \delta N$, implies that
    \begin{equation}\label{eq:sizeOfU}
      \delta  N \le |U| \le 2 \delta N.
    \end{equation}
    As $2\delta N \le |S|/4$, we can apply \Cref{stmt:lowDegreeVerticesInRandomGraph} to $S$ and $U$, obtaining as a result
    \begin{equation}\label{eq:complementOfAuxEvent2}
      \PP\big(|S'| \le |S| / 4\big) \le \exp\big(- 2^{-6} |U| \/ |S|\big) \le \exp\big(- 2^{-6} \delta^{5/3} N^2\big),
    \end{equation}
    where the last inequality follows from $|S| \ge \delta^{2/3} N $ by \eqref{eq:bfUSetInequalities} and $|U| \ge \delta N$ by \eqref{eq:sizeOfU}.
    The \namecref{stmt:splitIntoCasesWrtSPrime} now follows from splitting of $\cal{A}$ into $\cal{A}'$ and $\cal{A} \setminus \cal{A}'$ and bounding the latter using \eqref{eq:complementOfAuxEvent2}.
  \end{claimproof}

  We now focus our attention on bounding the first term in the right-hand side of \eqref{eq:splitIntoCasesWrtSPrime}.
  Recalling that our goal is to apply \Cref{stmt:extensionOfJansonCollection} and that this \namecref{stmt:extensionOfJansonCollection} requires only a single subgraph, instead of a colouring, we will restrict our attention to the subgraph $G_\ell$ whose colour $\ell \in [r]$ is the majority colour of the neighbourhood of most vertices in $S'$.

  \begin{claim}\label{stmt:unionBoundOverSetsSAndColoursL}
    \begin{equation*}
      \PP\big(\exists c \in [r]^{E(G)} : (G, c) \in \cal{A}' \big) \le \sum_{\ell \in [r]} \sum_{\substack{A \subset S \\ |A| \ge \frac{|S|}{4 r}}} \PP\Big(\exists G' \subset G : (G', G) \in \bigcap_{v \in A} \cal{M}_{v, \ell} \text{ and } G[U] = \tilde{G} \Big).
    \end{equation*}
  \end{claim}

  \begin{claimproof}
    Let $G$ be a graph, and suppose that there exists a colouring $c : E(G) \to [r]$ such that $(G, c) \in \cal{A}'$.
    We claim that $G[U] = \tilde{G}$, and that there exists a subgraph $G' \subset G$, a colour $\ell \in [r]$ and a subset $A \subset S$ with $|A| \ge |S|/(4 r)$, such that $(G', G) \in M_{v, \ell}$ for every $v \in A$.
    \Cref{stmt:unionBoundOverSetsSAndColoursL} will then follow by taking a union bound over the choices of $\ell$ and $A$.

    To show this, observe first that $G[U] = \tilde{G}$ holds because $(G, c) \in \cal{A}' \subset \cal{A}$, where $\tilde{G} = \bigcup_{i \in [r]} \tilde{G}_i$ was defined in \eqref{eq:defTildeG}.
    Next, note that for each $c : E(G) \to [r]$ and $v \in S'(G)$, the colouring $c$ partitions the edges connecting $v$ and $U$ into $r$ sets.
    As each $v \in S'(G)$ satisfies
    \begin{equation*}
      d_G(v, U) > \frac{|U|}{4}
    \end{equation*}
    by the definition of $S'(G)$, \eqref{eq:definitionOfSPrime}, there exists a colour $j(v) = j \in [r]$ such that
    \begin{equation}\label{eq:pigeonholeForDegrees}
      d_{G_j}(v, U) > \frac{|U|}{4r}.
    \end{equation}
    By the pigeonhole principle and \eqref{eq:pigeonholeForDegrees}, then, there exists a colour $\ell \in [r]$ such that, letting
    \begin{equation*}
      A = A(G, c) = \{v \in S'(G) : j(v) = \ell\},
    \end{equation*}
    we have, as a consequence of $(G, c) \in \cal{A}'$, that
    \begin{equation*}
      |A| \ge \frac{|S'(G)|}{r} \ge \frac{|S|}{4 r}.
    \end{equation*}

    Letting $G' = G_\ell$ be the graph of edges spanned by colour $\ell$ completes the proof of the \namecref{stmt:unionBoundOverSetsSAndColoursL} because $A \subset S \setminus U$ and $(G, c) \in \cal{A}$ satisfy \cref{item:propsOfCalA:extensionIsNotJanson} in the definition of the event $\cal{A}$, and this is the last property in the definition of $\cal{M}_{v, \ell}$ that we had yet to show holds for this choice of $G'$.
  \end{claimproof}

  The final \namecref{stmt:independence} that we need before the proof of \Cref{stmt:boundForEachSigmaInBfU} is the observation that the events $\cal{M}_{v, \ell}$ are independent for each $v \not \in U$ when conditioned on $\{G[U] = \tilde{G}\}$.
  This is a simple consequence of the conditioning causing $\cal{M}_{v, \ell}$ to depend only on the subgraph $G[v, U]$ corresponding to the edges between $v$ and $U$.
  Although \Cref{stmt:unionBoundOverSetsSAndColoursL} has $\{G[U] = \tilde{G}\}$ as an intersecting event, we can instead condition on it because its probability is non-zero.

  \begin{claim}\label{stmt:independence}
    For any set $A \subset U$, we have
    \begin{equation*}
      \PP\Big(\exists G' \subset G : (G', G) \in \bigcap_{v \in A} \cal{M}_{v, \ell} \, \Big \vert \, G[U] = \tilde{G} \Big) = \prod_{v \in A} \PP\Big(\exists G' \subset G: (G', G) \in \cal{M}_{v, \ell} \, \Big \vert \, G[U] = \tilde{G} \Big).
    \end{equation*}
  \end{claim}

  \begin{claimproof}
    First recall that $\cal{M}_{v, \ell}$ is defined in \eqref{eq:mainEventInsideKey} as
    \[
        \cal{M}_{v, \ell} = \left\{
      (G', G)
      :
      \begin{array}{@{}c@{}}
        G'[U] = \tilde{G}_\ell, ~ d_{G'}(v, U) \ge |U| / (4r) ~ \text{ and} \\
        \frakI_{H_\ell, G', G}\big[U \cup \{v\}\big] \,
        \text{ is not } \, (p, R_\ell + 1)\text{-Janson}
      \end{array}
      \right\}.
    \]
    Now, observe that after conditioning on $G[U]$, the existence of a subgraph $G'$ satisfying the three properties in the definition of $\cal{M}_{v, \ell}$ depends only on the edges of $G[v, U]$.
    Since these edges are chosen independently for each vertex $v \in A$, the \namecref{stmt:independence} follows.
  \end{claimproof}

  We can now combine the previous \namecrefs{stmt:independence} to prove \Cref{stmt:boundForEachSigmaInBfU}.

  \begin{claimproof}[Proof of \Cref{stmt:boundForEachSigmaInBfU}]
    Fix $\sigma = \big(S, U, (\tilde{G}_i)_{i \in [r]}\big) \in \mathbf{U}$ and let $\cal{A} = \cal{A}(\sigma)$.
    By \Cref{stmt:splitIntoCasesWrtSPrime}, we have
    \begin{equation}\label{eq:splitIntoCasesWrtSPrimeInsideProofOfFinalClaim}
      \PP\big(\exists c \in [r]^{E(G)} : (G, c) \in \cal{A} \big) \le \PP\big(\exists c \in [r]^{E(G)} : (G, c) \in \cal{A}' \big) + \exp\big(- 2^{-6} \delta^{5/3} N^2\big).
    \end{equation}
    \Cref{stmt:unionBoundOverSetsSAndColoursL} then implies that the first term in this right-hand side is at most
    \begin{equation}\label{eq:unionBoundInsideProofOfFinalClaim}
      \PP\big(\exists c \in [r]^{E(G)} : (G, c) \in \cal{A}' \big) \le \sum_{\ell \in [r]} \sum_{\substack{A \subset S \\ |A| \ge \frac{|S|}{4 r}}} \PP\Big(\exists G' \subset G : (G', G) \in \bigcap_{v \in A} \cal{M}_{v, \ell} \, \Big \vert \, G[U] = \tilde{G} \Big),
    \end{equation}
    where we moved the non-zero probability event $\{G[U] = \tilde{G}\}$ from the intersection to the conditioning, so each probability term inside the two sums of \eqref{eq:unionBoundInsideProofOfFinalClaim} satisfies
    \begin{equation}\label{eq:independenceOfNeighbourhoods}
      \PP\Big(\exists G' \subset G : (G', G) \in \bigcap_{v \in A} \cal{M}_{v, \ell} \, \Big \vert \, G[U] = \tilde{G} \Big) = \prod_{v \in A} \PP\Big(\exists G' : (G', G) \in \cal{M}_{v, \ell} \, \Big \vert \, G[U] = \tilde{G} \Big)
    \end{equation}
    as a consequence of \Cref{stmt:independence}.

    We now want to apply \Cref{stmt:extensionOfJansonCollection} to obtain an upper bound for the probabilities on the right-hand side of \eqref{eq:independenceOfNeighbourhoods}, recalling that \eqref{eq:mainEventInsideKey}, the definition of $\cal{M}_{v, \ell}$, corresponds to the event whose probability we bound in \eqref{eq:probOfGPrime} if we take $F = H_\ell$, $R' = R_\ell$ and $\tilde{G}' = \tilde{G}_\ell$.
    To check that the choice of parameters for this application is admissible, first observe that $s = v(H_\ell) - 1$ by assumption, that
    \begin{equation*}
      m = |U| \ge \delta N \ge \delta r^{C(k + t)} \ge r^{C k},
    \end{equation*}
    because $t \ge 1$, $C = 300$ and $\delta = r^{-50}$, and also that $v \not \in U$ by $v \in A \subset S \setminus U$.
    It follows from $(S, U, (\tilde{G}_i)_{i \in [r]}) \in \mathbf{U}$ and the definition of $\mathbf{U}$ that not only $$R_\ell \le \frac{p |U|}{2^9 r} = \frac{R}{16}$$ by \eqref{eq:bfUSetInequalities} and our choice of $R = 2^{-5} r^{-1} p |U|$ in \eqref{eq:paramsInExtensionOfJansonCollection}, but also that $\tilde{G}_\ell \subset \tilde{G}$ by \eqref{eq:defTildeG}.
    Finally, again by the properties of the tuples in $\mathbf{U}$, we have that
    \begin{enumerate}[\normalfont (1)]
      \item $\frakI_{H_\ell, \tilde{G}_\ell, \tilde{G}}[U]$ is $(p, R_\ell)$-Janson, and
        \smallskip
      \item $\frakI_{H_\ell^-, \tilde{G}_\ell, \tilde{G}}[W]$ is $(p, R)$-Janson for every $W \subset U$ with $|W| \ge |U|/(8r)$, because $$ p |W| \ge \frac{p |U|}{8 r} \ge 2^{-5} r^{-1} p |U| = R$$ and the Janson property is increasing (\Cref{stmt:jansonParametersMonotone}),
    \end{enumerate}
    so we can apply \Cref{stmt:extensionOfJansonCollection}.

    Applying \Cref{stmt:extensionOfJansonCollection}, we then obtain, for fixed $\ell \in [r]$ and $v \in S \setminus U$, that
    \begin{equation*}
      \PP\Big(\exists G' \subset G : (G', G) \in \cal{M}_{v, \ell} \, \Big \vert \, G[U] = \tilde{G} \Big) \le 2^{-|U|/(2^5 r)},
    \end{equation*}
    which, replaced in \eqref{eq:independenceOfNeighbourhoods}, yields
    \begin{equation}\label{eq:finalBoundForUnionBound}
      \PP\Big(\exists G' \subset G : (G', G) \in \bigcap_{v \in A} \cal{M}_{v, \ell} \, \Big \vert \, G[U] = \tilde{G} \Big) \le 2^{-|U| \/ |A| / (2^5 r)} \le 2^{-\delta^{5/3} N^2/(2^{7} r^2)}
    \end{equation}
    by $|A| \ge |S|/(4r) \ge \delta^{2/3} N / (4 r)$, where $A \subset S \setminus U$, and $|U| \ge \delta N$.
    Substituting \eqref{eq:finalBoundForUnionBound} in \eqref{eq:unionBoundInsideProofOfFinalClaim} and bounding the number of choices for $\ell$ and $A$ respectively by $r$ and $2^N$ then yields
    \begin{equation}\label{eq:thisBoundWasFinalButWeWantedToMakeItPrettier}
      \PP\big(\exists c \in [r]^{E(G)} : (G, c) \in \cal{A}' \big) \le r 2^N 2^{-\delta^{5/3} N^2/(2^{7} r^2)}.
    \end{equation}
    To deduce the \namecref{stmt:boundForEachSigmaInBfU}, we combine \eqref{eq:splitIntoCasesWrtSPrimeInsideProofOfFinalClaim} and \eqref{eq:thisBoundWasFinalButWeWantedToMakeItPrettier} to obtain
    \begin{equation*}
      \PP\big(\exists c \in [r]^{E(G)} : (G, c) \in \cal{A} \big) \le r 2^N 2^{-\delta^{5/3} N^2/(2^{7} r^2)} + \exp\big(- 2^{-6} \delta^{5/3} N^2\big) \le 2^{-8 r \delta^2 N^2}
    \end{equation*}
    by $r \ge 2$ and our assumptions that $N \ge r^{C (k + t)}$ and $\delta = r^{-50}$.
  \end{claimproof}

  We now apply \Cref{stmt:boundForEachSigmaInBfU} in every term of \eqref{eq:unionBound} to obtain
  \begin{equation}\label{eq:replacedInUnionBound}
    \PP\big(G \in \cal{B}'(\mathbf{H}) \cap \cal{E}(\mathbf{s})\big) \le \sum_{\sigma \in \mathbf{U}} 2^{- 8 r \delta^2 N^2}
  \end{equation}

  To bound the right-hand side of \eqref{eq:replacedInUnionBound}, we count the number of tuples $\sigma = \big(S, U, (\tilde{G}_i)_{i \in [r]}\big)$ in $\mathbf{U}$.
  There are at most $2^{2 N}$ choices for both $S \subset V$ and $U \subset V$, and at most $2^{|U|^2}$ choices for each $\tilde{G}_i$.
  It follows from $|U| \le 2 \delta N$ in \eqref{eq:sizeOfU} that there are at most
  $$2^{r |U|^2} \le 2^{4 r \delta^2 N^2}$$
  tuples $(\tilde{G}_i)_{i \in [r]}$, which replaced back in \eqref{eq:replacedInUnionBound} yields
  \begin{equation*}
    \PP\big(G \in \cal{B}'(\mathbf{H}) \cap \cal{E}(\mathbf{s})\big)
      \le 2^{2N + 4 r \delta^2 N^2 - 8 r \delta^2 N^2}
      \le 2^{- \delta^{2} N^2}
  \end{equation*}
  because we assumed that $N \ge r^{C (k + t)}$ and $\delta = r^{-50}$.
\end{proof}

\section{Containers for non-Janson sets}\label{sec:containersForNonJanson}

In this \namecref{sec:containersForNonJanson}, we prove our main technical result, and with it complete the proof of \Cref{stmt:multicolour}.
The statement of \Cref{stmt:containersForNonJanson} has three components that differ from \Cref{stmt:containersForNonJansonGeneral}: another hypergraph $\cal{F}$, which is $(p, R')$-Janson by assumption, a function $\pi$ and a vertex $v$ not in the set $U = V(\cal{F})$.
Recall that when applying this \namecref{stmt:containersForNonJanson}, we will take $\cal{F}$ to correspond to copies of $F$ completely contained in $U$, and $\pi$ to be the projection from $U \times \{0, 1\}$ onto $U$.
To explain the role of $v$ in the statement of \Cref{stmt:containersForNonJanson}, recall \Cref{def:edgeWiseInclusion},
\begin{equation*}
  \ilink{\cal{G}}{v} = \{E \cup \{v\} : E \in \cal{G}\},
\end{equation*}
the edge-wise inclusion of $v$ in a hypergraph $\cal{G}$.

Rather than obtaining containers for sets $L \subset V$ such that $\cal{H}[L]$ is not $(p / q, \eta R)$-Janson, \Cref{stmt:containersForNonJanson} provides containers for sets $L \subset V$ such that $\pi_v(\cal{H}[L]) \cup \cal{F}$ is not $(p, R' + \eta R)$-Janson, where $\pi_v = \ilink{\circ \pi}{v}$.
Another difference between this \namecref{stmt:containersForNonJanson} and \Cref{stmt:containersForNonJansonGeneral} is in the properties of the containers $X \in \cal{X}$.
Previously, we concluded that $\cal{H}[X]$ was not $(p, R)$-Janson, but we were not able to establish the same thing here due to vertices with high-degree.
Instead, what we show is that whenever $X \subset V$ has linear size, we have a set $Y \subset X$ containing almost all elements of $X$ such that $\pi(\cal{H}[Y])$ is not $(p, R)$-Janson.

\hypertarget{stmt:ContainersForNonJanson:restated:sec7}{\containersForNonJanson*}

We now highlight differences between the proof of \Cref{stmt:containersForNonJanson} and the one in \Cref{sec:simplerContainerForNonJanson}.
The first one is already in the auxiliary hypergraph $\cal{J}'$, which is defined here by
\begin{equation*}
  \cal{J}' = \Big\{ L \subset V : \pi_v\big(\cal{H}[L]\big) \cup \cal{F} \text{ is $(p, R'+\eta R)$-Janson} \Big\}.
\end{equation*}
Note that, as the Janson property is increasing by \Cref{stmt:JansonIsDownset}, each $L \subset V$ for which the hypergraph $\pi_v(\cal{H}[L]) \cup \cal{F}$ is not $(p, R' + \eta R)$-Janson is also an independent set in $\cal{J}'$.

The start of the proof of \Cref{stmt:containersForNonJanson} is analogous to the proof of \Cref{stmt:containersForNonJansonWithFingerprints}: we apply \Cref{stmt:containersHardcovers} with $\cal{G} = \cal{J}'$, define $\cal{C}'_T = \langle \cal{C}_T \rangle_{= s}$ for each $T \in \cal{T}$, and apply \Cref{stmt:containersCoversButJanson} with $\cal{G} = \cal{C}'_T$.
This application of \Cref{stmt:containersCoversButJanson} also provides our candidate containers $X$, and the bulk of the argument is proving that they fulfil the conclusions of the \namecref{stmt:containersForNonJansonWithFingerprints}, in particular that $\pi(\cal{H}[Y])$ is not $(p, R)$-Janson for some large $Y \subset X$.

Towards determining that $\cal{H}[X]$ was not $(p, R)$-Janson for fixed $X \in \cal{X}$, in the previous proof we took an arbitrary measure $\nu : \cal{H}[X] \to \RR_{\ge 0}$ and showed that
\begin{equation}\label{eq:originalGoalInEasierProof}
  \Lambda_p(\nu) \ge \frac{e(\nu)^2}{R},
\end{equation}
by considering two cases, depending on whether $e(\nu')$ was sufficiently large, where $\nu'$ was the restriction of $\nu$ to $\cal{C}'_T[X]$.
Our goal in \hyperlink{stmt:ContainersForNonJanson:restated:sec7}{\cref{item:containersForNonJansonAreNonJanson}} of \hyperlink{stmt:ContainersForNonJanson:restated:sec7}{\Cref{stmt:containersForNonJanson}} is to establish something like \eqref{eq:originalGoalInEasierProof} for measures $\mu$ supported on $\pi(\cal{H}[Y])$ for some large $Y \subset X$.
In order to do so, we will need to introduce some additional machinery, which will allow us to relate Janson properties of $\cal{H}$ and those of $\pi(\cal{H})$.

To reason about the Janson properties under the effect of $\pi$, we define the pullback of a measure $\vartheta$ with respect to $\pi$.
The resulting measure, denoted by $\vartheta  \normcomp \pi$, distributes the mass of $E \in \pi(\cal{G})$ equally among edges that are entirely contained in its pre-image.
\begin{defi}
  Let $\cal{G}$ be a hypergraph, $U$ be a set, and let $\pi : V(\cal{G}) \to U$.
  If $\vartheta : \pi(\cal{G}) \to \RR_{\ge 0}$ is a measure, then $\vartheta \normcomp \pi : \cal{G} \to \RR_{\ge 0}$ is defined by
  \begin{equation}\label{eq:defPullback}
    \vartheta \normcomp \pi(E) = \frac{\vartheta(\pi(E))}{\big|\{E_0 \in \cal{G} : \pi(E_0) = \pi(E)\}\big|}
  \end{equation}
  for all $E \in \cal{G}$.
\end{defi}

The crucial property of pullback measures is that, using them, we can show that for any hypergraph $\cal{G}$, if $\pi(\cal{G})$ is $(p, R)$-Janson, then so is $\cal{G}$.

\begin{restatable}{lem}{pullbackPreservesLambdaProperties}\label{stmt:pullbackPreservesLambdaProperties}
  Let $R > 0$ and $p > 0$.
  Further let $\cal{G}$ be a hypergraph, $U$ be a set, $\pi : V(\cal{G}) \to U$ be a function satisfying
  \begin{equation*}
    |\pi(E)| = |E| \qquad \text{for every } E \in \cal{G},
  \end{equation*}
  and $\vartheta : \pi(\cal{G}) \to \RR_{\ge 0}$ be a measure.
  If
  \begin{equation*}
    \Lambda_p(\vartheta) < \frac{e(\vartheta)^2}{R},
  \end{equation*}
  then
  \begin{equation*}
    \Lambda_p(\vartheta \normcomp \pi) < \frac{e(\vartheta \normcomp \pi)^2}{R}.
  \end{equation*}
  In particular, if $\pi(\cal{G})$ is $(p, R)$-Janson, then so is $\cal{G}$.
\end{restatable}

As the proof of \Cref{stmt:pullbackPreservesLambdaProperties} is just checking that the definitions fit nicely together, we will postpone it, and the proofs of intermediate results that we require, to Appendix~\ref{app:propertiesOfMeasures}.
Nevertheless, we reference two of those results, \Cref{stmt:pullbackIsUniform} and \Cref{stmt:edgesPullback}, because they will also be useful in the proof of \Cref{stmt:containersForNonJanson}.

\begin{restatable}{obs}{pullbackIsUniform}\label{stmt:pullbackIsUniform}
  Let $\cal{G}$ be a hypergraph and let $\pi : V(\cal{G}) \to U$ for some set $U$.
  Further let $\vartheta : \pi(\cal{G}) \to \RR_{\ge 0}$ be a measure.
  If $E' \in \pi(\cal{G})$, then
  \begin{equation*}
    \sum_{\substack{E \in \cal{G} \\ \pi(E) = E'}} \vartheta \normcomp \pi(E) = \vartheta(E').
  \end{equation*}
\end{restatable}

\Cref{stmt:pullbackIsUniform} almost immediately implies the second useful result about pullback measures that we use in the proof of \Cref{stmt:containersForNonJanson}.

\begin{restatable}{lem}{edgesPullback}\label{stmt:edgesPullback}
  Let $\cal{G}$ be a hypergraph.
  For all $\pi : V(\cal{G}) \to U$ and $\vartheta : \pi(\cal{G}) \to \RR_{\ge 0}$, we have
  \begin{equation*}
    e(\vartheta \normcomp \pi) = e(\vartheta).
  \end{equation*}
\end{restatable}

With the definition of the pullback measure and its crucial property, we can state the main inequalities in the proof that $\pi(\cal{H}[Y])$ is not $(p, R)$-Janson.
We will start with a simplified view of the proof, and add details as we proceed.
For technical reasons, it will be useful to assume instead the converse of our goal, i.e.\ we will fix one $X \in \cal{X}$ such that, for all large $Y$, the hypergraph $\pi(\cal{H}[Y])$ is $(p, R)$-Janson, and reach a contradiction.
In particular, $\pi(\cal{H}[X])$ is $(p, R)$-Janson, so there is a measure $\mu : \pi(\cal{H}[X]) \to \RR_{\ge 0}$ such that
\begin{equation}\label{eq:overview:muSatisfiesLambdaIneq}
  \Lambda_p(\mu) < \frac{e(\mu)^2}{R}
\end{equation}
and the pullback $\nu = \mu \normcomp \pi$ satisfies
\begin{equation*}
  \Lambda_p(\nu) < \frac{e(\nu)^2}{R}
\end{equation*}
by \Cref{stmt:pullbackPreservesLambdaProperties}.

Recall that a critical step in the proof of the simpler version of our container \namecref{stmt:containersForNonJansonWithFingerprints} was the inequality
\begin{equation}\label{eq:recallLambdaPropertiesInExpectation}
  \EE\big[\Lambda_{p / q}(\nu''_q) \mid V_q \in \cal{I}(\link{\cal{J}}{T})\big] \ge \frac{\EE\big[e(\nu''_q)^2 \mid V_q \in \cal{I}(\link{\cal{J}}{T})\big]}{\eta R},
\end{equation}
established in \Cref{stmt:lambdaPropertiesWithConditionalExpectation} via the characterization of independent sets in $\cal{J}$.
The analogous statement here depends on the definition of $\cal{J}'$, which implies that when $I \subset V(\cal{J}')$ is independent, then $\pi_v\big(\cal{H}[I]\big) \cup \cal{F}$ is not $(p, R' + \eta R)$-Janson.
To avoid too many technical details at once, let us assume that $R' > 0$, and that we have $\pi(\cdot)$ instead of $\pi_v(\cdot)$.
These simplifications mean that $\pi\big(\cal{H}[I]\big) \cup \cal{F}$ is not $(p, R' + \eta R)$-Janson when $I \in \cal{I}(\cal{J}')$.

In this simpler setup, we now use the assumption that $\cal{F}$ is $(p, R')$-Janson and $R' > 0$ to choose $\rho : \cal{F} \to \RR_{\ge 0}$ with
\begin{equation}\label{eq:overview:rhoSatisfiesLambdaIneq}
  \Lambda_p(\rho) < \frac{e(\rho)^2}{R'}.
\end{equation}
A case analysis, the same as in \Cref{stmt:nuPrimeHasLessThanHalfTheMeasureOfNu} in the proof of \Cref{stmt:containersForNonJansonWithFingerprints}, will allow us to focus on $\cal{H}' = \cal{H}[X]\setminus\cal{C}_T'$, so we define $\mu_q : \pi(\cal{H}[X]) \to \RR_{\ge 0}$ by
\begin{equation*}
  \mu_q(E) =  \frac{\gamma \cdot \mathds{1}\big[E \in \pi\big(\cal{H}'[V_q]\big)\big]}{P'_q(E)} \, \mu(E),
\end{equation*}
where we will choose $\gamma = \sqrt{8\eta}$ and
\begin{equation*}
  P'_q(E) = \PP\big(E \in \pi\big(\cal{H}'[V_q]\big) \mid V_q \in \cal{I}(\link{\cal{J}'}{T})\big).
\end{equation*}
The definition of $\cal{H}'$ will allow us to assume that
\begin{equation*}
  P'_q(E) > \Big(\frac{\,q\,}{2}\Big)^{|E|},
\end{equation*}
cf.\ \eqref{eq:lowerBoundOnPOfE}, so $\mu_q$ is well-defined.

The inequality corresponding to \eqref{eq:recallLambdaPropertiesInExpectation} in this simplified overview of the proof will then be
\begin{equation}\label{eq:lambdaPropertiesInExpectationEasier}
  \EE\big[\Lambda_p(\rho + \mu_q) \mid V_q \in \cal{I}(\link{\cal{J}'}{T})\big] \ge \frac{\EE\big[e(\rho + \mu_q)^2 \mid V_q \in \cal{I}(\link{\cal{J}'}{T})\big]}{R' + \eta R},
\end{equation}
which, note, has $\Lambda_p$ instead of $\Lambda_{p / q}$ due to some still undiscussed numerics related to the fact that, in this case, the value of $\eta$ is much smaller here than in \Cref{stmt:containersForNonJansonGeneral}.
In the other direction, we would ideally like to show that
\begin{equation}\label{eq:overview:lambdaUpperBound}
  \EE\big[\Lambda_p(\rho + \mu_q) \mid V_q \in \cal{I}(\link{\cal{J}'}{T})\big] \le \Lambda_p(\rho) + 2 \gamma \sqrt{\Lambda_p(\rho) \Lambda_p(\mu)} + \gamma^2 \Big(\frac{\,q\,}{2}\Big)^{-2 s} \Lambda_p(\mu),
\end{equation}
and
\begin{equation}\label{eq:overview:massLowerBound}
  \EE\big[e(\rho + \mu_q)^2 \mid V_q \in \cal{I}(\link{\cal{J}'}{T})\big] \ge  \big( e(\rho) + \gamma e(\mu)\big)^2,
\end{equation}
both of which, when combined with \eqref{eq:overview:muSatisfiesLambdaIneq}, \eqref{eq:overview:rhoSatisfiesLambdaIneq} and our choice of $\gamma$, yield
\begin{equation*}
  \EE\big[\Lambda_p(\rho + \mu_q) \mid V_q \in \cal{I}(\link{\cal{J}'}{T})\big] < \frac{\EE\big[e(\rho + \mu_q)^2 \mid V_q \in \cal{I}(\link{\cal{J}'}{T})\big]}{R' + \eta R}
\end{equation*}
a direct contradiction of \eqref{eq:lambdaPropertiesInExpectationEasier}.
Although we can establish inequalities resembling \eqref{eq:overview:lambdaUpperBound} and \eqref{eq:overview:massLowerBound} with methods similar to those in \Cref{sec:simplerContainerForNonJanson}, we have not justified some of our assumptions.

The first detail that we overlooked in the preceding overview is the assumption that $R' > 0$.
In the case $R' = 0$, we can simply take $\rho = 0$.
To handle both cases together, we use \Cref{stmt:normalisedJansonWitnesses} when $R' > 0$ and assume instead that
\begin{equation*}
  e(\rho) = \sqrt{R'} \qquad \text{and} \qquad \Lambda_p(\rho) < 1,
\end{equation*}
which is sufficient for our purposes.

The other omission in our discussion so far is that we simplified the definition of $\cal{J}'$.
The correct definition means that, to satisfy something like \eqref{eq:lambdaPropertiesInExpectationEasier}, we require a measure supported on $\pi_v\big(\cal{H}[I]\big) \cup \cal{F}$, instead of $\pi\big(\cal{H}[I]\big) \cup \cal{F}$, where $I \in \cal{I}(\cal{J}')$.
To address this change, we extend $\mu : \pi\big(\cal{H}[X]\big) \to \RR_{\ge 0}$ to a measure $\bar{\mu} : \pi_v\big(\cal{H}[X]\big) \to \RR_{\ge 0}$ by setting
\begin{equation*}
  \bar{\mu}\big(E \cup \{v\}\big) = \mu(E)
\end{equation*}
for all $E \in \pi(\cal{H})$, recalling that $v \not \in U \supset \pi(V)$.

Adjusting for this seemingly innocuous change requires some technicalities.
To prove the appropriate version of \eqref{eq:lambdaPropertiesInExpectationEasier}, with $\bar{\mu}_q$ replacing $\mu_q$, we now need an upper bound for $\sum_{u \in \pi(X)} d_\mu(u)^2$.
This is how we use \Cref{stmt:boundDegreeSquared} below, and is the main reason why we prove \hyperlink{stmt:ContainersForNonJanson:restated:sec7}{\cref{item:containersForNonJansonAreNonJanson}} in \hyperlink{stmt:ContainersForNonJanson:restated:sec7}{\Cref{stmt:containersForNonJanson}} for a large $Y \subset X$ instead of all of $X \in \cal{X}$; it is also why it is simpler to prove that same \namecref{item:containersForNonJansonAreNonJanson} by contradiction.
When we account for this term in the final version of \eqref{eq:lambdaPropertiesInExpectationEasier}, it dominates the term for $\Lambda_p(\nu)$, so a simpler bound for the latter suffices, and we do not need to consider $\Lambda_{p / q}$.

The proof of \Cref{stmt:boundDegreeSquared} is straightforward: we simply restrict the measure to avoid vertices with high degree.
If we were only dealing with hypergraphs then implementing this idea would be a triviality, but working with measures involves a few tedious calculations, so we postpone the proof of this \namecref{stmt:boundDegreeSquared} to Appendix~\ref{app:propertiesOfMeasures}.

\begin{restatable}{lem}{boundDegreeSquared}\label{stmt:boundDegreeSquared}
  Let $s \in \NN$, $R, p, \beta> 0$, and $\cal{G}$ be an $s$-uniform hypergraph.
  If for every $W \subset V(\cal{G})$ with $|W| \ge (1 - \beta) v(\cal{G})$, we have that $\cal{G}[W]$ is $(p, R)$-Janson, then there exists $\mu : \cal{G} \to \RR_{\ge 0}$ with
  \begin{equation}
    e(\mu) = \sqrt{R}, \qquad \Lambda_p(\mu) < \frac{e(\mu)^2}{R} \qquad \text{and} \qquad \sum_{v \in V(\cal{G})} d_\mu(v)^2 \le \frac{2 s^2 e(\mu)^2}{\beta v(\cal{G})}.
  \end{equation}
\end{restatable}

We are now ready to prove \hyperlink{stmt:ContainersForNonJanson:restated:sec7}{\Cref{stmt:containersForNonJanson}}; for ease of reference, let us restate \Cref{stmt:containersHardcovers}.

\hypertarget{stmt:containersHardcovers:restated}{\containersHardcovers*}

\begin{proof}[Proof of \Cref{stmt:containersForNonJanson}]
  \hypertarget{proof:containersForNonJanson}{Apply} \hyperlink{stmt:containersHardcovers:restated}{\Cref{stmt:containersHardcovers}} with $\cal{G} = \cal{J}'$, where
  \begin{equation*}
    \cal{J}' = \Big\{ L \subset V : \pi_v\big(\cal{H}[L]\big) \cup \cal{F} \text{ is $(p, R' + \eta R)$-Janson}\Big\},
  \end{equation*}
  and parameters $q$ and $\alpha = 1/2$ to obtain $\cal{T}$ and $\varphi$.
  Now, for each $T \in \cal{T}$, there is $\cal{C}_T$ satisfying \hyperlink{stmt:containersHardcovers:restated}{\cref{item:containersHardcoversAreCoverable} in \Cref{stmt:containersHardcovers}}, so we let $\cal{C}'_T = \langle \cal{C}_T \rangle_{= s}$ be the edges of $\langle \cal{C}_T \rangle$ of size $s$.
  Since $\cal{C}'_T$ is $s$-uniform and $p \le 1/(2^{11} r s^2)$ by assumption, we can apply \Cref{stmt:containersCoversButJanson} with $\cal{G} = \cal{C}'_T$ and $\zeta = 2^{-8} r^{-1}$.
  As in the proof of \Cref{stmt:containersForNonJansonWithFingerprints} we obtain $\varphi$, $\psi_T$, $\phi_T$ that satisfy
  \begin{enumerate}[(i)]
    \item For each $I \in \cal{I}(\cal{J}')$, we have $\phi_T(I) \subset I \subset \psi_T(\phi_T(I))$, for $T = \varphi(I)$. \label{item:containersCoversHaveFingerprints:restated}
    \item Each $S \in \cal{S}_T$ has at most $2^{11} r p s^2 n$ elements. \label{item:containersCoversHaveSmallFingerprints:restated}
    \hypertarget{item:containersAreNonJanson:restated}{\item For every $S \in \cal{S}_T$, letting $X = \psi_T(S)$, $\cal{C}'_T[X]$ is not $(p, 2^{-8} r^{-1} p |X|)$-Janson.}
  \end{enumerate}

  Setting $f(I) = X = \psi_T(S)$ for $T = \varphi(I)$ and $S = \phi_T(I)$, we define
  \begin{equation*}
    \cal{X} = \big\{f(I) : I \in \cal{I}(\cal{J})\big\},
  \end{equation*}
  which, by \cref{item:containersCoversHaveFingerprints:restated} and the fact that $I$ was arbitrary, is a definition that satisfies \hyperlink{stmt:ContainersForNonJanson:restated:sec7}{\cref{item:containersForNonJansonInclusions}} in \hyperlink{stmt:ContainersForNonJanson:restated:sec7}{\Cref{stmt:containersForNonJanson}}.

  As in the proofs of \Cref{stmt:containersForNonJansonGeneral} and \Cref{stmt:containersForNonJansonWithFingerprints} (cf.\ \eqref{eq:sizeOfContainerReferenceForTheFirstTime}, \eqref{eq:boundOnTheSizesOfFingerprintsForBinomial} and \eqref{eq:sizeOfFingerprintsInGeneralContainerJanson}), to show that the size of $\cal{X}$ is suitable, we count $X \in \cal{X}$ by choosing $T \in \cal{T}$ and then $S \in \cal{S}_T$.
  We therefore obtain
  \begin{equation}\label{eq:sizeOfContainerFamilyIsSuitable}
    |\cal{X}| \le \sum_{T \in \cal{T}} |\cal{S}_T| \le \sum_{m = 0}^{2^{11} p s^2 n} \binom{n}{m} \sum_{t = 0}^{2 q n} \binom{n}{t} \le \bigg(\frac{\,2\,}{q}\bigg)^{2 q n}
  \end{equation}
  combining \cref{item:containersCoversHaveSmallFingerprints:restated} above with \hyperlink{stmt:containersHardcovers:restated}{\cref{item:containersHardcoversHaveSmallFingerprints} in \Cref{stmt:containersHardcovers}} and $$2^{11}r p s^{2}n \le 2qn \le \frac{n}{4}.$$
  The bound in \eqref{eq:sizeOfContainerFamilyIsSuitable} matches \hyperlink{stmt:ContainersForNonJanson:restated:sec7}{\eqref{eq:containersForNonJansonSmallFamily}} in the statement, so it only remains to show that \hyperlink{stmt:ContainersForNonJanson:restated:sec7}{\cref{item:containersForNonJansonAreNonJanson}} holds.

  Now, assume by contradiction that there exists $I \in \cal{I}(\cal{J}')$ such that
  \begin{equation}
    \begin{gathered}
      f(I) = X \in \cal{X} \text{ satisfies } |X| \ge n/(8r), \text{ and} \\
      \pi\big(\cal{H}[Y]\big) \text{ is } (p, R)\text{-Janson} \text{ for every }
      Y \subset X \text{ with } |Y| \ge |X| - 2^{-8}r^{-1}n.
    \end{gathered}
    \label{item:negatedItem}\tag{$\ast$}
  \end{equation}
  This is the converse of \hyperlink{stmt:ContainersForNonJanson:restated:sec7}{\cref{item:containersForNonJansonAreNonJanson}} in \hyperlink{stmt:ContainersForNonJanson:restated:sec7}{\Cref{stmt:containersForNonJanson}}, so contradicting it will complete the proof.

  We want to apply \Cref{stmt:boundDegreeSquared} with $\cal{G} = \pi\big(\cal{H}[X]\big)$ and $\beta = 2^{-9}r^{-1}$.
  To do that, we first verify that this hypergraph satisfies the assumptions in the \namecref{stmt:boundDegreeSquared}.

  \begin{claim}
    For all
    \begin{equation}\label{eq:largePreImages:propertiesOfW}
      W \subset \pi(X) \qquad \text{with} \qquad |W| \ge (1 - 2^{-9}r^{-1}) |\pi(X)|,
    \end{equation}
    the hypergraph $\pi\big(\cal{H}[X]\big)[W]$ is $(p, R)$-Janson.
  \end{claim}

  \begin{claimproof}
    Fix $W$ satisfying \eqref{eq:largePreImages:propertiesOfW} and observe that letting $Y = \pi^{-1}(W) \cap X$, we have
    \begin{equation*}
      \pi\big(\cal{H}[Y]\big) \subset \pi\big(\cal{H}[Y]\big)[W]\subset \pi\big(\cal{H}[X]\big)[W]
    \end{equation*}
    where the first containment holds because $\pi(E)\subset W$, for all $E\subset Y$, by our choice of $Y\subset \pi^{-1}(W)$ and the second holds since $Y\subset X$.

    As $\pi(X \setminus Y) = \pi(X) \setminus W$ and we assumed that $|L| \le 2|\pi(L)|$ for every $L \subset V$, we have that
    \begin{equation}\label{eq:sizeOfY}
      |Y| = |X| - |X\setminus Y| \ge |X| - 2|\pi(X\setminus Y)| = |X| - 2|\pi(X)\setminus W| \ge |X| - 2^{-8}r^{-1}n.
    \end{equation}
    It follows from \eqref{eq:sizeOfY} and \eqref{item:negatedItem} that $\pi\big(\cal{H}[Y]\big)$ is $(p, R)$-Janson, but this is an increasing property by \Cref{stmt:JansonIsDownset}, so $\pi\big(\cal{H}[X]\big)[W]$ is also $(p, R)$-Janson.
  \end{claimproof}

  Applying \Cref{stmt:boundDegreeSquared} with $\cal{G} = \pi\big(\cal{H}[X]\big)$ and $\beta = 2^{-9}r^{-1}$, we obtain $\mu : \pi\big(\cal{H}[X]\big) \to \RR_{\ge 0}$ satisfying
  \begin{equation}\label{eq:muIsJanson}
    e(\mu) = \sqrt{R}, \qquad \Lambda_p(\mu) < 1
  \end{equation}
  and
  \begin{equation}\label{eq:muHasBddDegrees}
    \sum_{u \in \pi(X)} d_\mu(u)^2 \le \frac{2 s^2 e(\mu)^2}{\beta |\pi(X)|} \le \frac{2^{10}r s^2 e(\mu)^2}{ |\pi(X)|} \le \frac{2^{14}r^2 s^2 e(\mu)^2}{n},
  \end{equation}
  where the last inequality follows from the assumptions that $|L|/2 \le |\pi(L)|$ for every $L \subset V$ and $|X| \ge n /(8r)$.

  Let $\nu : \cal{H}[X] \to \RR_{\ge 0}$ be the pullback measure of $\mu$ with respect to $\pi$, i.e.\ $\nu = \mu \normcomp \pi$.
  As $\pi$ satisfies
  \begin{equation*}
    |\pi(E)| = |E| \qquad \text{for all } E \in \cal{H}[X] \subset \cal{H}
  \end{equation*}
  by \hyperlink{stmt:ContainersForNonJanson:restated:sec7}{\eqref{eq:assumptionsPi}}, we can apply \Cref{stmt:pullbackPreservesLambdaProperties} with $\cal{G} = \cal{H}[X]$ and $\vartheta = \mu$ to conclude that
  \begin{equation}\label{eq:nuSatisfiesLambdaProperties}
    \Lambda_p(\nu) < \frac{e(\nu)^2}{R}.
  \end{equation}

  Now take $T = \varphi(I)$, and recall that $\cal{C}'_T = \langle \cal{C}_T \rangle_{= s}$, where $\cal{C}_T$ is the cover given by \hyperlink{stmt:containersHardcovers:restated}{\cref{item:containersHardcoversAreCoverable} in \Cref{stmt:containersHardcovers}}.
  Let $\nu'$ be the restriction of the measure $\nu$ to $\cal{H}[X] \cap \cal{C}'_T[X]$, that is, for each $E \in \cal{H}[X]$, let
  \begin{equation*}
    \nu'(E) = \begin{cases*}
      \nu(E) \quad \text{if } E \in \cal{C}'_T[X], \\
      0 \phantom{(E)} \quad \text{otherwise.}
    \end{cases*}
  \end{equation*}

  \begin{claim}\label{stmt:nuPrimeHasAtMostHalfTheMeasureOfNu}
    The measure $\nu'$ satisfies
    \begin{equation*}
      e(\nu') < \frac{e(\nu)}{2}.
    \end{equation*}
  \end{claim}

  \begin{claimproof}
    Indeed, we have
    \begin{equation*}
      \frac{e(\nu)^2}{R} > \Lambda_p(\nu) \ge \Lambda_{p}(\nu') \ge \frac{2^8 r e(\nu')^2}{p |X|} \ge \frac{4 e(\nu')^2}{R}
    \end{equation*}
    first by \eqref{eq:nuSatisfiesLambdaProperties}, second because $\nu \ge \nu'$ and $\Lambda_p(\cdot)$ is monotone increasing, then since the hypergraph $\cal{C}'_T[X]$ is not $(p, 2^{-8} r^{-1} p |X|)$-Janson by \hyperlink{item:containersAreNonJanson:restated}{\cref{item:containersAreNonJanson} of \Cref{stmt:containersCoversButJanson}} and our choice of $\zeta = 2^{-8} r^{-1}$, and finally because $R = 2^{-6} r^{-1} p n$ by assumption.
  \end{claimproof}

  We now define the measure $\nu'' = \nu - \nu'$, which corresponds to the restriction of $\nu$ to the hypergraph ${\cal{H}' := \cal{H}[X] \setminus \cal{C}'_T} = \cal{H}[X] \setminus \cal{C}'_T[X]$.
  By \Cref{stmt:nuPrimeHasAtMostHalfTheMeasureOfNu}, we have
  \begin{equation}\label{eq:containersForNonJanson:nuDoublePrimeHasHalfTheMeasureOfNu}
    e(\nu'') = e(\nu) - e(\nu') > \frac{e(\nu)}{2}.
  \end{equation}
  Also define $\mu''$ to be the restriction of $\mu$ to $\pi(\cal{H}')$.
  Applying \Cref{stmt:pullbackIsUniform} with $\vartheta = \mu$ and $\nu = \mu \normcomp \pi$, then using \eqref{eq:containersForNonJanson:nuDoublePrimeHasHalfTheMeasureOfNu} and \Cref{stmt:edgesPullback}, yields
  \begin{equation}\label{eq:muDoublePrimeHasHalfTheMeasureOfMu}
    e(\mu'') = \sum_{E \in \pi(\cal{H}')} \mu(E) = \sum_{E \in \pi(\cal{H}')} \sum_{\substack{E_0 \in \cal{H} \\ \pi(E_0) = E}} \nu(E_0) \ge \sum_{E \in \cal{H}'} \nu''(E) = e(\nu'') > \frac{e(\nu)}{2} = \frac{e(\mu)}{2}.
  \end{equation}

  With the goal of defining a random measure $\mu''_q$, let, for all $E \in \pi(\cal{H}')$,
  \begin{equation}\label{eq:definitionOfPEPrime}
    P'_q(E) = \PP\big(E \in \pi(\cal{H}'[V_q]) \mid V_q \in \cal{I}(\link{\cal{J'}}{T})\big).
  \end{equation}

  \begin{claim}\label{stmt:lowerBoundOnPOfEPrime}
    For all $E \in \pi(\cal{H}')$,
    \begin{equation}\label{eq:lowerBoundOnPOfEPrime}
      P'_q(E) \ge \PP\big(E_0 \subset V_q \mid V_q \in \cal{I}(\link{\cal{J}'}{T})\big) > \Big(\frac{\,q\,}{2}\Big)^s
    \end{equation}
    where $E_0 \in \cal{H}'$ is fixed and satisfies $\pi(E_0) = E$.
  \end{claim}

  \begin{claimproof}
    The first inequality in \eqref{eq:lowerBoundOnPOfEPrime} follows from the fact that if $E_0\subset V_q$ for $E_0\in \cal{H}'$, then we have $\pi(E_0)\in \pi(\cal{H}')[V_q]$.
    Also notice that, by the same argument used to establish \eqref{eq:lowerBoundOnPOfE}, the hypergraphs $\cal{H}'$ and $\cal{C}_T$ are disjoint.
    Therefore, we can use \hyperlink{stmt:containersHardcovers:restated}{\eqref{eq:conditionalProbOfLInQSet}} in \hyperlink{stmt:containersHardcovers:restated}{\cref{item:containersHardcoversAreCoverable} of \Cref{stmt:containersHardcovers}} to obtain
    \[\PP\big(E_0 \subset V_q \mid V_q \in \cal{I}(\link{\cal{J}'}{T})\big) > \Big(\frac{\,q\,}{2}\Big)^s.
    \vspace{-15pt}
    \]
  \end{claimproof}

  Now, let $\gamma = \sqrt{8 \eta} > 0$ (a choice made with foresight) and $\mu_q'' : \pi(\cal{H}') \to \RR_{\ge 0}$ be defined by
  \begin{equation*}
    \mu''_q(E) = \mu''(E) \, \frac{\gamma \cdot \mathds{1}\big[E \in \pi(\cal{H}'[V_q])\big] }{P'_q(E)}.
  \end{equation*}
  Also define the extension $\bar{\mu}''_q$ supported on $\pi_v(\cal{H}'[X_q])$, where $X_q = V_q \cap X$, by
  \begin{equation*}
    \bar{\mu}''_q(E \cup \{v\}) = \mu''_q(E) \qquad \text{for all } E \in \pi(\cal{H}'[X_q]).
  \end{equation*}
  Our goal now is to show that analysing $\bar{\mu}''_q$ for our choice of $\gamma$ contradicts $I \in \cal{I}(\cal{J}')$.

  To reason about properties of $\cal{J}'$, we define a measure $\rho$ supported on $\cal{F}$.
  If $R' > 0$, then we can apply \Cref{stmt:normalisedJansonWitnesses} with $\cal{G} = \cal{F}$ and $y = \sqrt{R'}$, since $\cal{F}$ is $(p, R')$-Janson, to obtain $\rho : \cal{F} \to \RR_{\ge 0}$ satisfying
  \begin{equation}\label{eq:rhoIsJanson}
    \Lambda_p(\rho) < \dfrac{e(\rho)^2}{R'} \qquad \text{and} \qquad e(\rho) = \sqrt{R'}.
  \end{equation}
  If, on the other hand, $R' = 0$, then we take the measure
  \begin{equation}\label{eq:rhoIsZero}
    \rho = 0.
  \end{equation}
  Regardless of the value of $R'$, or the choice of $\rho$ as either \eqref{eq:rhoIsJanson} or \eqref{eq:rhoIsZero}, we have
  \begin{equation}\label{eq:rhoIsNormalisedJanson}
    \Lambda_p(\rho) < 1 \qquad \text{ and } \qquad e(\rho) = \sqrt{R'},
  \end{equation}
  which, together on being supported on $\cal{F}$, are the only properties that we will use of $\rho$.

  Our goal is to obtain inequalities bounding $\Lambda_p(\rho + \bar{\mu}_q'')$ from $e(\rho + \bar{\mu}_q'')^2$, so we start easily, relating the expected value of $e(\bar{\mu}''_q)$ to $e(\mu'')$.
  \begin{claim}\label{stmt:massOfBarMuDoublePrimeEqualsInExpectation}
    \begin{equation*}
      \EE\big[e(\bar{\mu}''_q) \mid V_q \in \cal{I}(\link{\cal{J}}{T})\big] = \gamma e(\mu'').
    \end{equation*}
  \end{claim}

  \begin{claimproof}
    The definitions of $e(\bar{\mu}''_q)$, $\bar{\mu}''_q$ and $\mu''_q$,
    \begin{equation*}
      e(\bar{\mu}''_q) = \sum_{E \cup \{v\} \in \pi_v(\cal{H}')} \bar{\mu}''_q\big(E \cup \{v\}\big) = \sum_{E \in \pi(\cal{H}')} \mu''(E) \, \frac{\gamma \cdot \mathds{1}\big[E \in \pi(\cal{H}'[V_q])\big]}{P'_q(E)},
    \end{equation*}
    imply that
    \begin{equation*}
      \EE\big[e(\bar{\mu}''_q) \mid V_q \in \cal{I}(\link{\cal{J}'}{T})\big] = \sum_{E \in \pi(\cal{H}')} \gamma \mu''(E) =  \gamma e(\mu''),
    \end{equation*}
    since, for all $E \in \pi(\cal{H}')$, we have that
    \begin{equation}\label{eq:indicatorAndPEPrime}
      \EE\big[\mathds{1}\big[E \in \pi(\cal{H}'[V_q])\big] \mid V_q \in \cal{I}(\link{\cal{J}'}{T})\big] = P'_q(E)
    \end{equation}
    from the definition of $P'_q(E)$, \eqref{eq:definitionOfPEPrime}.
  \end{claimproof}

  Combining \Cref{stmt:massOfBarMuDoublePrimeEqualsInExpectation} with \eqref{eq:muDoublePrimeHasHalfTheMeasureOfMu}, we have
  \begin{equation}\label{eq:boundone(mubard)}
    \EE\big[e(\bar{\mu}''_q) \mid V_q \in \cal{I}(\link{\cal{J}'}{T})\big] > \gamma \frac{e(\mu)}{2}.
  \end{equation}

  Recall that the definition of $d_\rho(L)$ for any set $L \subset U$ is
  \begin{equation*}
    d_\rho(L) = \sum_{L \subset E \in \cal{F}} \rho(E).
  \end{equation*}
  As $\rho$ is supported over $\cal{F}$, we have $\rho(E) = 0$ when $E \not \subset U = V(\cal{F})$.
  We can use this observation to obtain the expansion
  \begin{equation}\label{eq:changeFromUCupVToU}
    \Lambda_p(\rho + \bar{\mu}''_q) = \sum_{\substack{L \subset U \\ |L| \ge 2}} d_\rho(L)^2 p^{-|L|} + 2 \sum_{\substack{L \subset U \\ |L| \ge 2}} d_\rho(L) d_{\bar{\mu}''_q}(L) p^{-|L|} + \sum_{\substack{L \subset U \cup \{v\} \\ |L| \ge 2}} d_{\bar{\mu}''_q}(L)^2 p^{-|L|}
  \end{equation}
  where the first two terms range over $L \subset U$ instead of $L \subset U \cup \{v\}$ because if $v \in L$, then $d_\rho(L) = 0$ as $v \not \in V(\cal{F})$ by assumption.
  The first sum in \eqref{eq:changeFromUCupVToU} is now exactly $\Lambda_p(\rho)$, which we can immediately bound with \eqref{eq:rhoIsNormalisedJanson}, but we need some simple claims before we analyse the other two sums.
  We start with an observation that relates the degrees $d_{\bar{\mu}''_q}$ to the degrees $d_{\mu''_q}$.
  \begin{claim}\label{stmt:degreeOfBarMuC}
    For all $L \subset U \cup \{v\}$, we have
    \begin{equation*}
      d_{\bar{\mu}''_q}(L) = d_{\mu''_q}(L \setminus \{v\}).
    \end{equation*}
  \end{claim}

  \begin{claimproof}
    Let $E_v = E \cup \{v\}$ for each $E\in \pi(\cal{H}')$.
    The definitions of $d_{\bar{\mu}''_q}$ and $\bar{\mu}''_q$ imply that, for every $L \subset U \cup \{v\}$,
    \begin{equation*}
      d_{\bar{\mu}''_q}(L) = \sum_{L \subset E_v \in \pi_v(\cal{H}')} \bar{\mu}''_q(E_v) = \sum_{L \subset E_v \in \pi_v(\cal{H}')} \mu''_q(E_v \setminus \{v\}) = \sum_{L \setminus \{v\} \subset E \in \pi(\cal{H}')} \mu''_q(E) = d_{\mu''_q}(L \setminus \{v\}),
    \end{equation*}
    where the last equality is the definition of $d_{\mu''_q}(L \setminus \{v\})$.
  \end{claimproof}

  We can now use \Cref{stmt:degreeOfBarMuC} to relate $\Lambda_p(\bar{\mu}''_q)$ and $\Lambda_p(\mu''_q)$.
  \begin{claim}\label{stmt:lambdaBarMuQEqualityLambdaMuQ}
    \begin{equation}\label{eq:lambdaBarMuQEqualityLambdaMuQ}
      \Lambda_p(\bar{\mu}''_q) = \left(1 + \frac{\,1\,}{p}\right) \Lambda_p(\mu''_q) + \frac{1}{p^2} \sum_{u \in U} d_{\mu''_q}(u)^2.
    \end{equation}
  \end{claim}

  \begin{claimproof}
    First, recall the definition of $\Lambda_p(\bar{\mu}''_q)$,
    \begin{equation*}
      \Lambda_p(\bar{\mu}''_q) = \sum_{\substack{L \subset U \cup \{v\} \\ |L| \ge 2}} d_{\bar{\mu}''_q}(L)^2 p^{-|L|}.
    \end{equation*}
    Splitting that sum according to whether $v \in L$ or not, we obtain
    \begin{equation}\label{eq:lambdaBarMuQEqualityLambdaMuQStep}
      \Lambda_p(\bar{\mu}''_q) = \sum_{\substack{L \subset U \\ |L| \ge 2}} d_{\mu''_q}(L)^2 p^{-|L|} + \sum_{u \in U} d_{\mu''_q}(u)^2 p^{-2} + \sum_{\substack{L \subset U \\ |L| \ge 2}} d_{\bar{\mu}''_q}\big(L \cup \{v\}\big)^2 p^{-|L| - 1},
    \end{equation}
    where the second term corresponds to $L = \{u, v\}$ for $u \in U$, that is, the case $v \in L$ and $|L| = 2$.
    Applying \Cref{stmt:degreeOfBarMuC} in the third sum of \eqref{eq:lambdaBarMuQEqualityLambdaMuQStep} and using the definition of $\Lambda_p(\mu''_q)$ yields \eqref{eq:lambdaBarMuQEqualityLambdaMuQ}.
  \end{claimproof}

  Now, we prove a deterministic upper bound for $\Lambda_p(\mu''_{q})$ in terms of $\Lambda_p(\mu'')$.
  We do not optimise this bound, like the analogous one in \Cref{sec:simplerContainerForNonJanson} (cf. \Cref{stmt:finalBoundLambdaPQNuDoublePrimeQ}), since that will not be necessary in this proof.

  \begin{claim}\label{stmt:lambdaMuDoublePrimeQUpperBound}
    \begin{equation*}
      \Lambda_p(\mu''_{q}) \le \gamma^2 \Big(\frac{\,q\,}{2}\Big)^{-2 s} \Lambda_p(\mu'').
    \end{equation*}
  \end{claim}

  \begin{claimproof}
    Recall that
    \begin{equation*}
      P'_q(E) > \Big(\frac{\,q\,}{2}\Big)^s
    \end{equation*}
    for all $E \in \pi(\cal{H}')$, by \Cref{stmt:lowerBoundOnPOfEPrime}, so by definition of $d_{\mu''_q}$ and $\mu''_q$, we deterministically have that
    \begin{equation}\label{eq:ignoringTheIndicator}
      d_{\mu''_q}(L)^2 = \bigg( \sum_{L \subset E \in \pi(\cal{H}')} \mu''(E) \, \frac{\gamma \cdot \mathds{1}\big[E \in \pi(\cal{H}'[V_q])\big]}{P'_q(E)} \bigg)^2 \le \gamma^2 \bigg(\frac{\,q\,}{2}\bigg)^{-2 s} d_{\mu''}(L)^2,
    \end{equation}
    holds for all $L \subset U$, by ignoring the indicators.
    The inequality in the \namecref{stmt:lambdaMuDoublePrimeQUpperBound} now follows from the definition of $\Lambda_p(\cdot)$:
    \begin{equation*}
      \Lambda_p(\mu''_q) = \sum_{\substack{L \subset U \\ |L| \ge 2}} d_{\mu''_q}(L)^2 p^{-|L|} \le \gamma^2 \Big(\frac{\,q\,}{2}\Big)^{-2 s} \sum_{\substack{L \subset U \\ |L| \ge 2}} d_{\mu''}(L)^2 p^{-|L|} = \gamma^2 \Big(\frac{\,q\,}{2}\Big)^{-2 s} \Lambda_p(\mu'').
      \vspace{-15pt}
    \end{equation*}
  \end{claimproof}

  We will now combine the bound given by \Cref{stmt:boundDegreeSquared} for the sum of the square of the $\mu$-degrees, \eqref{eq:muHasBddDegrees}, with \Cref{stmt:lambdaBarMuQEqualityLambdaMuQ,stmt:lambdaMuDoublePrimeQUpperBound} and another simple calculation to complete the proof of a deterministic inequality relating $\Lambda_p(\bar{\mu}''_q)$ and $e(\mu)^2$.

  \begin{claim}\label{stmt:lambdaMuBarDoublePrimeQUpperBound}
    \begin{equation*}
      \Lambda_p(\bar{\mu}''_{q}) < \frac{\gamma^2}{2 \sqrt{\eta}}.
    \end{equation*}
  \end{claim}

  \begin{claimproof}
    Repeating what we did in \eqref{eq:ignoringTheIndicator}, we have
    \begin{equation}\label{eq:firstBoundOnDegreesSquaredOfMuDoublePrimeQ}
      \sum_{u \in U} d_{\mu''_q}(u)^2 \le \gamma^2 \Big(\frac{\,q\,}{2}\Big)^{-2 s} \sum_{u \in U} d_{\mu''}(u)^2 \le \gamma^2 \Big(\frac{\,q\,}{2}\Big)^{-2 s} \sum_{u \in U} d_\mu(u)^2
    \end{equation}
    where the last step is using that $\mu'' \le \mu$.
    Now, recall that the support of $\mu$ is $\pi(\cal{H}[X])$, so if an edge $E \subset U$ is not fully contained in $\pi(X)$, then $\mu(E) = 0$.
    As so, $d_\mu(u) = 0$ if $u \not \in \pi(X)$ and thus
    \begin{equation*}
      \sum_{u \in U} d_\mu(u)^2 = \sum_{u \in \pi(X)} d_\mu(u)^2.
    \end{equation*}
    We conclude that \eqref{eq:firstBoundOnDegreesSquaredOfMuDoublePrimeQ} is at most
    \begin{equation}\label{eq:boundOnDegreesSquaredOfMuDoublePrimeQ}
      \sum_{u \in U} d_{\mu''_q}(u)^2 \le \gamma^2 \Big(\frac{\,q\,}{2}\Big)^{-2 s} \sum_{u \in U} d_\mu(u)^2 = \gamma^2 \Big(\frac{\,q\,}{2}\Big)^{-2 s} \sum_{u \in \pi(X)} d_\mu(u)^2 \le \gamma^2 \Big(\frac{\,q\,}{2}\Big)^{-2 s} \frac{2^{14}r^2 s^2 e(\mu)^2}{n}
    \end{equation}
    where the last step is
    \begin{equation*}
      \sum_{u \in \pi(X)} d_\mu(u)^2 \le \frac{2^{14} r^2 s^2 e(\mu)^2}{n},
    \end{equation*}
    the inequality in \eqref{eq:muHasBddDegrees}.

    Combining \eqref{eq:boundOnDegreesSquaredOfMuDoublePrimeQ} with \Cref{stmt:lambdaBarMuQEqualityLambdaMuQ,stmt:lambdaMuDoublePrimeQUpperBound} thus yields
    \begin{equation}\label{eq:boundOnLambdaPMuBarDoublePRimeQ}
      \begin{aligned}
      \Lambda_p(\bar{\mu}''_q)
        \le \gamma^2 \bigg(\frac{2}{q}\bigg)^{2 s} \left(\Big(1 + \frac{\,1\,}{p}\Big) \Lambda_p(\mu'') + \frac{2^{14} r^2 s^2 e(\mu)^2}{p^2n} \right).
      \end{aligned}
    \end{equation}
    Also recall that we chose
    \begin{equation*}
      p \le \frac{\,q\,}{2^{11} r s^2}, \qquad q < \frac{1}{8}, \qquad R = 2^{-6} r^{-1} p n, \qquad \text{and } \qquad \eta=p^4\Big(\frac{\,q\,}{2}\Big)^{4 s},
    \end{equation*}
    in \hyperlink{stmt:ContainersForNonJanson:restated:sec7}{\eqref{eq:assumptionsOnContainerParams}}, and that
    \begin{equation*}
      \Lambda_p(\mu'') \le \Lambda_p(\mu) < \frac{e(\mu)^2}{R} \qquad \text{and} \qquad e(\mu) = \sqrt{R}
    \end{equation*}
    by $\mu'' \le \mu$ and \eqref{eq:muIsJanson}, so \eqref{eq:boundOnLambdaPMuBarDoublePRimeQ} is at most
    \begin{equation*}
      \Lambda_p(\bar{\mu}''_q) < \gamma^2 \bigg(\frac{2}{q}\bigg)^{2 s} \left(1 + \frac{\,1\,}{p} + \frac{2^{8} r s^2}{p} \right) \le \gamma^2 \bigg(\frac{2}{q}\bigg)^{2 s} \, \frac{2^{10} r s^2}{p} \le \frac{\gamma^2}{2\sqrt{\eta}}
    \end{equation*}
    as desired.
  \end{claimproof}

  Replacing the bound of \Cref{stmt:lambdaMuBarDoublePrimeQUpperBound} in \eqref{eq:changeFromUCupVToU} and then taking the conditional expectation with respect to $\{V_q \in \cal{I}(\link{\cal{J}'}{T})\}$ yields
  \begin{equation}\label{eq:secondToLastBoundOnLambdaPRhoMuBarDoublePrimeQ}
    \EE\big[\Lambda_p(\rho + \bar{\mu}''_q) \mid V_q \in \cal{I}(\link{\cal{J}'}{T})\big] < 1 + 2 \sum_{\substack{L \subset U \\ |L| \ge 2}} d_\rho(L) \frac{\EE\big[d_{\bar{\mu}''_q}(L) \mid V_q \in \cal{I}(\link{\cal{J}'}{T})\big]}{p^{|L|}} + \frac{\gamma^2}{2\sqrt{\eta}}.
  \end{equation}
  In particular, this inequality motivates our final \namecref{stmt:aBoundForDegreeOfLSubsetUInRandomMeasure}.
  \begin{claim}\label{stmt:aBoundForDegreeOfLSubsetUInRandomMeasure}
    If $L \subset U$, then
    \begin{equation*}
      \EE\big[d_{\bar{\mu}''_q}(L) \mid V_q \in \cal{I}(\link{\cal{J}'}{T})\big] = \gamma d_{\mu''}(L).
    \end{equation*}
  \end{claim}

  \begin{claimproof}
    Note that $v \not \in L$ if $L \subset U$.
    It then follows from \Cref{stmt:degreeOfBarMuC} and the definition of $d_{\mu''_q}(L)$ that
    \begin{equation*}
      d_{\bar{\mu}''_q}(L) = d_{\mu''_q}(L) = \sum_{L \subset E \in \pi(\cal{H}')} \mu''(E) \, \frac{\gamma \cdot \mathds{1}\big[E \in \pi(\cal{H}'[V_q])\big]}{P'_q(E)},
    \end{equation*}
    so taking the conditional expectation yields, for all $L \subset U$,
    \begin{equation*}
      \EE\big[d_{\bar{\mu}''_q}(L) \mid V_q \in \cal{I}(\link{\cal{J}'}{T})\big] = \sum_{L \subset E \in \pi(\cal{H}')} \gamma \mu''(E) = \gamma d_{\mu''}(L)
    \end{equation*}
    since
    \begin{equation*}
      \EE\big[\mathds{1}\big[E \in \pi(\cal{H}'[V_q])\big] \mid V_q \in \cal{I}(\link{\cal{J}'}{T})\big] = P'_q(E),
    \end{equation*}
    \eqref{eq:indicatorAndPEPrime}, holds for all $E \in \pi(\cal{H}')$.
  \end{claimproof}

  Substituting the bound in \Cref{stmt:aBoundForDegreeOfLSubsetUInRandomMeasure} in the middle sum of \eqref{eq:secondToLastBoundOnLambdaPRhoMuBarDoublePrimeQ} and applying the Cauchy--Schwarz inequality, we obtain
  \begin{equation*}
    \gamma \sum_{\substack{L \subset U \\ |L| \ge 2}} \frac{d_\rho(L) d_{\mu''}(L)}{p^{|L|}}
    \le \gamma \Bigg(\sum_{\substack{L \subset U \\ |L| \ge 2}} \frac{d_\rho(L)^2}{p^{|L|}}\Bigg)^{1/2} \Bigg(\sum_{\substack{L \subset U \\ |L| \ge 2}} \frac{d_{\mu''}(L)^2}{p^{|L|}}\Bigg)^{1/2}
    = \gamma \sqrt{\Lambda_p(\rho)} \sqrt{\Lambda_p(\mu'')}
    \le \gamma,
  \end{equation*}
  which, replaced back in \eqref{eq:secondToLastBoundOnLambdaPRhoMuBarDoublePrimeQ}, yields
  \begin{equation}\label{eq:finalBoundOnLambda}
    \EE\big[\Lambda_p(\rho + \bar{\mu}''_q) \mid V_q \in \cal{I}(\link{\cal{J}'}{T})\big] < 1 + 2 \gamma + \frac{\gamma^2}{2 \sqrt{\eta}}.
  \end{equation}

  Now, to bound the expected edge mass of the measure $\rho + \bar{\mu}''_q$, observe that
  \begin{equation*}
    \EE\big[e(\rho + \bar{\mu}''_q) \mid V_q \in \cal{I}(\link{\cal{J}'}{T})\big] = e(\rho) + \EE\big[e(\bar{\mu}''_q) \mid V_q \in \cal{I}(\link{\cal{J}'}{T})\big] \ge \sqrt{R'} + \frac{\gamma\sqrt{R}}{2},
  \end{equation*}
  where the last inequality is due to \eqref{eq:rhoIsNormalisedJanson} and \eqref{eq:boundone(mubard)}.
  Jensen's inequality thus implies that
  \begin{equation}\label{eq:boundOnExpectedEdges}
    \EE\big[e(\rho + \bar{\mu}''_q)^2 \mid V_q \in \cal{I}(\link{\cal{J}'}{T})\big] \ge \Big(\sqrt{R'} + \frac{\gamma\sqrt{R}}{2}\Big)^2 > (1 + 4\gamma)\Big(R' + \frac{\gamma^2 R}{8}\Big)
  \end{equation}
  because $R \ge 16 R'$ by assumption, and we can choose $\gamma < 1/4$.
  We fix $\gamma = \sqrt{8 \eta}$, which is less than $1/4$ because
  \[\eta = p^4\Big(\frac{\,q\,}{2}\Big)^{4 s} < \frac{1}{2^7}.\]
  This choice, replaced in \eqref{eq:finalBoundOnLambda}, implies that
  \begin{equation*}
    \EE\big[\Lambda_p(\rho + \bar{\mu}''_q) \mid V_q \in \cal{I}(\link{\cal{J}'}{T})\big] < 1 + 4 \gamma = \frac{(1 + 4 \gamma)(R' + \gamma^2 R / 8)}{R' + \eta R} \le \frac{\EE\big[e(\rho + \bar{\mu}''_q)^2 \mid V_q \in \cal{I}(\link{\cal{J}'}{T})\big]}{R' + \eta R}
  \end{equation*}
  and the last step is \eqref{eq:boundOnExpectedEdges}.

  To reach a contradiction, observe that if $V_q \in \cal{I}(\link{\cal{J}'}{T})$ then $V_q \in \cal{I}(\cal{J}')$ follows from \Cref{stmt:independentInLinkImpliesIndependentInOriginal}, and hence $\pi_v(\cal{H}'[V_q]) \cup \cal{F}$ is not $(p, R' + \eta R)$-Janson, by the definition of $\cal{J}'$.
  In particular, it follows that
  \begin{equation*}
    \Lambda_{p}(\rho + \bar{\mu}''_q) \ge \frac{e(\rho + \bar{\mu}''_q)^2}{R' + \eta R}
  \end{equation*}
  since this measure is supported on $\pi_v(\cal{H}'[V_q]) \cup \cal{F}$, so we cannot have
  \begin{equation*}
    \EE\big[\Lambda_p(\rho + \bar{\mu}''_q) \mid V_q \in \cal{I}(\link{\cal{J}'}{T})\big] < \frac{\EE\big[e(\rho + \bar{\mu}''_q)^2 \mid V_q \in \cal{I}(\link{\cal{J}'}{T})\big]}{R' + \eta R}
  \end{equation*}
  like we previously established, and we have a contradiction.
  We conclude that \hyperlink{stmt:ContainersForNonJanson:restated:sec7}{\cref{item:containersForNonJansonAreNonJanson}} in \hyperlink{stmt:ContainersForNonJanson:restated:sec7}{\Cref{stmt:containersForNonJanson}} holds, and the proof is complete.
\end{proof}

\section*{Acknowledgements}

We would like to greatly thank Rob Morris for carefully reading the paper, and for the many improvements and corrections he suggested.
We are also grateful to Wojciech Samotij and Julian Sahasrabudhe for helpful comments and discussions on the presentation and proof.

\bibliographystyle{abbrvnat}
\def\bibfont{\footnotesize}
\bibliography{refs}

\appendix

\section{Proof of \texorpdfstring{\Cref{stmt:containersHardcovers}}{\ref{stmt:containersHardcovers}}}\label{app:containersHardcovers}

This \namecref{app:containersHardcovers} is dedicated to the proof of \Cref{stmt:containersHardcovers}, which we restate for the reader's convenience.

\hypertarget{appendix:stmt:containersHardcovers}{\containersHardcovers*}

The following proof is essentially a subset of the corresponding argument in \cite{CS24+}, but the exact statement of \hyperlink{appendix:stmt:containersHardcovers}{\Cref{stmt:containersHardcovers}} admits a simpler, self-contained proof that in particular does not assume the result of \citeauthor{CS24+}.

\begin{proof}
  Given $I \in \cal{I}(\cal{G})$, let $T \subset I$ be maximal with respect to
  \begin{equation}\label{eq:containersHardCoversDefinitionOfFingerprint}
    \PP(T \subset V_q \mid V_q \in \cal{I}(\cal{G})) \le (1 - \alpha)^{|T|} q^{|T|}.
  \end{equation}
  Set $\varphi(I) = T$ and $\cal{T} = \{\varphi(I) : I \in \cal{I}(\cal{G})\}.$
  We claim that these choices satisfy the requirements of \hyperlink{appendix:stmt:containersHardcovers}{\Cref{stmt:containersHardcovers}}.
  Notice that \hyperlink{appendix:stmt:containersHardcovers}{\cref{item:containersHardcoversHaveFingerprints}} trivially holds, since $\varphi(I)$ is defined as a subset of $I$.

  Now, take an arbitrary $T \in \cal{T}$.
  We have, on one hand,
  \begin{equation}\label{eq:lowerBoundOnYProb}
    (1 - q)^{n - |T|} q^{|T|} = \PP(V_q = T) \le \PP(T \subset V_q \wedge V_q \in \cal{I}(\cal{G})),
  \end{equation}
  where the inequality is using that $T \in \cal{I}(\cal{G})$, and, on the other hand,
  \begin{equation}\label{eq:upperBoundOnYProb}
    \PP(T \subset V_q \wedge V_q \in \cal{I}(\cal{G})) \le \PP(T \subset V_q \mid V_q \in \cal{I}(\cal{G})) \le (1-\alpha)^{|T|}q^{|T|}
  \end{equation}
  due to \eqref{eq:containersHardCoversDefinitionOfFingerprint}.
  Combining \eqref{eq:lowerBoundOnYProb} and \eqref{eq:upperBoundOnYProb}, we obtain
  \begin{equation*}
    (1 - q)^{n - |T|} q^{|T|} \le (1 - \alpha)^{|T|} q^{|T|},
  \end{equation*}
  which, as $x \mapsto (1 - x)^{1/x}$ is decreasing on $(0, 1)$, implies $|T| \le qn/\alpha$.
  But we took $T$ arbitrarily, so the above establishes \hyperlink{appendix:stmt:containersHardcovers}{\cref{item:containersHardcoversHaveSmallFingerprints}}.

  To show that \hyperlink{appendix:stmt:containersHardcovers}{\cref{item:containersHardcoversAreCoverable}} holds, take any $T \in \cal{T}$ and set
  \begin{equation}\label{eq:defOfCalCY}
    \cal{C}_T = \left\{L \subset V(\cal{G}):\, \PP\big(L \subset V_q \mid V_q \in \cal{I}(\link{\cal{G}}{T})\big) \le (1 - \alpha)^{|L|} q^{|L|} \right\}.
  \end{equation}
  Observe that not only $\cal{C}_T$ is a cover of $\cal{G}$, but also $\cal{G} \subset \cal{C}_T$: for all $E \in \cal{G}$, we have
  \begin{equation*}
    \PP\big(E \subset V_q \mid V_q \in \cal{I}(\link{\cal{G}}{T})\big) = 0,
  \end{equation*}
  because $V_q \in \cal{I}(\link{\cal{G}}{T})$ and $E \subset V_q$ together imply $(E  \setminus T) \in \cal{I}(\link{\cal{G}}{T})$, which directly contradicts the definition of $\link{\cal{G}}{T}$.
  The definition in \eqref{eq:defOfCalCY} also immediately implies that, for all $L \not \in \cal{C}_T$,
  \begin{equation*}
    \PP\big(L \subset V_q \mid V_q \in \cal{I}(\link{\cal{G}}{T})\big) > (1 - \alpha)^{|L|} q^{|L|}
  \end{equation*}
  holds, and that is exactly \hyperlink{appendix:stmt:containersHardcovers}{\eqref{eq:conditionalProbOfLInQSet} in \cref{item:containersHardcoversAreCoverable}}.
  It remains only to establish the ``moreover'' part of \hyperlink{appendix:stmt:containersHardcovers}{\cref{item:containersHardcoversAreCoverable}}.

  Our goal is to show that for all $I \in \cal{I}(\cal{G})$, if $\varphi(I) = T$, then $I \in \cal{I}(\cal{C}_T)$, so we take $L \in \cal{C}_T$ and must determine that $L \not \subset I$.
  Note that $L \not \subset T$ follows from $L \in \cal{C}_T$, as otherwise
  \begin{equation*}
    \PP\big(L\subset V_q \mid V_q \in \cal{I}(\link{\cal{G}}{T}) \big)=\PP(L\subset V_q)=q^{|L|}>(1-\alpha)^{|L|}q^{|L|},
  \end{equation*}
  because $\{V_q \in \cal{I}(\link{\cal{G}}{T})\}$ and $\{L \subset V_q\}$ are independent when $L \subset T$.

  We claim that
  \begin{equation}\label{eq:breakingProbabilityForYCupL}
    \PP\big(L \cup T \subset V_q \mid V_q \in \cal{I}(\cal{G})\big) = \PP(T \subset V_q \mid V_q \in \cal{I}(\cal{G})\big) \, \PP\big(L \setminus T \subset V_q \mid V_q \in \cal{I}(\link{\cal{G}}{T})\big).
  \end{equation}
  First, observe that we can reveal $T_q = T \cap V_q$ and then $V'_q = V_q \setminus T$, obtaining as a result
  \begin{equation*}
    \PP\big(L \cup T \subset V_q \mid V_q \in \cal{I}(\cal{G})\big) = \PP(T \subset V_q \mid V_q \in \cal{I}(\cal{G})\big) \, \PP\big(L \setminus T \subset V'_q \mid T \subset V_q \land V_q \in \cal{I}(\cal{G})\big).
  \end{equation*}
  But $\{T \subset V_q\} = \{T_q = T\}$ and $V_q \in \cal{I}(\cal{G})$ are together equivalent to $V'_q \in \cal{I}(\link{\cal{G}}{T})$, so we have
  \begin{equation*}
    \PP\big(L \setminus T \subset V_q \setminus T \mid T \subset V_q \land V_q \in \cal{I}(\cal{G})\big) = \PP\big(L \setminus T \subset V'_q \mid V'_q \in \cal{I}(\link{\cal{G}}{T})\big),
  \end{equation*}
  and therefore
  \begin{equation}\label{eq:breakingProbabilityForYCupLWithVPrime}
    \PP\big(L \cup T \subset V_q \mid V_q \in \cal{I}(\cal{G})\big) = \PP(T \subset V_q \mid V_q \in \cal{I}(\cal{G})\big) \, \PP\big(L \setminus T \subset V'_q \mid V'_q \in \cal{I}(\link{\cal{G}}{T})\big).
  \end{equation}
  Finally, $L \setminus T$ and $V_q \setminus V'_q \subset T$ are disjoint and $T_q$ is independent of $\{V'_q \in \cal{I}(\link{\cal{G}}{T})\}$, so
  \begin{equation}\label{eq:replacingVPrimeByV}
    \PP\big(L \setminus T \subset V'_q \mid V'_q \in \cal{I}(\link{\cal{G}}{T})\big) = \PP\big(L \setminus T \subset V_q \mid V_q \in \cal{I}(\link{\cal{G}}{T})\big)
  \end{equation}
  and substituting \eqref{eq:replacingVPrimeByV} in \eqref{eq:breakingProbabilityForYCupLWithVPrime} yields \eqref{eq:breakingProbabilityForYCupL}.

  Now, partitioning $L$ according to its intersection with $T$, we have
  \begin{equation*}
    \PP\big(L \subset V_q \mid V_q \in \cal{I}(\link{\cal{G}}{T})\big) = \PP\big(L \cap T \subset V_q\big) \, \PP\big(L \setminus T \subset V_q \mid V_q \in \cal{I}(\link{\cal{G}}{T})\big)
  \end{equation*}
  because $\{L \cap T \subset V_q\}$ and $\{L \setminus T \subset V_q\}$ are independent, and so are $\{L \cap T \subset V_q\}$ and $\{V_q \in \cal{I}(\link{\cal{G}}{T})\}$.
  Therefore,
  \begin{equation}\label{eq:breaking-prob-L-subset-Vq}
    \PP\big(L \setminus T \subset V_q \mid V_q \in \cal{I}(\link{\cal{G}}{T})\big) = \PP\big(L \subset V_q \mid V_q \in \cal{I}(\link{\cal{G}}{T})\big) q^{-|L \cap T|}.
  \end{equation}

  Combining \eqref{eq:breaking-prob-L-subset-Vq} with \eqref{eq:breakingProbabilityForYCupL}, we obtain
  \begin{equation*}
    \PP\big(L \cup T \subset V_q \mid V_q \in \cal{I}(\cal{G})\big) = \PP\big(L  \subset V_q \mid V_q \in \cal{I}(\link{\cal{G}}{T})\big) \, \PP(T \subset V_q \mid V_q \in \cal{I}(\cal{G})\big) q^{-|L \,\cap \, T|},
  \end{equation*}
  and now we can use that our choice of $T$ satisfies \eqref{eq:containersHardCoversDefinitionOfFingerprint}, and $L$ satisfies \eqref{eq:defOfCalCY} to obtain
  \begin{equation}\label{eq:tPrimeWouldBeBetterThanT}
    \PP\big(L \cup T \subset V_q \mid V_q \in \cal{I}(\cal{G})\big) \le (1 - \alpha)^{|T|+|L|} q^{|T|+|L|-|L \, \cap\, T|} \le (1-\alpha)^{|L \cup T|} q^{|L \cup T|}
  \end{equation}
  because $0 < \alpha < 1$.

  Taking $T' = L \cup T$, it is obvious that $T \subset T'$.
  We have picked $T \subset I$ to be maximal satisfying \eqref{eq:containersHardCoversDefinitionOfFingerprint}, so it follows from \eqref{eq:tPrimeWouldBeBetterThanT} that if $T' \subset I$, then we would have picked it instead of $T$.
  Therefore, $T' \not \subset I$ and thus $L \not \subset I$.
  As we took $L \in \cal{C}_T$ arbitrarily, we conclude that $I \in \cal{I}(\cal{C}_T)$ and so \hyperlink{appendix:stmt:containersHardcovers}{\cref{item:containersHardcoversAreCoverable}} holds, completing the proof.
\end{proof}

\section{Proof of \texorpdfstring{\Cref{stmt:containersCoversButJanson}}{\ref{stmt:containersCoversButJanson}}}\label{app:containersCoversButJanson}

We recall the statement of \Cref{stmt:containersCoversButJanson} for the reader's convenience.

\hypertarget{stmt:containersCoversButJanson:appendixB}{\containersCoversButJanson*}

To prove it, we need one of the container \namecrefs{stmt:containersCovers} of \citet[Theorem A]{CS24+}, which we state below as \Cref{stmt:containersCovers}.
Recall from \cite{CS24+} that when $\cal{C}$ is a hypergraph and $0 < p < 1$, we denote by
\begin{equation*}
  w_p(\cal{C}) = \sum_{E \in \cal{C}} p^{|E|}
\end{equation*}
what is called the $p$-weight of $\cal{C}$.

\begin{thm}[{\citet[Theorem A]{CS24+}}]\label{stmt:containersCovers}
  Let $\cal{G}$ be an $s$-uniform hypergraph with $n$ vertices.
  For every $0 < p' \le 1/(8 s^2)$, there exists a family $\cal{S} \subset 2^{V(\cal{G})}$ and functions
  \begin{equation}
    \phi : \cal{I}(\cal{G}) \to \cal{S} \quad \text{ and } \quad \psi : \cal{S} \to 2^{V(\cal{G})}
  \end{equation}
  such that:
  \begin{enumerate}[\normalfont (A)]
    \item For each $I \in \cal{I}(\cal{G})$, we have $\phi(I) \subset I \subset \psi(\phi(I))$.\label{item:containerscoverhavefingerprintappendix}
    \item Each $S \in \cal{S}$ has at most $8 s^2 p' n$ elements.\label{item:containerscoverissmallappendix}
    \item For every $S \in \cal{S}$, letting $X = \psi(S)$, there exists a hypergraph $\cal{C}$ on $X$ with
      \begin{equation}\label{eq:containersCoversCoverIsCheap}
        w_{p'}(\cal{C}) \le p' |X|
      \end{equation}
      that covers $\cal{G}[X]$ and satisfies $|E| \ge 2$ for all $E \in \cal{C}$. \label{item:containersCoversHaveCheapCovers}
  \end{enumerate}
\end{thm}

\begin{proof}[Proof of \Cref{stmt:containersCoversButJanson}]
  We apply \Cref{stmt:containersCovers} to $\cal{G}$ with parameter $p' = p / \zeta \le 1/(8 s^2)$ and obtain $\phi$, $\psi$ and $\cal{S}$.
  \hyperlink{stmt:containersCoversButJanson:appendixB}{\Cref{item:containersCoversHaveFingerprints,item:containersCoversHaveSmallFingerprints}} in \hyperlink{stmt:containersCoversButJanson:appendixB}{\Cref{stmt:containersCoversButJanson}} are direct consequences of \cref{item:containerscoverhavefingerprintappendix,item:containerscoverissmallappendix} in \Cref{stmt:containersCovers}, so it remains only to show that \cref{item:containersCoversHaveCheapCovers} implies \hyperlink{stmt:containersCoversButJanson:appendixB}{\cref{item:containersAreNonJanson}}.
  Fix $S \in \cal{S}$ and the corresponding $X = \psi(S)$, and let $\cal{C}$ be the cover of $\cal{G}[X]$ given by \cref{item:containersCoversHaveCheapCovers} in \Cref{stmt:containersCovers}, and notice that by \eqref{eq:containersCoversCoverIsCheap} we have
  \begin{equation}\label{eq:containersCoversCoverIsCheap-p'}
    w_{p'}(\cal{C}) \le p'|X| = p |X| / \zeta.
  \end{equation}
  We can use \eqref{eq:containersCoversCoverIsCheap-p'} and combine the assumption $\zeta \le 1$ with the fact that every edge in $\cal{C}$ has size at least $2$ to bound the $p$-weight of $\cal{C}$ from its $p'$-weight:
  \begin{equation}\label{eq:changeofweights}
    w_{p}(\cal{C}) = \sum_{E \in \cal{C}} p^{|E|} = \sum_{E\in \cal{C}} (p / \zeta)^{|E|} \zeta^{|E|} \le \zeta^{2} \sum_{E \in \cal{C}} (p / \zeta)^{|E|} = \zeta^2 w_{p'}(\cal{C}) \le \zeta p|X|.
  \end{equation}
  To show that $\cal{G}[X]$ is not $(p, \zeta p|X|)$-Janson, take any measure $\nu : \cal{G}[X] \to \RR_{\ge 0}$, so our goal is to establish that
  \begin{equation*}
    \Lambda_p(\nu) \ge \frac{e(\nu)^2}{\zeta p|X|}.
  \end{equation*}
  Observe first that
  \begin{equation}\label{eq:appendix:firstLowerBoundOnLambda}
    \Lambda_p(\nu)=\sum_{\substack{L \subset V(\cal{H}), \\ |L| \ge 2}} d_{\nu}(L)^2 \, p^{-|L|}\ge \sum_{L \in \cal{C}} d_{\nu}(L)^2 p^{-|L|},
  \end{equation}
  since every edge in $\cal{C}$ has size at least $2$ by \cref{item:containersCoversHaveCheapCovers} in \Cref{stmt:containersCovers} and all the terms in the sum are non-negative.
  Massaging \eqref{eq:appendix:firstLowerBoundOnLambda}, we can apply the Cauchy--Schwartz inequality to obtain
  \begin{equation}\label{eq:appendix:secondLowerBoundOnLambda}
    \Lambda_p(\nu) \ge \frac{1}{w_p(\cal{C})} \left( \sum_{L \in \cal{C}} d_\nu(L)^2 \, p^{-|L|} \right) \left(\sum_{L \in \cal{C}} p^{|L|}\right)
    \ge \frac{1}{w_p(\cal{C})} \left( \sum_{L \in \cal{C}} d_\nu(L)\right)^2.
  \end{equation}
  Now, note that, as $\cal{C}$ is a cover of $\cal{G}[X]$, we have
  \begin{equation}\label{eq:appendix:preMuBoundFromCoverDegrees}
    \sum_{L \in \cal{C}} d_\nu(L) = \sum_{L \in \cal{C}} \sum_{L \subset E} \nu(E) \ge \sum_{E \in \cal{G}[X]} \nu(E)=e(\nu).
  \end{equation}
  Combining \eqref{eq:changeofweights} and \eqref{eq:appendix:preMuBoundFromCoverDegrees} with \eqref{eq:appendix:secondLowerBoundOnLambda}, we obtain
  \begin{equation*}
    \Lambda_p(\nu) \ge \frac{e(\nu)^2}{\zeta p|X|},
  \end{equation*}
  which completes the proof because $\nu$ was arbitrary.
\end{proof}

\section{Properties of measures}\label{app:propertiesOfMeasures}

\subsection{The pullback measure}\label{app:pullbackPreservesLambdaProperties}

The main goal of this \namecref{app:pullbackPreservesLambdaProperties} is to prove \Cref{stmt:pullbackPreservesLambdaProperties}, which we restate for convenience.
Observe that \Cref{stmt:wholeCopiesAreJanson} is a direct corollary of this statement.

\hypertarget{pullbackCrucialPropRestated}{\pullbackPreservesLambdaProperties*}

The missing proof of \hyperlink{pullbackCrucialPropRestated}{\Cref{stmt:pullbackPreservesLambdaProperties}} is a trivial combination of \hyperlink{edgesPullbackRestated}{\Cref{stmt:edgesPullback}}, which says that $e(\vartheta \normcomp \pi) = e(\vartheta)$, and \Cref{stmt:lambdaPullback}, which establishes $\Lambda_p(\vartheta \normcomp \pi) \le \Lambda_p(\vartheta)$.
Towards these two \namecrefs{stmt:edgesPullback}, we recall the statement of \Cref{stmt:pullbackIsUniform} about pullback measures.

\hypertarget{pullbackIsUniformRestated}{\pullbackIsUniform*}

\begin{proof}
  Fix $E' \in \pi(\cal{G})$.
  We have by \eqref{eq:defPullback} that
  \begin{equation*}
   \sum_{\substack{E \in \cal{G} \\ \pi(E) = E'}} \vartheta \normcomp \pi(E) =   \sum_{\substack{E \in \cal{G} \\ \pi(E) = E'}} \frac{\vartheta(E')}{\big|\{E \in \cal{G} : \pi(E) = E'\}\big|} = \vartheta(E')
  \end{equation*}
\end{proof}

We can now prove \hyperlink{edgesPullbackRestated}{\Cref{stmt:edgesPullback}}, an easy consequence of the definition of the pullback measure.

\hypertarget{edgesPullbackRestated}{\edgesPullback*}

\begin{proof}
  Expanding the definition of $e(\vartheta \normcomp \pi)$ and using \hyperlink{pullbackIsUniformRestated}{\Cref{stmt:pullbackIsUniform}}, we obtain
  \begin{equation*}
    e(\vartheta) = \sum_{E' \in \pi(\cal{G})} \vartheta(E') = \sum_{E' \in \pi(\cal{G})} \sum_{\substack{E \in \cal{G} \\ \pi(E) = E'}} \vartheta \normcomp \pi(E).
  \end{equation*}
  But each $E \in \cal{G}$ appears on the right-hand side exactly once, only when $E' \in \pi(\cal{G})$ satisfies $\pi(E) = E'$.
  Therefore,
  \begin{equation*}
    \sum_{E' \in \pi(\cal{G})} \sum_{\substack{E \in \cal{G} \\ \pi(E) = E'}} \vartheta \normcomp \pi(E) = \sum_{E \in \cal{G}} \vartheta \normcomp \pi(E) = e(\vartheta \normcomp \pi)
  \end{equation*}
  as we wanted to show.
\end{proof}

The next \namecref{stmt:degPullback} requires an assumption about $\pi$, motivating its appearance in the statements of \hyperlink{pullbackCrucialPropRestated}{\Cref{stmt:pullbackPreservesLambdaProperties}} and \Cref{stmt:containersForNonJanson}.
The proof is easy, and follows from expanding the definitions and a simple counting argument.
It also requires defining the uniformity-preserving pre-image of $L' \subset \pi(V(\cal{G}))$, the set
\begin{equation*}
  \uppi{\pi}(L') = \{L \subset V(\cal{G}) : \pi(L) = L' \text{ and } |L| = |L'|\}.
\end{equation*}

\begin{lem}\label{stmt:degPullback}
  Let $\cal{G}$ be a hypergraph, let $\pi : V(\cal{G}) \to U$ satisfy
  \begin{equation*}
    |\pi(E)| = |E| \qquad \text{for every } E \in \cal{G},
  \end{equation*}
  and let $\vartheta : \pi(\cal{G}) \to \RR_{\ge 0}$ be a measure.
  For all $L' \subset \pi(V(\cal{G}))$, if $\lambda = \vartheta \normcomp \pi$, then
  \begin{equation}\label{eq:degPullback}
    d_\vartheta(L') = \sum_{L \in \uppi{\pi}(L')} d_\lambda(L).
  \end{equation}
\end{lem}

\begin{proof}
  Fix $L' \subset \pi(V(\cal{G}))$.
  The definition of $d_\vartheta(L')$, combined with \hyperlink{pullbackIsUniformRestated}{\Cref{stmt:pullbackIsUniform}}, yields
  \begin{equation}\label{eq:degPullback:expansion}
    d_\vartheta(L') = \sum_{L' \subset E' \in \pi(\cal{G})} \vartheta(E') = \sum_{L' \subset E' \in \pi(\cal{G})} \sum_{\substack{E \in \cal{G} \\ \pi(E) = E'}} \lambda(E) = \sum_{\substack{E \in \cal{G} \\ L' \subset \pi(E)}} \lambda(E),
  \end{equation}
  where the last step holds because each $E \in \cal{G}$ with $L' \subset \pi(E)$ appears exactly once in the second-to-last sum.
  On the other hand, the right-hand side of \eqref{eq:degPullback} is, by definition, equal to
  \begin{equation}\label{eq:degPullback:expansionRhs}
    \sum_{L \in \uppi{\pi}(L')} d_\lambda(L) = \sum_{L \in \uppi{\pi}(L')} \, \sum_{L \subset E \in \cal{G}} \lambda(E).
  \end{equation}
  We also know that $|L| = |L'|$ for all $L \in \uppi{\pi}(L')$.
  It then follows from $|\pi(E)| = |E|$ for all $E \in \cal{G}$ that $\pi|_E$ is a bijection, so there is a unique $L \subset E$ satisfying $\pi(L) = L'$ and $|\pi(L)| = |L|$.
  The conclusion is that
  \begin{equation*}
    \sum_{L \in \uppi{\pi}(L')} \, \sum_{L \subset E \in \cal{G}} \lambda(E) = \sum_{\substack{E \in \cal{G} \\ L' \subset \pi(E)}} \lambda(E),
  \end{equation*}
  which, together with \eqref{eq:degPullback:expansion}, \eqref{eq:degPullback:expansionRhs} and the fact that $L'$ was arbitrary, completes the proof.
\end{proof}

The proof of \hyperlink{pullbackCrucialPropRestated}{\Cref{stmt:pullbackPreservesLambdaProperties}} will be complete once we establish \Cref{stmt:lambdaPullback}, which also admits a simple proof from \Cref{stmt:degPullback}.

\begin{lem}\label{stmt:lambdaPullback}
  Let $R > 0$ and $p > 0$.
  Further let $\cal{G}$ be a hypergraph, $\pi : V(\cal{G}) \to U$ be a function satisfying
  \begin{equation*}
    |\pi(E)| = |E| \qquad \text{for every } E \in \cal{G}.
  \end{equation*}
  For all $\vartheta : \pi(\cal{G}) \to \RR_{\ge 0}$, we have
  \begin{equation*}
    \Lambda_p(\vartheta \normcomp \pi) \le \Lambda_p(\vartheta).
  \end{equation*}
\end{lem}

\begin{proof}
  Let $V = V(\cal{G})$, $\lambda = \vartheta \normcomp \pi$ and recall \eqref{eq:defEdgesAndLambda}, the definition of $\Lambda_p$,
  \begin{equation*}
    \Lambda_p(\vartheta) = \sum_{\substack{L' \subset \pi(V) \\ |L'| \ge 2}} d_\vartheta(L')^2 p^{-|L'|}.
  \end{equation*}
  As $|\pi(E)| = |E|$ for every $E \in \cal{G}$, we can use \Cref{stmt:degPullback} to conclude that
  \begin{equation}\label{eq:pullbackPreservesLambdaProperties:lambdaOfThetaInTermsOfLambdaOfLambda}
    \Lambda_p(\vartheta) = \sum_{\substack{L' \subset \pi(V) \\ |L'| \ge 2}} \bigg( \sum_{L \in \uppi{\pi}(L')} d_\lambda(L) \bigg)^2 p^{-|L'|} \ge \sum_{\substack{L' \subset \pi(V) \\ |L'| \ge 2}} \, \sum_{L \in \uppi{\pi}(L')} d_\lambda(L)^2 p^{-|L|}
  \end{equation}
  where the last inequality uses that $d_\lambda(L)$ is always non-negative and that $|L| = |L'|$ for $L \in \uppi{\pi}(L')$.

  Now, whenever $d_\lambda(L) > 0$ for $L \subset V$, we know that there is $L'\subset \pi(V)$ such that $L \in \uppi{\pi}(L')$.
  We conclude that the rightmost part of \eqref{eq:pullbackPreservesLambdaProperties:lambdaOfThetaInTermsOfLambdaOfLambda} is at least
  \begin{equation*}
    \Lambda_p(\vartheta) \ge \sum_{\substack{L' \subset \pi(V) \\ |L'| \ge 2}} \, \sum_{L \in \uppi{\pi}(L')} d_\lambda(L)^2 p^{-|L|} \ge \sum_{\substack{L \subset V \\ |L| \ge 2}} d_\lambda(L)^2 p^{-|L|} = \Lambda_p(\lambda)
  \end{equation*}
  where the last step is the definition, and the proof is complete.
\end{proof}

\subsection{General properties}

Now, we prove \Cref{stmt:boundDegreeSquared}, which, despite its length, is as simple as restricting the original measure to another that zeroes the mass of edges containing high degree vertices.

\boundDegreeSquared*

\begin{proof}
  Observe that taking $W = V(\cal{G})$, we conclude that $\cal{G} = \cal{G}[W]$ is $(p, R)$-Janson by assumption.
  Therefore, we can apply \Cref{stmt:normalisedJansonWitnesses} to obtain $\mu : \cal{G} \to \RR_{\ge 0}$ such that
  \begin{equation}\label{eq:muIsJansonWitness}
    e(\mu) = \sqrt{R} \qquad \text{and} \qquad \Lambda_p(\mu) < \frac{e(\mu)^2}{R}.
  \end{equation}
  Take such a $\mu$ minimizing $\sum_{v \in V(\cal{G})} d_\mu(v)^2$, and assume by contradiction that
  \begin{equation}\label{eq:contradictionAssumptionInDegControl}
    \sum_{v \in V(\cal{G})} d_\mu(v)^2 > \frac{2 s^2 e(\mu)^2}{\beta v(\cal{G})}.
  \end{equation}

  Now, take
  \begin{equation}\label{eq:choiceOfW}
    W = \{v \in V(\cal{G}): d_\mu(v) \le s \, e(\mu)/(\beta v(\cal{G}))\}
  \end{equation}
  observe that it satisfies $|W| \ge (1 - \beta) v(\cal{G})$ since $\cal{G}$ is $s$-uniform, and therefore $G[W]$ is $(p, R)$-Janson by assumption.
  We conclude that there is a measure $\mu' : \cal{G}[W] \to \RR_{\ge 0}$ which satisfies
  \begin{equation}\label{eq:muPrimeIsAlsoJansonWitness}
    \Lambda_p(\mu') < \frac{e(\mu')^2}{R} \qquad \text{and} \qquad e(\mu') = \sqrt{R},
  \end{equation}
  again by \Cref{stmt:normalisedJansonWitnesses} applied with $y = \sqrt{R} > 0$.
  Moreover, we claim that setting $$\mu'' = (1 - \tau)\mu + \tau \mu',$$ for a suitable $0 < \tau \le 1$, also yields a measure satisfying
  \begin{equation}\label{eq:muDoublePrimeIsAlsoJansonWitness}
    \Lambda_p(\mu'') < \frac{e(\mu'')^2}{R}, \qquad e(\mu'') = \sqrt{R}
  \end{equation}
  and
  \begin{equation}\label{eq:contradictionOfDegreeMinimality}
    \sum_{v \in V(\cal{G})} d_{\mu''}(v)^2 < \sum_{v \in V(\cal{G})} d_\mu(v)^2.
  \end{equation}
  If $\mu''$ satisfies both \eqref{eq:muDoublePrimeIsAlsoJansonWitness} and \eqref{eq:contradictionOfDegreeMinimality}, it contradicts the minimality of our original choice of $\mu$.

  We start the proof of our claims by showing that \eqref{eq:muDoublePrimeIsAlsoJansonWitness} holds.
  The equality $e(\mu'') = \sqrt{R}$ follows by linearity of $e(\cdot)$ and the fact that both $\mu$ and $\mu'$ have edge measure equal to $\sqrt{R}$.
  Expand $\Lambda_p(\mu'')$ as
  \begin{equation}\label{eq:expansionOfLambdaOfNuDoublePrime}
    \Lambda_p(\mu'') = (1 - \tau)^2 \Lambda_p(\mu) + 2\tau (1 - \tau) \sum_{\substack{L \subset V(\cal{G}) \\ |L| \ge 2}} d_{\mu}(L) d_{\mu'}(L) p^{-|L|} + \tau^2 \Lambda_p(\mu').
  \end{equation}
  Applying the Cauchy--Schwarz inequality to the second term in \eqref{eq:expansionOfLambdaOfNuDoublePrime}, we obtain
  \begin{equation}\label{eq:cauchySchwarzInSecondTerm}
    \sum_{\substack{L \subset V(\cal{G}) \\ |L| \ge 2}} \frac{d_{\mu}(L) d_{\mu'}(L)}{p^{|L|}} \le \bigg(\sum_{\substack{L \subset V(\cal{G}) \\ |L| \ge 2}} \frac{d_{\mu}(L)^2}{p^{|L|}}\bigg)^{1/2} \bigg(\sum_{\substack{L \subset V(\cal{G}) \\ |L| \ge 2}} \frac{d_{\mu'}(L)^2}{p^{|L|}}\bigg)^{1/2} = \sqrt{\Lambda_p(\mu) \Lambda_p(\mu')}.
  \end{equation}
  Replacing \eqref{eq:cauchySchwarzInSecondTerm} back in \eqref{eq:expansionOfLambdaOfNuDoublePrime} and simplifying yields
  \begin{equation*}
    \Lambda_p(\mu'') \le (1 - \tau)^2 \Lambda_p(\mu) + 2 \tau (1 - \tau) \sqrt{\Lambda_p(\mu) \Lambda_p(\mu')} + \tau^2 \Lambda_p(\mu') = \big( (1-\tau) \Lambda_p(\mu)^{1/2} + \tau \Lambda_p(\mu')^{1/2} \big)^2,
  \end{equation*}
  which, by \eqref{eq:muIsJansonWitness} and \eqref{eq:muPrimeIsAlsoJansonWitness}, establishes \eqref{eq:muDoublePrimeIsAlsoJansonWitness}:
  \begin{equation*}
    \Lambda_p(\mu'') < \frac{1}{R}\big((1 - \tau) e(\mu) + \tau e(\mu')\big)^2 = \frac{e(\mu'')^2}{R}.
  \end{equation*}

  Towards establishing \eqref{eq:contradictionOfDegreeMinimality}, we expand
  \begin{equation}\label{eq:expansionOfDegreeSquaredSum}
    \sum_{v \in V(\cal{G})} d_{\mu''}(v)^2 = (1 - \tau)^2 \sum_{v \in V(\cal{G})} d_{\mu}(v)^2 + 2 \tau(1 - \tau) \sum_{v \in V(\cal{G})} d_{\mu}(v) d_{\mu'}(v) +\tau^2 \sum_{v \in V(\cal{G})} d_{\mu'}(v)^2,
  \end{equation}
  an expression whose terms we will bound separately.
  Recall that $d_{\mu'}(v) = 0$ for $v \not \in W$, so
  \begin{equation*}
    \sum_{v \in V(\cal{G})} d_{\mu}(v)d_{\mu'}(v) = \sum_{v \in W} d_{\mu}(v) d_{\mu'}(v)
  \end{equation*}
  but now, \eqref{eq:choiceOfW} implies that
  \begin{equation}\label{eq:preBoundOnTheSecondTerm}
    \sum_{v \in W} d_\mu(v) d_{\mu'}(v) \le \frac{s \, e(\mu)}{\beta v(\cal{G})} \sum_{v \in W} d_{\mu'}(v) = \frac{s^2 e(\mu)^2}{\beta v(\cal{G})}
  \end{equation}
  since $\cal{G}$ is $s$-uniform and $e(\mu') = e(\mu)$.
  Using \eqref{eq:contradictionAssumptionInDegControl} in \eqref{eq:preBoundOnTheSecondTerm}, we obtain, for the second term,
  \begin{equation}\label{eq:boundOnTheSecondTerm}
    \sum_{v \in W} d_{\mu}(v) d_{\mu'}(v) < \frac{1}{2} \sum_{v \in V(\cal{G})} d_\mu(v)^2.
  \end{equation}
  The $s$-uniformity of $\cal{G}$ and $e(\mu) = e(\mu')$ also imply an easy bound on the third term in \eqref{eq:expansionOfDegreeSquaredSum}:
  \begin{equation}\label{eq:boundOnThirdTerm}
    \sum_{v \in V(\cal{G})} d_{\mu'}(v)^2 \le \bigg(\sum_{v \in V(\cal{G})} d_{\mu'}(v) \bigg)^2 = \big(e(\mu) s\big)^2 \le \dfrac{\beta v(\cal{G})}{2} \sum_{v \in V(\cal{G})} d_\mu(v)^2 \le  \beta v(\cal{G}) \sum_{v \in V(\cal{G})} d_\mu(v)^2,
  \end{equation}
  where the first inequality holds because $d_\mu'$ is always non-negative, and the second is \eqref{eq:contradictionAssumptionInDegControl}.

  We can now replace \eqref{eq:boundOnTheSecondTerm} and \eqref{eq:boundOnThirdTerm} in \eqref{eq:expansionOfDegreeSquaredSum} and use that $\tau$ is positive to obtain, after simplification, that
  \begin{equation*}
    \sum_{v \in V(\cal{G})} d_{\mu''}(v)^2 < (1 - \tau + \beta v(\cal{G}) \tau^2) \sum_{v \in V(\cal{G})} d_\mu(v)^2.
  \end{equation*}
  Choosing $\tau$ to satisfy
  \begin{equation*}
    0 < \tau < \min\Big\{1, \frac{1}{\beta v(\cal{G})}\Big\}
  \end{equation*}
  results in
  \begin{equation*}
    \sum_{v \in V(\cal{G})} d_{\mu''}(v)^2 < \sum_{v \in V(\cal{G})} d_\mu(v)^2
  \end{equation*}
  which contradicts the fact that $\mu$ minimises $\sum_{v \in V(\cal{G})} d_\mu(v)^2$ and completes the proof.
\end{proof}

\end{document}